%% file: preprint.tex
\documentclass[english]{article}

\usepackage[utf8]{inputenc}
\usepackage[T1]{fontenc}      
\usepackage[english]{babel}
\usepackage{amsmath, amsthm}
\usepackage{amsfonts,amssymb}
\usepackage{graphicx}
\usepackage{lmodern}
\usepackage[a4paper]{geometry}
\usepackage{geometry}
\usepackage{dsfont}
\usepackage{algorithm}
\usepackage{algpseudocode}
\usepackage{subcaption}
\usepackage{comment}
\usepackage{enumitem}
\usepackage[toc,page]{appendix}
\usepackage{color}
\usepackage{authblk} 
\usepackage{abstract}

\newtheorem{theorem}{Theorem}[section]
\newtheorem{lemme}{Lemma}[section]
\newtheorem{proposition}{Proposition}[section]
\newtheorem{Rem}{Remark}[section]
\newtheorem{hyp}{Assumption}[section]
\newtheorem{cor}{Corollary}[section]
\newtheorem{defi}{Definition}[section]

\newcommand*\samethanks[1][\value{footnote}]{\footnotemark[#1]}
\newcommand{\ind}[1]{\mathds 1_{#1}}
\newcommand{\Gfrac}[2]{\genfrac{}{}{0pt}{2}{#1}{#2}}
\newcommand{\E}{\mathbb E }

\newcommand{\bb}{\mathbb}
\newcommand{\smax}{\vee}
\newcommand{\smin}{\wedge}
\newcommand{\cali}{\mathcal}

\newcommand{\ut}{\partial_t u}
\newcommand{\ux}{\partial_\nu u}

\newcommand{\utt}{\partial^2_{tt} u}
\newcommand{\uxx}{\partial^2_{\nu\nu} u}
\newcommand{\utx}{\partial^2_{t\nu} u}

\newcommand{\gt}{\partial_t g}
\newcommand{\gnu}{\partial_\nu g}
\newcommand{\gx}{\partial_\nu g}

\begin{document}

\title{\textbf{Thinning and Multilevel Monte Carlo for 
Piecewise Deterministic (Markov) Processes.\\
Application to a 
 stochastic Morris-Lecar model.}}
\author{\textbf{Vincent Lemaire}\thanks{Laboratoire de Probabilités, Statistique et Mod\'elisation (LPSM), UMR CNRS 8001, Sorbonne Universit\'e-Campus Pierre et Marie Curie, Case 158, 4 place Jussieu, F-75252 Paris Cedex 5, France} \footnote{vincent.lemaire@upmc.fr}\hspace{1cm}\textbf{Michèle Thieullen}\samethanks[1] \footnote{michele.thieullen@upmc.fr}  \hspace{1cm}\textbf{Nicolas Thomas}\samethanks[1] \footnote{nicolas.thomas@upmc.fr}}

\date{}

\maketitle

\begin{abstract}
\noindent
In the first part of this paper we study approximations of trajectories of Piecewise Deterministic Processes (PDP) when the flow is not explicit by the thinning method. We also establish a strong error estimate for PDPs as well as a weak error expansion for Piecewise Deterministic Markov Processes (PDMP). These estimates are the building blocks of the Multilevel Monte Carlo (MLMC) method which we study in the second part. The coupling required by MLMC is based on the thinning procedure. In the third part we apply these results to a 2-dimensional Morris-Lecar model with stochastic ion channels. In the range of our simulations the MLMC estimator outperforms the classical Monte Carlo one.
\end{abstract}



\section{Introduction}

In this paper we are interested in the approximation of the trajectories of PDPs. We establish strong error estimates for a PDP and a weak error expansion for a PDMP. Then we study  the application of the Multilevel Monte Carlo (MLMC) method in order to approximate expectations of functional of PDMPs. Our motivation comes from Neuroscience where the whole class of stochastic conductance-based neuron models can be interpreted as PDMPs. The response of a neuron to a stimulus, called neural coding, is considered as a relevant information to understand the functional properties of such excitable cells. Thus many quantities of interest such as mean first spike latency, mean interspike intervals and mean firing rate can be modelled as expectations of functionals of PDMPs.

\bigskip

\noindent PDPs have been introduced by Davis in \cite{davarticle} as a general class of stochastic processes characterized by a deterministic evolution between two successive random times. In the case where the deterministic evolution part follows a family of Ordinary Differential Equations (ODEs) the corresponding PDP enjoys the Markov property and is called a PDMP. The distribution of a PDMP is thus determined by three parameters called  the characteristics of the PDMP: a family of vector fields, a jump rate (intensity function) and a transition measure. 

\noindent We consider first a general PDP $(x_t)$ which is not necessarily Markov on a finite time interval $[0,T]$ for which the flow is not explicitly solvable. Approximating its flows by the classical Euler scheme and using our previous work \cite{thinPDMP}, we build a thinning algorithm which provides us with an exact simulation of an approximation of $(x_t)$ that we denote $(\overline x_t)$. The process $(\overline x_t)$ is a PDP constructed by thinning of a homogeneous Poisson process which enjoys explicitly solvable flows.  
 
\noindent Actually this thinning construction provides a whole family of approximations indexed by the time step $h>0$ of the Euler scheme. We prove that for any real valued smooth function $F$ the following strong estimate holds

\begin{equation}\label{hyp_strong}
\exists\, \,  V_1>0, V_2>0,
\hspace{0.5cm}  
\E[ |F(\overline x_T) - F(x_T)|^2 ] \leq V_1 h + V_2 h^2.
\end{equation}
Moreover if $(x_t)$ is a PDMP the following weak error expansion holds
\begin{equation}\label{hyp_weak}
\exists\, \,  c_1>0,
\hspace{0.5cm}  
\E[ F(\overline x_T)]  - \E[ F(x_T) ] = c_1 h +o(h^{2}).
\end{equation}

\noindent The estimate \eqref{hyp_strong} is mainly based on the construction of the couple $(x_t, \overline x_t)$ and on the fact that the Euler scheme is of order 1 this is why it is valid for a general PDP and its Euler scheme. On the contrary, the estimate \eqref{hyp_weak} relies on properties which are specific to PDMPs such as the Feynman-Kac formula. 

\bigskip \noindent The MLMC method relies simultaneously on estimates \eqref{hyp_strong} and \eqref{hyp_weak} that is why we study its application to the PDMP framework instead of the more general PDP one.
MLMC extends the classical Monte Carlo (MC) method which is a very general approach to estimate expectations using stochastic simulations. The complexity (i.e the number of operations necessary in the simulation) associated to a MC estimation can be prohibitive especially when the complexity of an individual random sample is very high. 
MLMC relies on repeated independent random samplings taken on different levels of accuracy which differs from the classical MC method.
MLMC can then greatly reduces the complexity of the classical MC by performing most simulations with low accuracy  but with low complexity and only few simulations with high accuracy at high complexity. MLMC have been introduced by S. Heinrich in \cite{Hei01} and developed by M. Giles in \cite{Gi08}.
The MLMC estimator has been efficiently used in various fields of numerical probability such as SDEs \cite{Gi08}, Markov chains \cite{AnHi12}, \cite{AnHiSu14}, \cite{GlRhot14}, Lévy processes \cite{kyp2014}, jump diffusions \cite{Xia}, \cite{Der11}, \cite{DH11} or nested Monte Carlo \cite{LP}, \cite{Giorgi}. See \cite{giles_2015} for more references. To the best of our knowledge, application of MLMC to PDMPs has not been considered. 

\bigskip
\noindent For the sake of clarity, we describe here the general improvement of MLMC. We are interested in the estimation of $\E[X]$ where $X$ is a real valued square integrable random variable on a probability space $\left( \Omega,\cali F, \bb P \right)$. 
When $X$ can be simulated exactly the classical MC estimator $(1/N)\sum_{k=1}^N X^k$ with $X^k, k \geq 1$ independent random variables identically distributed as $X$, provides an unbiased estimator. The associated L$^2$ - error satisfies $\parallel Y-\E[X] \parallel_2^2 = \text{Var}(Y)=\frac{1}{N} \text{Var}(X)$. If we quantify the precision by the L$^2$ - error, then a user-prescribed precision $\epsilon^2 > 0$ is achieved for $N =O(\epsilon^{-2})$ so that in this case the global complexity is of order $O(\epsilon^{-2})$.

\noindent Assume now that $X$ cannot be simulated exactly (or cannot be simulated at a reasonable cost) and that we can build a family of real valued random variables $(X_h, h > 0)$ on $\left( \Omega,\cali F, \bb P \right)$ which converges weakly and strongly to $X$ as $h \rightarrow 0$ in the following sense

\begin{equation}\label{hyp_weak_general}
\exists\, \,  c_1>0, \alpha >0,
\hspace{0.5cm}  
\E[ X_h]  - \E[ X ] = c_1 h^\alpha +o(h^{2\alpha}),
\end{equation}
and
\begin{equation}\label{hyp_strong_general}
\exists\, \,  V_1>0, \beta >0,
\hspace{0.5cm}  
\E[ |X_h - X|^2 ] \leq V_1 h^\beta.
\end{equation}

\noindent Assume moreover that for $h > 0$ the random variable $X_h$ can be simulated at a reasonable complexity (the complexity increases as $h \rightarrow 0$). The classical MC estimator now consists in a sequence of random variables
\begin{equation}\label{generic_MC_estimator}
Y = \frac{1}{N} \sum_{k=1}^N X^k_h,
\end{equation}

\noindent where $X^k_h, k\geq 1$ are independent random variables identically distributed as $X_h$. The bias and the variance of the estimator \eqref{generic_MC_estimator} are respectively given by $\E[Y] - \E[X]=\E[X_h] - \E[X]\simeq c_1 h^\alpha$ and $\text{Var}(Y)=\frac{1}{N} \text{Var}(X_h)$. From the strong estimate \eqref{hyp_strong_general} we have that $\text{Var}(X_h) \rightarrow \text{Var}(X)$ as $h \rightarrow 0$ so that $\text{Var}(X_h)$ is asymptotically a constant independent of $h$. If as above we quantify the precision by the L$^2$ - error and use that $\parallel Y-\E[X] \parallel_2^2 = (\E[Y] - \E[X])^2 + \text{Var}(Y)$, we obtain that the estimator \eqref{generic_MC_estimator} achieves a user-prescribed precision $\epsilon^2 > 0$ for $h = O(\epsilon^{1/\alpha})$ and $N =O(\epsilon^{-2})$ so that the global complexity of the estimator is now $O(\epsilon^{-2- \frac{1}{\alpha}})$. 

\noindent 
\noindent The MLMC method takes advantage of the estimate \eqref{hyp_strong_general} in order to reduce the global complexity. Let us fix $L \geq 2$ and consider for $l \in \{1,\ldots,L\}$ a geometrically decreasing sequence $(h_l,1 \leq l \leq L)$ where $h_l = h^* M^{-(l-1)}$ for fixed $h^* >0$ and $M>1$. The indexes $l$ are called the levels of the MLMC and the complexity of $X_{h_l}$ increases as the level increases. Thanks to the weak expansion \eqref{hyp_weak_general}, the quantity $\E [X_{h_L}]$ approximates $\E [X]$.
Using the linearity of the expectation the quantity $\E [X_{h_L}]$ can be decomposed over the levels $l \in \{1,\ldots,L \}$ as follows 
\begin{equation}\label{expectation_over_levels}
\E [X_{h_L}] = \E [X_{h^*}] + \sum_{l=2}^L \E [X_{h_l} - X_{h_{l-1}}].
\end{equation} 

\noindent For each level $l \in \{1,\ldots,L\}$, a classical MC estimator is used to approximate $\E [X_{h_l} - X_{h_{l-1}}]$ and $\E [X_{h^*}]$. At each level, a number $N_l \geq 1$ of samples are required and the key point is that the random variables $X_{h_l}$ and $X_{h_{l-1}}$ are assumed to be correlated in order to make the variance of $X_{h_l} - X_{h_{l-1}}$ small. Considering at each level $l=2,\ldots,L$ independent couples $(X_{h_l},X_{h_{l-1}})$ of correlated random variables, the MLMC estimator then reads
\begin{equation}\label{generic_MLMC_estimator}
Y = \frac{1}{N_1} \sum_{k=1}^{N_1} X_{h^*}^{k} + \sum_{l=2}^L \frac{1}{N_l} \sum_{k=1}^{N_l} (X_{h_l}^{k} - X_{h_{l-1}}^{k}),
\end{equation}

\noindent where $(X_{h^*}^{k}, k \geq 1)$ is a sequence of independent and identically distributed random variables distributed as $X_{h^*}$ and $\left( (X^{k}_{h_l},X^{k}_{h_{l-1}}), k \geq 1\right)$ for $l=2,\ldots,L$ are independent sequences of independent copies of $(X_{h_l},X_{h_{l-1}})$ and independent of $(X_{h^*}^{k})$. It is known, see \cite{Gi08} or \cite{LP}, that given a precision $\epsilon > 0$ and provided that the family $(X_h,h > 0)$ satisfies the strong and weak error estimates \eqref{hyp_strong_general} and \eqref{hyp_weak_general}, the multilevel estimator \eqref{generic_MLMC_estimator} achieves a precision $\parallel Y-\E[X] \parallel_2^2 =\epsilon^2$ with a global complexity of order $O(\epsilon^{-2})$ if $\beta > 1$, $O(\epsilon^{-2}(\log(\epsilon))^2)$ if $\beta = 1$ and $O(\epsilon^{-2 - (1-\beta)/\alpha})$ if $\beta < 1$. This complexity result shows the importance of the parameter $\beta$. Finally, let us mention that in the case $\beta > 1$ it possible to build an unbiased multilevel estimator, see \cite{RhGl15}.

\bigskip \noindent Estimates \eqref{hyp_strong} and \eqref{hyp_weak} suggest to investigate the use of the MLMC method in the PDMP framework with $\beta=1$ and $\alpha =1$. Letting $X = F(x_T)$ and $X_h = F(\overline x_T)$ for $h > 0$ and $F$ a smooth function, we define a MLMC estimator of $\E [F(x_T)]$ just as in \eqref{generic_MLMC_estimator} (noted $Y^{\text{MLMC}}$ in the paper) where the processes involved at the level $l$ are correlated by thinning. Since these processes are constructed using two different time steps, the probability of accepting a proposed jump time differs from one process to the other. Moreover the discrete components of the post-jump locations may also be different. This results in the presence of the term $V_1 h$ in the estimate \eqref{hyp_strong}.
In order to improve the convergence rate (to increase the parameter $\beta$) in \eqref{hyp_strong}, we show that for a given PDMP $(x_t)$ we have the following auxiliary representation
\begin{equation}\label{aux_representation}
\E [F(x_T)] = \E [F(\tilde x_T) \tilde R_T].
\end{equation}

\noindent The PDMP $(\tilde x_t)$ and its Euler scheme are such that their discrete components jump at the same times and in the same state. 
$(\tilde R_t)$ is a process which depends on $(\tilde x_t, t \in [0,T])$. The representation \eqref{aux_representation} is inspired by the change of probability introduced in \cite{Xia} and is actually valid for a general PDP (Proposition \ref{DECOMP2}) so that $\E [F(\overline x_T)] = \E [F(\underline{\tilde x}_T) \underline{\tilde R}_T]$ where $(\underline{\tilde x}_t)$ is the Euler scheme corresponding to $(\tilde x_t)$ and $(\underline{\tilde R}_t)$ is a process which depends on $(\underline{\tilde x}_t, t \in [0,T])$. Letting $X=F(\tilde x_T) \tilde R_T$ and $X_h=F(\underline{\tilde x}_T) \underline{\tilde R}_T$ we define a second MLMC estimator (noted $\tilde Y^{\text{MLMC}}$) where now the discrete components of the Euler schemes $(\underline{\tilde x}_t)$ involved at the level $l$ always jump in the same states and at the same times. To sum up, the first MLMC estimator we consider ($ Y^{\text{MLMC}}$) derives from \eqref{expectation_over_levels} where the corrective term at level $l$ is $\E[F(\overline x_T^{h_l}) - F(\overline x_T^{h_{l-1}})]$ whereas the corrective term of the second estimator ($\tilde Y^{\text{MLMC}}$) is $\E[F(\underline{\tilde x}_T^{h_l})\underline{\tilde R}_T^{h_l} - F(\underline{\tilde x}_T^{h_{l-1}})\underline{\tilde R}_T^{h_{l-1}}]$. For readability, we no longer write the dependence of the approximations on the time step.
For the processes $(F(\underline{\tilde x}_t) \underline{\tilde R}_t)$ and $(F(\tilde x_t) \tilde R_t$) we show the following strong estimate  
\begin{equation*}
\exists\, \,  \tilde V_1>0, 
\hspace{0.5cm}  
\E[ |F(\underline{\tilde x}_T) \underline{\tilde R}_T - F(\tilde x_T) \tilde R_T|^2 ] \leq \tilde V_1 h^2,
\end{equation*}

\noindent so that we end up with $\beta=2$ and the complexity goes from a $O(\epsilon^{-2}(\log(\epsilon))^2)$ to a $O(\epsilon^{-2})$.

\bigskip \noindent As an application we consider the PDMP version of the 2-dimensional Morris-Lecar model, see \cite{wainfluid}, which takes into account the precise description of the ionic channels and in which the flows are not explicit. Let us mention \cite{benaim2012} for the application of quantitative bounds for the long time behavior of PDMPs to a stochastic 3-dimensional Morris-Lecar model. The original deterministic Morris-Lecar model has been introduced in \cite{lecar} to account for various oscillating states in the barnacle giant muscle fiber. Because of its low dimension, this model is among the favourite conductance-based models in computational Neuroscience. Furthermore, this model is particularly interesting because it reproduces some of the main features of excitable cells response such as the shape, amplitude and threshold of the action potential, the refractory period. 
We compare the classical MC and the MLMC estimators on the 2-dimensional stochastic Morris-Lecar model to estimate the mean value of the membrane potential at fixed time. It turns out that in the range of our simulations the MLMC estimator outperforms the MC one. It suggests that MLMC estimators can be used successfully in the framework of PDMPs. 

\noindent  As mentioned above, the quantities of interest such as mean first spike latency, mean interspike intervals and mean firing rate can be modelled as expectations of path-dependent functional of PDMPs. This setting can then be considered as a natural extension of this work. 

\bigskip \noindent The paper is organised as follows. In section 2, we construct a general PDP by thinning and we give a representation of its distribution in term of the thinning data (Proposition 1). In section 3, we establish strong error estimates (Theorems 1-2). In section 4, we establish a weak error expansion (Theorem 3). In section 5, we compare the efficiency of the classical and the multilevel Monte Carlo estimators on the 2-dimensional stochastic Morris-Lecar model.

\section{Piecewise Deterministic Process by thinning}\label{section_PDP}
\subsection{Construction}\label{construction}

In this section we introduce the setting and recall some results on the thinning method from our previous paper \cite{thinPDMP}.
Let $E := \Theta \times \bb R^d$ where $\Theta$ is a finite or countable set and $d \geq 1$. A piecewise deterministic process (PDP) is defined from the following characteristics 
\begin{itemize}
\item a family of functions $\left( \Phi_\theta \right)_{\theta \in \Theta}$ such that $\Phi_\theta : \bb R_+ \times \bb R^d \rightarrow \bb R^d$ for all $\theta \in \Theta$,

\item a measurable function $\lambda: E \rightarrow ]0,+\infty[$,

\item a transition measure $Q : E \times \cali B(E) \rightarrow [0,1]$.

\end{itemize}

\noindent We denote by $x = (\theta,\nu) $ a generic element of $E$. We only consider PDPs with continuous $\nu$-component so that for  $A \in \cali B(\Theta)$ and $B \in \cali B(\bb R^d)$, we write
\begin{equation}\label{Q_extension}
Q(x, A \times B) = Q(x, A) \delta_\nu(B).
\end{equation}

\noindent If we write $x=(\theta_x,\nu_x)$, then it holds that 
\begin{equation*} 
Q((\theta_x,\Phi_{\theta_x}(t,\nu_x)),d\theta d\nu)=Q((\theta_x,\Phi_{\theta_x}(t,\nu_x)),d\theta)\delta_{\Phi_{\theta_x}(t,\nu_x)}(d\nu).
\end{equation*}

\noindent Our results do not depend on the dimension of the variable in $\bb R^d$ so we restrict ourself to $\bb R$ ($d=1$) for the readability. We work under the following assumption

\begin{hyp}\label{hyp_lambda1}
There exists $\lambda^* < +\infty$ such that, for all $x \in E$, $ \lambda(x) \leq \lambda^*$.
\end{hyp}

\noindent In \cite{thinPDMP} we considered a general upper bound $\lambda^*$. In the present paper $\lambda^*$ is constant (see Assumption \ref{hyp_lambda1}). Let $\left( \Omega,\cali F,\bb P \right)$ be a probability space on which we define 

\begin{enumerate}
\item an homogeneous Poisson process $(N^*_t,t \geq 0)$ with intensity $\lambda^*$ (given in Assumption \ref{hyp_lambda1}) whose successive jump times are denoted $\left(T^*_k, k \geq 1 \right)$. We set $T^*_0 =0$.

\item two sequences of iid random variables with uniform distribution on $[0,1]$, $(U_k, k \geq 1)$ and $(V_k, k \geq 1)$ independent of each other and independent of $\left( T^*_k, k \geq 1 \right)$.
\end{enumerate}

\noindent  Given $T>0$ we construct iteratively the sequence of jump times and post-jump locations $(T_n,(\theta_n,\nu_n), n \geq 0)$ of the $E$-valued PDP $(x_t, t \in [0,T])$
 that we want to obtain in the end using its characteristics $\left( \Phi, \lambda, Q \right)$. 
Let  $(\theta_0,\nu_0) \in E$ be fixed and let $T_0 = 0$. We construct $T_1$ by thinning of $(T^*_k)$, that is
\begin{equation}\label{T1}
T_1 := T^*_{\tau_1},
\end{equation} 

\noindent where
\begin{equation}\label{tau1}
\tau_1 := \inf \left\{ k > 0: U_k \lambda^* \leq \lambda( \theta_0, \Phi_{\theta_0}( T^*_k,\nu_0))\right\}.
\end{equation}

\noindent We denote by $|\Theta|$ the cardinal of $\Theta$ (which may be infinite) and we set $\Theta = \{k_1,\ldots,k_{|\Theta|}\}$. For $j \in \{ 1,\ldots,|\Theta| \}$ we introduce the functions $a_j$ defined on $E$ by
\begin{equation}\label{aj}
a_j(x) := \sum_{i=1}^j Q(x,\{k_i\}),
\hspace{0.5cm}
\forall x \in E.
\end{equation} 

\noindent By convention, we set $a_0 := 0$.  We also introduce the function $H$ defined by
\begin{equation*}
H(x,u) := \sum_{i=1}^{|\Theta|} k_i \ind{ a_{i-1}(x) < u \leq a_i(x)},
\hspace{0.5cm}
\forall x \in E, \forall u \in [0,1].
\end{equation*}

\noindent For all $x \in E$, $H(x,.)$ is the inverse of the cumulative distribution function of $Q(x,.)$ (see for example \cite{dev}). Then, we construct $(\theta_1,\nu_1)$ from the uniform random variable $V_1$ and the function $H$ as follows
\begin{eqnarray}
(\theta_1,\nu_1)&=& 
\left( H\left( (\theta_{0},\Phi_{\theta_0}( T^*_{\tau_1},\nu_0)), V_1 \right), \phi_{\theta_0}( T^*_{\tau_1},\nu_0) \right),\nonumber\\
&=&\left( H\left( (\theta_{0},\Phi_{\theta_0}( T_{1},\nu_0)), V_1 \right), \phi_{\theta_0}( T_1,\nu_0) \right)\nonumber.
\end{eqnarray}

\noindent Thus, the distribution of $(\theta_1,\nu_1)$ given $\left( \tau_1, ( T^*_k)_{k \leq \tau_1} \right)$ is $Q((\theta_{0},\Phi_{\theta_0}( T^*_{\tau_1},\nu_0)),.)$ or in view of (\ref{Q_extension}),
\begin{equation*}
\sum_{k \in \Theta} Q \left( (\theta_{0},\Phi_{\theta_0}( T^*_{\tau_1},\nu_0)), \{k\} \right) \delta_{\left( k,\phi_{\theta_0}( T^*_{\tau_1},\nu_0) \right)}.
\end{equation*}

\noindent For $n>1$, assume that $\left(\tau_{n-1}, (T^*_k)_{k \leq \tau_{n-1}}, (\theta_{n-1},\nu_{n-1}) \right)$ is constructed. Then, we construct $T_n$ by thinning of $(T^*_k)$ conditionally to $\left(\tau_{n-1}, ( T^*_k)_{k \leq \tau_{n-1}}, (\theta_{n-1},\nu_{n-1}) \right)$ , that is
\begin{equation*}
T_n :=  T^*_{\tau_n},
\end{equation*} 

\noindent where
\begin{equation*}
\tau_n := \inf \left\{ k > \tau_{n-1}: U_k \lambda^* \leq \lambda(\theta_{n-1}, \Phi_{\theta_{n-1}}(T^*_k-T^*_{\tau_{n-1}},\nu_{n-1}) ) \right\}.
\end{equation*}

\noindent Then, we construct $(\theta_n,\nu_n)$ using the uniform random variable $V_n$ and the function $H$ as follows
\begin{eqnarray*}
(\theta_n,\nu_n)&:=& 
\left( H\left( (\theta_{n-1}, \Phi_{\theta_{n-1}}( T^*_{\tau_n} -  T^*_{\tau_{n-1}},\nu_{n-1})), V_n \right), \Phi_{\theta_{n-1}}( T^*_{\tau_n} -  T^*_{\tau_{n-1}},\nu_{n-1}) \right)\\
&=&\left( H\left( (\theta_{n-1}, \Phi_{\theta_{n-1}}( T_{n} -  T_{n-1},\nu_{n-1})), V_n \right), \Phi_{\theta_{n-1}}( T_n -  T_{n-1},\nu_{n-1}) \right)\nonumber.
\end{eqnarray*}

\noindent We define the PDP $x_t$ for all $t \in [0,T]$ from the process $(T_n,(\theta_n,\nu_n))$ by
\begin{equation}\label{PDMP}
x_t 
:= 
\left( \theta_{n}, \Phi_{\theta_{n}}(t-T_{n},\nu_{n}) \right),
\hspace{0.5cm}
t \in [T_n,T_{n+1}[.
\end{equation}

\noindent Thus, $x_{T_n}=(\theta_n,\nu_n)$ and $x^-_{T_n} = (\theta_{n-1},\nu_n)$. We also define the counting process associated to the jump times $N_t := \sum_{n \geq 1} \ind{T_n \leq t}$.

\subsection{Approximation of a PDP}
In applications we may not know explicitly the functions $\Phi_\theta$. In this case, we use a numerical scheme $\overline \Phi_\theta$ approximating $\Phi_\theta$. In this paper, we consider schemes such that there exits positive constants $C_1$ and $C_2$ independent of $h$ and $\theta$ such that 
\begin{equation}\label{error_flow}
\sup_{t\in{[0,T]}} | \Phi_\theta(t,\nu_1)-\overline \Phi_\theta(t,\nu_2) | \leq e^{C_1 T} |\nu_1-\nu_2| + C_2 h , 
\hspace{1cm}
\forall \theta \in \Theta, \forall (\nu_1, \nu_2) \in \bb R^2.
\end{equation}

\noindent To the family $(\overline \Phi_\theta)$ we can associate a PDP constructed as above that we denote $(\overline x_t)$. We emphasize that there is a positive probability that $(x_t)$ and $(\overline x_t)$ jump at different times and/or in different states even if they are both constructed from the same data $(N^*_t)$, $(U_k)$ and $(V_k)$. However if the characteristics $(\Phi,\tilde \lambda, \tilde Q)$ of a PDP $(\tilde x_t)$  are such that $\tilde \lambda$ and $\tilde Q$ depend only on $\theta$, that is $\tilde \lambda(x) = \tilde \lambda(\theta)$ and $\tilde Q(x,.) = \tilde Q(\theta,.)$ for all $x=(\theta,\nu) \in E$, then its embedded Markov chain $(\tilde T_n, (\tilde \theta_n,\tilde \nu_n), n\geq 0)$ is such that $(\tilde \theta_n, n\geq 0)$ is an autonomous Markov chain with kernel $\tilde Q$ and $(\tilde T_n, n\geq 0)$ is a counting process with intensity $\tilde \lambda_t = \sum_{n \geq 0} \tilde \lambda(\tilde \theta_n) \ind{\tilde T_n \leq t < \tilde T_{n+1}}$. In particular, both $(\tilde \theta_n)$ and $(\tilde \tau_n)$ do not depend on $\Phi$. The particular form of the characteristics $\tilde \lambda$ and $\tilde Q$ implies that the PDP $(\tilde x_t)$ and its approximation $(\underline{\tilde x_t})$ are correlated via the same process $(\tilde \tau_n,\tilde \theta_n)$. In other words, these processes always jump exactly at the same times and their $\theta$-component always jump in the same states. Such processes $(\tilde x_t)$ are easier theoretically as well as numerically than the general case. They will be useful for us in the sequel.

\bigskip \noindent The following lemma (which is important for several proofs below) gives a direct consequence of the estimate \eqref{error_flow}.
\begin{lemme}\label{strong_error_lemme2}
Let $(\Phi_\theta)$ and $(\overline \Phi_\theta)$ satisfying \eqref{error_flow}. Let $(t_n, n \geq 0)$ be an increasing sequence of non-negative real numbers with $t_0=0$ and let $(\alpha_n, n \geq 0)$ be a sequence of $\Theta$-valued components. For a given $\nu \in \bb R$ let us define iteratively the sequences $(\beta_n, n \geq 0)$ and $(\overline \beta_n, n \geq 0)$ as follows
\begin{equation*}
\left\{
\begin{array}{ll}
\beta_n = \Phi_{\alpha_{n-1}}(t_n-t_{n-1},\beta_{n-1}),\\
\beta_0=\nu,
\end{array}
\right.
\hspace{0.5cm}
\text{and}
\hspace{0.5cm}
\left\{
\begin{array}{ll}
\overline \beta_n = \overline \Phi_{\alpha_{n-1}}(t_n-t_{n-1},\overline \beta_{n-1}),\\
\overline \beta_0=\nu.
\end{array}
\right.
\end{equation*}

\noindent Then, for all $n \geq 1$ we have
\begin{equation*}
|\overline \beta_n - \beta_n|
\leq
e^{C_1 t_n} n C_2 h, 
\end{equation*}

\noindent where $C_1$ and $C_2$ are positive constants independent of $h$.
\end{lemme}

\begin{proof}[Proof of Lemma \ref{strong_error_lemme2}]
Let $n \geq 1$. 
From the estimate \eqref{error_flow}, we have for all $k \leq n$
\begin{equation*}
\begin{split}
\left| \overline \beta_k -  \beta_k \right|
&\leq
e^{C_1(t_k - t_{k-1})}|\overline \beta_{k-1} - \beta_{k-1}| + C_2 h,
\end{split}
\end{equation*}

\noindent and therefore
\begin{equation*}
\begin{split}
e^{- C_1 t_k} \left| \overline \beta_k -  \beta_k \right|
&\leq
e^{-C_1 t_{k-1}}|\overline \beta_{k-1} - \beta_{k-1}| + C_2 h.
\end{split}
\end{equation*}

\noindent By summing up these inequalities for $1 \leq k \leq n$ and since $\beta_0 = \overline \beta_0$ we obtain
\begin{equation*}
\left| \overline \beta_n -  \beta_n \right|
\leq 
e^{C_1 t_n} n C_2 h.
\end{equation*}
\end{proof}

\subsection{Application to the construction of a PDMP and its associated Euler scheme}\label{section_application}

In this section we define a PDMP and its associated Euler scheme from the construction of the section \ref{construction}. For all $\theta \in \Theta$, we consider a family of vector fields $(f_\theta, \theta \in \Theta)$ satisfying

\begin{hyp}\label{hyp_f_lip}
For all $\theta \in \Theta$, the function $f_\theta : \bb R \rightarrow \bb R$ is bounded and Lipschitz with constant $L$ independent of $\theta$.
\end{hyp}

\noindent If we choose $\Phi_\theta = \phi_\theta$ in the above construction where for all $x = (\theta,\nu) \in E$, we denote by $(\phi_\theta(t,\nu), t \geq 0)$ the unique solution of the ordinary differential equation (ODE)
\begin{equation}\label{ODE}
\left\{
\begin{array}{l}
\frac{d y(t)}{dt} = f_\theta \left( y(t) \right), \\
y(0) = \nu,
\end{array}
\right.
\end{equation}

\noindent
then the corresponding PDP is Markov since $\phi$ satisfies the semi-group property which reads $\phi_\theta(t+s,\nu) = \phi_\theta(t,\phi_\theta(s,\nu))$ for all $t,s \geq 0$ and for all $(\theta,\nu) \in E$. In this case, the process $(x_t)$ is a piecewise deterministic Markov process (see \cite{dav} or \cite{jacobsen}).

\bigskip \noindent Let $h > 0$. We approximate the solution of \eqref{ODE} by the Euler scheme with time step $h$. First, we define the Euler subdivision of $[0,+\infty[$ with time step $h$, noted $(\overline t_i, i \geq 0)$, by $\overline t_i := ih$. 

\noindent Then, for all $x = (\theta,\nu) \in E$, we define the sequence $(\overline y_i(x), i \geq 0)$, the classical Euler scheme, iteratively by
\begin{equation*}
\left\{
\begin{array}{l}
\overline y_{i+1}(x) = \overline y_i(x) + h f_\theta(\overline y_i(x)), \\
\overline y_0(x) = \nu,
\end{array}
\right.
\end{equation*}

\noindent to emphasize its dependence on the initial condition. Finally, for all $x = (\theta,\nu) \in E$, we set
\begin{equation}\label{continuous_euler_scheme}
\overline \phi_\theta(t,\nu)
:=
\overline y_i(x) + (t - \overline t_i) f_\theta(\overline y_i(x)),
\hspace{0.5cm}
\forall t \in [\overline t_i,\overline t_{i+1}].
\end{equation}

\noindent We construct the approximating process $(\overline x_t)$ as follows. Its continuous component starts from $\nu_0$ at time 0 and follows the flow $\overline \phi_{\theta_0}(t,\nu_0)$ until the first jump time $\overline T_1$ that we construct by \eqref{T1} and \eqref{tau1} of section \ref{construction} where we replace $\Phi_{\theta_0}( T^*_k,\nu_0)$ by $\overline \phi_{\theta_0}( T^*_k,\nu_0)$. At time $\overline T_1$ the continuous component of $\overline x_{\overline T_1}$ is equal to $\overline \phi_{\theta_0}( \overline T_1,\nu_0):=\overline \nu_1$ since there is no jump in the continuous component. The discrete component jumps to $\overline \theta_1$. We iterate this procedure with the new flow $\overline \phi_{\overline \theta_1}(t- \overline T_1,\overline \nu_1)$ until the next jump time $\overline T_2$ given by \eqref{T1} and \eqref{tau1} with $\overline \phi_{\overline \theta_1}( T^*_k-\overline T_1, \overline \nu_1)$ and so on. We proceed by iteration to construct $(\overline x_t)$ on $[0,T]$. 

\noindent Consequently, the discretisation grid for $(\overline x_t)$ on the interval $[0,T]$ 
is random and is formed by the points $\overline T_n + kh$ for $n = 0,\ldots, \overline N_T$ and $k = 0, \ldots, \lfloor (\overline T_{n+1} \smin T - \overline T_n)/h \rfloor$. This differs from the SDE case where the classical grid is fixed. 

\bigskip \noindent By classical results of numerical analysis (see \cite{ODE} for example), the continuous Euler scheme \eqref{continuous_euler_scheme} (also called Euler polygon) satisfies estimate \eqref{error_flow}. If we choose $\Phi_\theta = \overline \phi_\theta$ in the above construction then the corresponding PDP $(\overline x_t)$ is not Markov since the functions $\overline \phi_\theta(.,\nu)$ do not satisfy the semi-group property (see \cite{jacobsen}). 
\color{black}

\subsection{Thinning representation for the marginal distribution of a PDP}\label{section_distrib_representation}

\noindent The sequence $(T_n,(\theta_n,\nu_n), n\geq 0)$ is an $\bb R_+ \times E$-valued Markov chain with respect to its natural filtration $\cali F_n$ and with kernel $K$ defined by
\begin{equation}\label{PDMPkernel}
K\Big{(}(t,\theta,\nu),dudjdz\Big{)} 
:= 
\ind{u \geq t}\, \lambda(\theta,\Phi_\theta(u-t,\nu))e^{-\int_{0}^{u-t}\lambda(\theta,\Phi_\theta(s,\nu))ds} Q((\theta,\Phi_\theta(u-t,\nu)),djdz) du\,  .
\end{equation}

\noindent For $n \geq 0$, the law of the random variable $T_n-T_{n-1}$ given $\cali F_{n-1}$ admits the density given for $t \geq 0$ by 
\begin{equation}\label{interjumpdensity}
\lambda(\theta_{n-1},\Phi_{\theta_{n-1}}(t,\nu_{n-1})) e^{- \int_0^t \lambda(\theta_{n-1},\Phi(s,\nu_{n-1}))ds}.
\end{equation}

\noindent Classically the marginal distribution of $x_t$ is expressed using \eqref{PDMP}, the intensity $\lambda$ via (\ref{interjumpdensity}) and the kernel $K$ (see \eqref{PDMPkernel}). Indeed for fixed $x_0=x \in E$ and for any bounded measurable function $g$ we can write,
\begin{align}
\E \left[ g(x_t) \right] 
&= 
\sum_{n \geq 0} \E \left[ g( \theta_n, \Phi_{\theta_n}(t-T_n,\nu_n) ) \ind{N_t =n} \right] \nonumber\\
&= 
\sum_{n \geq 0} \E \left[ g( \theta_n,\Phi_{\theta_n}(t-T_n,\nu_n) ) \ind{T_n \leq t} \E [ \ind{T_{n+1}>t} | \cali F_n ] \right] \nonumber\\
&=
\sum_{n \geq 0} \E \left[ g(\theta_n, \Phi_{\theta_n}(t-T_n,\nu_n) ) \ind{T_n \leq t} e^{-\int_0^{t-T_n} \lambda(\theta_n, \Phi_{\theta_n}(u,\nu_n) ) du } \right] \nonumber
\\&= 
\sum_{n \geq 0} \int_0^t \int_E g(\theta, \Phi_\theta(t-s,\nu) ) e^{-\int_0^{t-s} \lambda( \theta, \Phi_\theta(u,\nu) ) du } K^n ((0,x),ds d\theta d\nu) \nonumber
\end{align}  

\noindent where $K^0 := \delta$ and $K^n = K \circ \ldots \circ K$ $n$ times, that is 
\begin{equation*}
\int_0^t \int_E K^n \left( (0,x),dsdy \right) 
= 
\int_0^t \int_E \int_{(\bb R_+ \times E)^{n-1}} K((0,x),dt_1 dy_1) \ldots K((t_{n-1},y_{n-1}),ds dy).
\end{equation*}

\noindent However since we have constructed $(x_t)$ by thinning, we would prefer to express the distribution of $x_t$ using the upper bound $\lambda^*$, the Poisson process $( N^*_t, t \geq 0)$ and the sequences $(U_k, k\in {\mathbb N})$, $(V_k, k\in {\mathbb N})$.

\begin{proposition}\label{marginaldistribution2}
Let $(x_t, t \in [0,T])$ be a PDP with characteristics $(\Phi,\lambda,Q)$ constructed in section \ref{construction} and let $n \in \bb N$. Then
\begin{align*}
\E[g(x_t)\ind{\{N_t=n\}}]
&=  
\sum_{1\leq p_1<p_2....<p_n\leq m}  \sum_{\theta\in \Theta} \E[Q(x^-_{T^*_{p_{n-1}}},\theta)\, g( \theta, \Phi_{\theta}(t-T^*_{p_n},\nu_n))\ind{\{\tau_i=p_i, 1\leq i\leq n, N^*_t=m\}}\\
&
\prod_{q=p_n+1}^{m}(1-\frac{\lambda(\theta,\Phi_{\theta}(T_{q}^*-T_{p_n}^*,\nu_n))}{\lambda^*})].
\end{align*}
\end{proposition}

\bigskip \noindent The following proposition and its corollaries will be useful in section \ref{section_strg_estimate}. In their statements $(x_t, t \in [0,T])$ and $(\tilde x_t, t \in [0,T])$ are PDPs constructed in section \ref{construction} using the same data $( N^*_t)$, $(U_k)$, $(V_k)$ and the same initial point $x \in E$ but with different sets of characteristics.

\noindent The following results are inspired by the change of probability introduced in \cite{Xia} where the authors are interested in the application of the MLMC to jump-diffusion SDEs with state-dependent intensity. In our case, we need a change of probability which guarantees not only that the processes jump at the same times but also in the same states. 

\begin{proposition}\label{DECOMP2} Let us denote by $(\Phi,\lambda,Q)$ ( resp. $(\Phi,\tilde \lambda,\tilde Q)$) the characteristics of $(x_t)$ (resp. $(\tilde x_t)$). Let us assume that $\tilde\lambda$ and $\tilde Q$ depend only on $\theta$, that $\tilde Q$ is always positive and $0< \tilde \lambda(\theta) < \lambda^*$ for all $\theta \in \Theta$. 
For all integer $n$, let us define on the event $\{ \tilde N_t = n \},$ 
\begin{equation*}
\tilde Z_n=\frac{Q(\tilde x^-_{T^*_{\tilde \tau_n}},\tilde\theta_n)}{\tilde Q(\tilde\theta_{n-1},\tilde\theta_n)} \left( \left(1-\frac{\tilde\lambda(\tilde\theta_n)}{\lambda^*} \right)^{N_t^*-\tilde\tau_n} \right)^{-1} \prod_{q=\tilde\tau_n+1}^{N_t^*} \left( 1-\frac{\lambda(\tilde\theta_n,\Phi_{\tilde\theta_n}(T_{q}^*-T_{\tilde \tau_n}^*,\tilde\nu_n))}{\lambda^*} \right),
\end{equation*}
the product being equal to $1$ if $\tilde \tau_n=N_t^*$ and for all $1\leq \ell\leq n-1,$
\begin{align*}
\tilde Z_\ell=&\frac{Q(\tilde x^-_{T^*_{\tilde \tau_\ell}},\tilde\theta_\ell)}{\tilde Q(\tilde\theta_{\ell-1},\tilde\theta_\ell)}\, \Bigg(\frac{\tilde\lambda(\tilde\theta_\ell)}{\lambda^*}\left(1-\frac{\tilde\lambda(\tilde\theta_\ell)}{\lambda^*} \right)^{\tilde\tau_{\ell+1}-\tilde\tau_\ell-1}\Bigg)^{-1}\nonumber\\
&\frac{\lambda(\tilde\theta_\ell,\Phi_{\tilde\theta_\ell}(T_{\tilde\tau_{\ell+1}}^*-T_{\tilde\tau_\ell}^*,\tilde\nu_\ell))}{\lambda^*}\prod_{q=\tilde\tau_\ell +1}^{\tilde\tau_{\ell+1}-1} \left(1-\frac{\lambda(\tilde\theta_\ell,\Phi_{\tilde\theta_\ell}(T_{q}^*-T_{\tilde\tau_\ell}^*,\tilde\nu_\ell))}{\lambda^*} \right),\\
\tilde Z_0=&\left( \frac{\tilde \lambda(\tilde \theta_0)}{\lambda^*} \left( 1-\frac{\tilde \lambda(\tilde \theta_0)}{\lambda^*} \right)^{\tilde \tau_1 -1} \right)^{-1}
\frac{\lambda(\tilde \theta_{0},\Phi_{\tilde \theta_{0}}(T^*_{\tilde \tau_1},\tilde \nu_{0}))}{\lambda^*}  \prod_{q=1}^{\tilde \tau_1 -1} \left( 1-\frac{\lambda(\tilde \theta_{0},\Phi_{\tilde \theta_{0}}(T^*_q,\tilde \nu_{0}))}{\lambda^*} \right),\\
\tilde R_n=&\tilde Z_n\, \, \, \prod_{\ell=0}^{n-1} \, \tilde Z_\ell. 
\end{align*}

\noindent Then, for all $n \geq 0$ we have
\begin{eqnarray*}
\E[g(\tilde x_t)\, \tilde R_n\, \ind{\{\tilde N_t=n\}}]=\E[g(x_t)\, \ind{\{N_t=n\}}].
\end{eqnarray*}
\end{proposition}

\begin{cor}\label{cor2_prop2}
Under the assumptions of Proposition \ref{DECOMP2}, setting $\tilde R_t = \tilde R_{\tilde N_t}$, we have
\begin{equation*}
\E[g(\tilde x_t) \tilde R_t ]=\E[g(x_t)].
\end{equation*}
\end{cor}

\begin{Rem}
Proposition \ref{DECOMP2} looks like a Girsanov theorem (see \cite{palmowski2002}) however we do not use the martingale theory here.
\end{Rem}

\begin{Rem}\label{generalisation_change_characteristics}
We have chosen to state Proposition \ref{DECOMP2} with a PDP $(\tilde x_t)$ whose intensity and transition measure only depend on $\theta$ for readability purposes. Actually the arguments of the proof are valid for non homogeneous intensity and transition measure of the form $\tilde \lambda (x,t)$ and $\tilde Q((x,t),dy)$ for $x=(\theta,\nu) \in E$. A possible choice of such characteristics is $\tilde \lambda (x,t) = \lambda(\theta, \tilde \Phi_\theta(t,\nu))$ and  $\tilde Q((x,t),dy) = Q( (\theta,\tilde \Phi_\theta(t,\nu)), dy )$ for $\tilde \Phi$ a given function. This remark will be implemented in section \ref{section_numerical_results}. 
\end{Rem}

\begin{cor}\label{cor1_propr2} Let $(\Phi,\lambda,Q)$ (resp. $(\tilde \Phi, \lambda, Q)$) be the set of characteristics of $(x_t)$ (resp. $(\tilde x_t))$. We assume that $Q$ is always positive and that $0<  \lambda(x) < \lambda^*$ for all $x \in E$.
Let $(\mu_n)$ be the sequence defined by $\mu_0 = \nu$ and $\mu_n = \tilde \Phi_{\theta_{n-1}}(T_n-T_{n-1},\mu_{n-1})$ for $n \geq 1$.  For all integer $n$, let us define on the event $\{ N_t = n \}$, 
\begin{align*}
\tilde Z_n
&=
\frac{Q \big( (\theta_{n-1},\mu_n),\theta_n \big)}{Q \big( (\theta_{n-1},\nu_n),\theta_n \big)}
\left( \prod_{q=\tau_n+1}^{N_t^*} 1-\frac{\lambda \big( \theta_n, \Phi_{\theta_n}(T^*_q - T^*_{\tau_n},\nu_n) \big)}{\lambda^*} \right)^{-1} 
\\&
\prod_{q=\tau_n+1}^{N_t^*} \left(1-\frac{\lambda \big( \theta_n, \tilde \Phi_{\theta_n}(T^*_q - T^*_{\tau_n},\mu_n) \big)}{\lambda^*} \right),
\end{align*}
the products being equal to $1$ if $ \tau_n=N_t^*$ and for all $1\leq \ell\leq n-1,$
\begin{align*}
\tilde Z_\ell=&
\frac{Q \big( (\theta_{\ell-1},\mu_\ell),\theta_\ell \big)}{ Q \big( (\theta_{\ell-1},\nu_\ell),\theta_\ell \big)}\, 
\Bigg( \frac{\lambda \big( \theta_\ell, \Phi_{\theta_\ell}(T^*_{\tau_{\ell+1}} - T^*_{\tau_\ell},\nu_\ell) \big)}{\lambda^*}\prod_{q=\tau_\ell +1}^{\tau_{\ell+1}-1} \left( 1-\frac{\lambda \big( \theta_\ell, \Phi_{\theta_\ell}(T^*_{q} - T^*_{\tau_\ell},\nu_\ell) \big)}{\lambda^*} \right) \Bigg)^{-1} \\
&\frac{\lambda \big( \theta_\ell, \tilde \Phi_{\theta_\ell}(T^*_{\tau_{\ell+1}} - T^*_{\tau_\ell},\mu_\ell) \big)}{\lambda^*}\prod_{q=\tau_\ell +1}^{\tau_{\ell+1}-1} \left( 1-\frac{\lambda \big( \theta_\ell, \tilde \Phi_{\theta_\ell}(T^*_{q} - T^*_{\tau_\ell},\mu_\ell) \big)}{\lambda^*} \right),
\nonumber\\
\tilde Z_0=& 
\Bigg(\frac{\lambda \big( \theta_0, \Phi_{\theta_0}(T^*_{\tau_{1}},\nu_0) \big)}{\lambda^*} \prod_{q=1}^{\tau_{1}-1} \left( 1-\frac{\lambda \big( \theta_0, \Phi_{\theta_0}(T^*_{q},\nu_0) \big)}{\lambda^*} \right) \Bigg)^{-1} \nonumber\\
& \frac{\lambda \big( \theta_0, \tilde \Phi_{\theta_0}(T^*_{\tau_{1}},\mu_0) \big)}{\lambda^*} \prod_{q=1}^{\tau_{1}-1} \left( 1-\frac{\lambda \big( \theta_0, \tilde \Phi_{\theta_0}(T^*_{q},\mu_0) \big)}{\lambda^*} \right),
\nonumber\\
\tilde R_n=&\tilde Z_n\, \, \, \prod_{\ell=0}^{n-1} \, \tilde Z_\ell. \nonumber
\end{align*}

\noindent Then, for all $n \geq 0$ we have
\begin{eqnarray*}
\E[g\left(\theta_n, \tilde \Phi_{\theta_n}(t-T_n,\mu_n) \right)\, \tilde R_n\, \ind{\{ N_t=n\}}]=\E[g(\tilde x_t)\, \ind{\{\tilde N_t=n\}}].
\end{eqnarray*}
\end{cor}

\begin{proof}[Proof of Proposition \ref{marginaldistribution2}] It holds that $\{N_t=n, \tau_i=p_i,\, 1\leq i\leq n\}\subset \{N_t^*\geq p_n\}$. Then
$$
\E[g(x_t)\ind{\{N_t=n\}}]=\sum_{1\leq p_1<p_2<...<p_{n}\leq m}\, \, \E[g(x_t)\ind{\{N_t=n, \tau_i=p_i,\, 1\leq i\leq n, N_t^*=m\}}].
$$
\noindent The set $\{N_t=n, \tau_i=p_i,\, 1\leq i\leq n, N_t^*=m\}$ is equivalent to the following

\noindent - $N_t^*=m$,

\noindent - among the times $T^*_\ell, 1\leq \ell\leq m$ exactly $n$ are accepted by the thinning method they are the $T^*_{p_i}, 1\leq i\leq n$, all the others are rejected. 

\noindent We proceed by induction starting from the fact that all the $T^*_q, \, p_n+1\leq q\leq m$ are rejected which corresponds to the event 
\begin{equation*}
\forall\quad  \, p_n+1\leq q\leq m,\quad \quad U_q> \frac{\lambda(\theta_n,\Phi_{\theta_n}(T_q^*-T_{p_n}^*,\nu_n))}{\lambda^*}.
\end{equation*}
The random variable  $\ind{\{\tau_i=p_i,\,  1\leq i\leq n\}}$ depends on $(\theta_\ell, \nu_\ell, 1\leq \ell \leq n-1, T^*_i, 1\leq i\leq p_n, U_j, 1\leq j\leq p_n )$ where by construction $\nu_\ell=\phi_{\theta_{\ell-1}}(T^*_{p_\ell}-T^*_{p_{\ell-1}},\nu_{\ell-1})$, $\theta_\ell=H((\theta_{\ell-1},\nu_\ell),V_\ell)$ which implies that $(\theta_\ell, \nu_\ell, 1\leq \ell \leq n-1)$ depend on $(T^*_i, 1\leq i\leq p_{n-1}, U_j, 1\leq j\leq p_{n-1}, V_k, 1\leq k\leq n-1)$. Thus $V_n$ is independent of all the other random variables of thinning that are present in $g(x_t)\ind{\{N_t=n, \tau_i=p_i,\, 1\leq i\leq n, \, N_t^*=m\}}$. The conditional expectation of $g(x_t)\ind{\{N_t=n, \tau_i=p_i,\, 1\leq i\leq n, N_t^*=m\}}$ w.r.t. the vector $(T^*_i, 1\leq i\leq m+1, U_j, 1\leq j\leq m, V_k, 1\leq k\leq n-1)$ is therefore an expectation indexed by this vector as parameters. Since the law of $H(x,V_n)$ is $Q(x,\cdot)$ for all $x \in E$ we obtain for $p_1<p_2<...<p_n\leq m$,
\begin{align}
&\E[g(x_t)\ind{\{N_t=n, \tau_i=p_i,\, 1\leq i\leq n, \, N_t^*=m\}}]\nonumber\\
&= \E[\sum_{\theta\in \Theta} Q(x^-_{T^*_{p_{n-1}}},\theta)\, g( \theta, \Phi_{\theta}(t-T^*_{p_n},\nu_n))\nonumber\\
&F(\theta,U_j,1\leq j\leq m,\, T^*_\ell,1\leq \ell\leq m+1, V_k, 1\leq k\leq n-1)],\label{lastgen 2}
\end{align}

\noindent with 
\begin{eqnarray*}
&&F(\theta,U_j,1\leq j\leq m,\, T^*_\ell,1\leq \ell\leq m+1, V_k, 1\leq k\leq n-1)\nonumber\\
&=&\ind{\{N_t^*=m, \tau_i=p_i,\,  1\leq i\leq n\}}
\prod_{q=p_n+1}^{m}\, \ind{U_{q}> \frac{\lambda(\theta,\Phi_{\theta}(T_{q}^*-T_{p_n}^*,\nu_n))}{\lambda^*}}.
\end{eqnarray*}

\noindent In (\ref{lastgen 2}) the random variables $(U_{q}, \, p_n+1\leq q\leq m)$ are independent of the vector $(T^*_i, 1\leq i\leq m+1, U_j, 1\leq j\leq p_n, V_k, 1\leq k\leq n-1)$. Conditioning by this vector we obtain
\begin{eqnarray}
&&\E[g(x_t)\ind{\{N_t=n, \tau_i=p_i,\, 1\leq i\leq n, \, N_t^*=m\}}] \nonumber\\
&=&\sum_{\theta\in \Theta}\, \E[Q(x^-_{T^*_{p_{n-1}}},\theta)\, g( \theta, \Phi_{\theta}(t-T^*_{p_n},\nu_n))\ind{\{N_t^*=m, \tau_i=p_i,\,  1\leq i\leq n\}}\nonumber\\
&&\prod_{q=p_n+1}^{m}(1-\frac{\lambda(\theta,\Phi_{\theta}(T_{q}^*-T_{p_n}^*,\nu_n))}{\lambda^*})].\nonumber
\end{eqnarray}

\noindent We can iterate on the latter form by first conditioning $V_{n-1}$ by all the other r.v. and then conditioning $(U_{q}, \, p_{n-1}+1\leq q \leq p_n)$ by all the remaining ones and so on. However the terms that appear do not have the same structure since the $U_{q}$ correspond to a rejection for $p_{n-1}+1 \leq q \leq p_n-1$ whereas $U_{p_{n}}$ corresponds to an acceptation. So that the next step yields  
\begin{eqnarray}
&&\E[g(x_t)\ind{\{N_t=n, \tau_i=p_i,\, 1\leq i\leq n, \, N_t^*=m\}}] \nonumber\\
&=&\sum_{\alpha\in \Theta}\sum_{\theta\in \Theta}\, \E[Q(x^-_{T^*_{p_{n-2}}},\alpha)Q\left( (\alpha,\nu_n),\theta \right)\, g( \theta, \Phi_{\theta}(t-T^*_{p_n},\nu_n))\ind{\{N_t^*=m, \tau_i=p_i,\,  1\leq i\leq n-1\}}\nonumber\\
&&\frac{\lambda(\alpha,\Phi_{\alpha}(T_{p_n}^*-T_{p_{n-1}}^*,\nu_{n-1}))}{\lambda^*}\prod_{q=p_{n-1}+1}^{p_n-1}(1-\frac{\lambda(\alpha,\Phi_{\alpha}(T_{q}^*-T_{p_{n-1}}^*,\nu_{n-1}))}{\lambda^*})\nonumber\\
&&\prod_{q=p_n+1}^{m}(1-\frac{\lambda(\theta,\Phi_{\theta}(T_{q}^*-T_{p_n}^*,\nu_n))}{\lambda^*})],\label{iteration1}
\end{eqnarray}

\noindent where we write $\nu_n$ for simplicity keeping in mind that  $\nu_n=\Phi_{\theta_{n-1}}(T^*_{p_n}-T^*_{p_{n-1}},\nu_{n-1})=\Phi_{\theta_{n-1}}(T^*_{p_n}-T^*_{p_{n-1}},\Phi_{\theta_{n-2}}(T^*_{p_{n-1}}-T^*_{p_{n-2}},\nu_{n-2}))=\Phi_{\alpha}(T^*_{p_n}-T^*_{p_{n-1}},\Phi_{\theta_{n-2}}(T^*_{p_{n-1}}-T^*_{p_{n-2}},\nu_{n-2}))$.

\noindent Moreover the previous arguments apply to $\E(g(x_t) f(\theta_i,\nu_i, 1\leq i\leq n-1, \theta_n,\nu_n, T^*_k, 1\leq k\leq m)\,  \ind{\{N_t=n, \tau_i=p_i,\, 1\leq i\leq n, \, N_t^*=m\}})$ and provide
\begin{eqnarray}
&&\E[g(x_t) f(\theta_i,\nu_i, 1\leq i\leq n-1, \theta_n,\nu_n, T^*_k, 1\leq k\leq m)\,  \ind{\{N_t=n, \tau_i=p_i,\, 1\leq i\leq n, \, N_t^*=m\}}]\nonumber\\
&&=\sum_{\theta\in \Theta}  \E[Q(x^-_{T^*_{p_{n-1}}},\theta) g( \theta, \Phi_{\theta}(t-T^*_{p_n},\nu_n))f(\theta_i,\nu_i, 1\leq i\leq n-1, \theta, \nu_n, T^*_k, 1\leq k\leq m)\nonumber\\
&& \ind{\{N_t^*=m, \tau_i=p_i,\, 1\leq i\leq n\}} \, \prod_{q=p_n+1}^{m}(1-\frac{\lambda(\theta,\Phi_{\theta}(T_{q}^*-T_{p_n}^*,\nu_n))}{\lambda^*})]\label{pourR}.
\end{eqnarray}
\end{proof}

\noindent We prove below Proposition \ref{DECOMP2}. The other statements can be proved analogously.

\begin{proof}[Proof of Proposition \ref{DECOMP2}] By assumption the (jump) characteristics $(\tilde \lambda, \tilde Q)$ of $(\tilde x_t)$ 
depend only on $\theta$. Let $p_1<p_2<...<p_n\leq m$. Applying the same arguments as in (\ref{pourR}) to $(\tilde x_t)$ and using the definitions of $\tilde Z_\ell, \, 0\leq \ell\leq n$ and $\tilde R_n$ we obtain, 
\begin{eqnarray*}
&&\E[g(\tilde x_t)\, \tilde R_n\, \ind{\{\tilde N_t=n, \tilde\tau_i=p_i, 1\leq i\leq n, N_t^*=m\}}]\\
&&=\sum_{\theta\in \Theta}  \E[\tilde Q(\tilde\theta_{n-1},\theta)\, g( \theta, \Phi_{\theta}(t-T^*_{p_n},{\tilde\nu}_n))\, \tilde Z_n\, \prod_{\ell=0}^{n-1} \, \tilde Z_\ell\,
 \ind{\{N_t^*=m, {\tilde\tau}_i=p_i,\, 1\leq i\leq n\}}] \, \, (1-\frac{\tilde\lambda(\theta)}{\lambda^*})^{m-p_n}\\
&&=\sum_{\theta\in \Theta}\, \E[\tilde Q(\tilde\theta_{n-1},\theta)\, g( \theta, \Phi_{\theta}(t-T^*_{p_n},\tilde\nu_n)) \, \prod_{\ell=0}^{n-1} \, \tilde Z_\ell\,\,  \ind{\{N_t^*=m, \tilde\tau_i=p_i,\,  1\leq i\leq n\}}\, (1-\frac{\tilde\lambda(\theta)}{\lambda^*})^{m-p_n}\\
&&\Big((1-\frac{\tilde\lambda(\theta)}{\lambda^*})^{m-p_n}\Big)^{-1}\frac{Q(\tilde x^-_{ T^*_{p_{n-1}}},\theta)}{\tilde Q(\tilde\theta_{n-1},\theta)}\prod_{q=p_n+1}^{m}(1-\frac{\lambda(\theta,\Phi_{\theta}(T_{q}^*-T_{p_n}^*,\tilde\nu_n))}{\lambda^*})]\\
&&=\sum_{\theta\in \Theta}\, \E[Q(\tilde x^-_{T^*_{p_{n-1}}},\theta)\, g( \theta, \Phi_{\theta}(t-T^*_{p_n},\tilde\nu_n)) \,\tilde Z_{n-1}\,  \prod_{\ell=0}^{n-2} \, \tilde Z_\ell\,\,  \ind{\{N_t^*=m, \tilde\tau_i=p_i,\,  1\leq i\leq n\}}\\
&&\, \prod_{q=p_n+1}^{m}(1-\frac{\lambda(\theta,\Phi_{\theta}(T_{q}^*-T_{p_n}^*,\tilde\nu_n))}{\lambda^*})].
\end{eqnarray*}

\noindent We iterate the previous argument based on the use of (\ref{pourR}) and we use the definition of $\tilde Z_{n-1}$ to obtain
\color{black}
\begin{eqnarray*}
&&\E[g(\tilde x_t) \tilde R_n \ind{\{\tilde N_t=n, \tilde\tau_i=p_i,\, 1\leq i\leq n, \, N_t^*=m\}}] \nonumber\\
&=&\sum_{\alpha\in \Theta}\sum_{\theta\in \Theta}\, \E[Q(\tilde x^-_{T^*_{p_{n-2}}},\alpha) Q( (\alpha,\tilde \nu_n),\theta)\, g( \theta, \Phi_{\theta}(t-T^*_{p_n},\tilde\nu_n)) \nonumber\\
&&\prod_{\ell=0}^{n-2} \, \tilde Z_\ell\, \,  \ind{\{N_t^*=m, \tilde\tau_i=p_i,\,  1\leq i\leq n-1\}}
\prod_{q=p_n+1}^{m}(1-\frac{\lambda(\theta,\Phi_{\theta}(T_{q}^*-T_{p_n}^*,\tilde\nu_n))}{\lambda^*})\\
&&\frac{\lambda(\alpha,\Phi_{\alpha}(T_{p_n}^*-T_{p_{n-1}}^*,\tilde\nu_{n-1}))}{\lambda^*}\prod_{q=p_{n-1}+1}^{p_{n}-1}(1-\frac{\lambda(\alpha,\Phi_{\alpha}(T_{q}^*-T_{p_{n-1}}^*,\tilde\nu_{n-1}))}{\lambda^*})],
\end{eqnarray*}
where for short  $\tilde\nu_{n}=\phi_{\alpha}(T^*_{p_n}-T^*_{p_{n-1}},\tilde\nu_{n-1})$ and $\tilde\nu_{n-1}=\phi_{\tilde\theta_{n-2}}(T^*_{p_{n-1}}-T^*_{p_{n-2}},\tilde\nu_{n-2})$. Comparing the latter expression to  \eqref{iteration1} and using an induction we conclude that 
\begin{equation*}
\E[g(\tilde x_t) \tilde R_n \ind{\{\tilde N_t=n, \tilde\tau_i=p_i,\, 1\leq i\leq n, \, N_t^*=m\}}]=\E[g(x_t)\,\ind{\{N_t=n, \tau_i=p_i, 1\leq i\leq n, N_t^*=m\}}].
\end{equation*}
It remains to sum up on $p_i, 1\leq i\leq n$ and $m$.
\end{proof}

\section{Strong error estimates}\label{section_strg_estimate}

In this section we are interested in strong error estimates. Below, we state the main assumptions and theorems of this section, the proofs are given in sections \ref{strong_error_x}, \ref{strong_error_tilde_x} respectively.

\begin{hyp}\label{hyp_lambda_lip}
For all $\theta \in \Theta$ and for all $A \in \cali B(\Theta)$, the functions $\nu \mapsto \lambda(\theta,\nu)$ and $\nu \mapsto Q( (\theta,\nu),A)$ are Lipschitz with constants $L_\lambda>0$, $L_Q>0$ respectively independent of $\theta$.
\end{hyp}

\begin{theorem}\label{thm_strong_cv}
Let $\Phi_\theta$ and $\overline \Phi_\theta$ satisfying \eqref{error_flow} and let $(x_t, t \in [0,T])$ and $(\overline x_t, t \in [0,T])$ be the corresponding PDPs constructed in section \ref{construction} with $x_0 = \overline x_0 = x$ for some $x \in E$. Assume that $\Theta$ is finite and that $\lambda$ and $Q$ satisfy Assumption \ref{hyp_lambda_lip}. Then, for all bounded functions $F: E \rightarrow \bb R$ such that for all $\theta \in \Theta$ the function $\nu \mapsto F(\theta,\nu)$ is $L_F$-Lipschitz where $L_F$ is positive and independent of $\theta$, there exists constants $V_1 > 0$ and $V_2 > 0$ independent of the time step $h$ such that  
\begin{equation*}
\E \left[ |F(\overline x_T) - F(x_T)|^2 \right] \leq V_1 h + V_2 h^2.
\end{equation*}
\end{theorem}

\begin{Rem}
When the numerical scheme $\overline \Phi_\theta$ is of order $p \geq 1$, which means $\sup_{t \in [0,T]} |\Phi_\theta(t,\nu_1) - \overline \Phi_\theta(t,\nu_2)| \leq e^{C_1 T}|\nu_1-\nu_2| + C_2 h^p $ we have $\E \left[ |F(\overline x_T) - F(x_T)|^2 \right] \leq V_1 h^p + V_2 h^{2p}$.
\end{Rem}

\begin{hyp}\label{hyp_tilde}
There exist positive constants $\rho$, $\tilde \lambda_{\min}$, $\tilde \lambda_{\max}$ such that for all $(i,j) \in \Theta^2$, $\rho \leq \tilde Q(i,j)$ and $\tilde \lambda_{\min} \leq \tilde\lambda(i) \leq \tilde \lambda_{\max}<\lambda^*$.
\end{hyp}

\begin{theorem}\label{thm_strong_cv_tilde}
Let $\Phi_\theta$ and $\overline \Phi_\theta$ satisfying \eqref{error_flow} and let $(\tilde x_t, t \in [0,T])$ and $(\underline{ \tilde x}_t, t \in [0,T])$ be the corresponding PDPs constructed in section \ref{construction} with $\underline{\tilde x}_0 =\tilde x_0 =  x$ for some $x \in E$. 
Let $(\tilde R_t, t \in [0,T])$ and $(\underline{ \tilde R}_t, t \in [0,T])$ be defined as in Corollary \ref{cor2_prop2}.  
Under assumptions \ref{hyp_lambda_lip} and \ref{hyp_tilde} and for all bounded functions $F: E \rightarrow \bb R$ such that for all $\theta \in \Theta$ the function $\nu \mapsto F(\theta,\nu)$ is $L_F$-Lipschitz ($L_F > 0$), there exists a positive constant $\tilde V_1$ independent of the time step $h$ such that
\begin{equation*}
\bb E \left[ |F(\underline{\tilde x}_T) \underline{\tilde R}_T - F(\tilde x_T) \tilde R_T|^2 \right]
\leq 
\tilde V_1 h^2,
\end{equation*}
 where $\tilde R_T$ has been defined in Corollary \ref{cor2_prop2}.
\end{theorem}

\noindent We now introduce the random variable $\overline \tau^\dagger$ which will play an important role in the strong error estimate of Theorem \ref{thm_strong_cv} as well as in the identification of the coefficient $c_1$ in the weak error expansion in section \ref{section_weak_expansion} (see the proof of Theorem \ref{weak_error_thm} in section \ref{errfaible}).
\begin{defi}\label{tau*}
Let us define $\overline \tau^\dagger := \inf\left\{ k > 0 : (\tau_k,\theta_k) \neq (\overline \tau_k,\overline \theta_k) \right\}$.
\end{defi}

\noindent The random variable $\overline \tau^\dagger$  enables us to partition the trajectories of the couple $(x_t,\overline x_t)$ in a sense that we precise now. Consider the event 
\begin{equation}\label{min_tau_dag}
\{ \min( T_{\overline \tau^\dagger}, \overline T_{\overline \tau^\dagger}) > T \} = \left\{ N_T = \overline N_T, (T_1,\theta_1) = (\overline T_1,\overline \theta_1),\ldots, (T_{N_T},\theta_{N_T}) = (\overline T_{\overline N_T},\overline \theta_{\overline N_T}) \right\},
\end{equation}

\noindent where $(T_n)$ and $(\overline T_n)$ denote the sequences of jump times of $(x_t)$ and $(\overline x_t)$. On this event $\{ \min( T_{\overline \tau^\dagger}, \overline T_{\overline \tau^\dagger}) > T\}$ the trajectories of the discrete time processes $(T_n,\theta_n)$ and $(\overline T_n, \overline \theta_n)$ are equal for all $n$ such that $T_n \in [0,T]$ (or equivalently $\overline T_n \in [0,T]$). Moreover the complement i.e $\{ \min( T_{\overline \tau^\dagger}, \overline T_{\overline \tau^\dagger}) \leq T \}$ contains the trajectories for which $(T_n,\theta_n)$ and $(\overline T_n,\overline \theta_n)$ differ on $[0,T]$ (there exits $n \leq N_T \smax \overline N_T$ such that $T_n \neq \overline T_n$ or $\theta_n \neq \overline \theta_n$).

\subsection{Preliminary lemmas}

In this section we start with two lemmas which will be useful to prove Theorems \ref{strong_error_x} and \ref{strong_error_tilde_x}.

\begin{lemme}\label{strong_error_lemme1}
Let $ K $ be a finite set. We denote by $|K|$ the cardinal of $K$ and for $i=1,\ldots,|K|$ we denote by $k_i$ its elements. Let $( p_i, 1 \leq i \leq |K| )$ and $(\overline p_i, 1 \leq i \leq |K|)$ be two probabilities on $ K $. Let $a_j := \sum_{i=1}^{j} p_i $ and $ \overline a_j := \sum_{i=1}^{j} \overline p_i$ for all $j \in \{1,\ldots,|K|\}$. 
By convention, we set $a_0 = \overline a_0 := 0$. 
Let $ X $ and $ \overline X $ be two $K$-valued random variables defined by 
\[
X := G(U),
\hspace{0.5cm}
\overline X := \overline G (U),
 \]
where $ U \sim \mathcal U ([0,1]) $, $ G(u) = \sum_{j=1}^{|K|} k_j \ind { a_{j-1} < u \leq a_{j} }$ and $ \overline G (u) = \sum_{j=1}^{|K|} k_j \ind { \overline a_{j-1} < u \leq \overline a_j }$ for all $ u\in [0,1]$. Then, we have
\[
\bb{P}(X \neq \overline X) \leq \sum_{j=1}^{|K|-1}|a_{j}-\overline a_{j}|.
\]
\end{lemme}

\begin{proof}[Proof of Lemma \ref{strong_error_lemme1}]
By definition of $X$ and $\overline X$ and since the intervals $] a_{j-1} , a_j ] \cap ] \overline a_{j-1} , \overline a_j]$ are disjoints for $j=1,\ldots,K$, we have
\begin{equation*}
\bb P ( X = \overline X ) 
=
\sum_{j=1}^{|K|} \bb P \Big ( U \in { ] a_{j-1} , a_j ] \cap ] \overline a_{j-1} , \overline a_j] } \Big ).
\end{equation*}

\noindent Moreover, for all $1 \leq j \leq |K|$, we have
\begin{equation*}
\bb P \Big ( U \in { ] a_{j-1} , a_j ] \cap ] \overline a_{j-1} , \overline a_j] } \Big )
=\left\{
\begin{array}{ll}
0 \hfill \text{ if } ] a_{j-1} , a_j ] \cap ] \overline a_{j-1} , \overline a_j] = \emptyset, \\
a_j \smin \overline a_j - a_{j-1} \smax \overline a_{j-1} \hfill \text{ if } ] a_{j-1} , a_j ] \cap ] \overline a_{j-1} , \overline a_j] \neq \emptyset.
\end{array}
\right.
\end{equation*}

\noindent Thus, denoting by $x^+ := \max(x,0)$ the positive part of $x \in \bb R$ and using that $x^+ \geq x$, we obtain
\begin{equation*}
\bb P ( X = \overline X )
\geq
\sum_{j=1}^{|K|} ( a_j \smin \overline a_j - a_{j-1} \smax \overline a_{j-1} ).
\end{equation*}

\noindent Adding and subtracting $ a_j \smax \overline a_j$ in the the above sum yields
\begin{equation*}
\bb P ( X = \overline X )
\geq
\sum_{j=1}^{|K|} (a_j \smax \overline a_j - a_{j-1} \smax \overline a_{j-1}) + \sum_{j=1}^{|K|}  (a_j \smin \overline a_j - a_j \smax \overline a_j).
\end{equation*}

\noindent The first sum above is a telescopic sum. Since $a_{|K|} = \overline a_{|K|} = 1$ and $a_{0} = \overline a_{0} = 0$, we have $\bb P ( X = \overline X ) \geq 1 - \sum_{j=1}^{|K|-1} | a_j - \overline a_j |$.

\end{proof}

\begin{lemme}\label{lemme_product}
Let $(a_n, n \geq 1)$ and $(b_n, n \geq 1)$ be two real-valued sequences. For all $n \geq 1$, we have 
\[
\prod_{i=1}^n a_i - \prod_{i=1}^n b_i = \sum_{i=1}^n \left( a_i - b_i \right) \prod_{j=i+1}^n a_j \prod_{j=1}^{i-1} b_j
\] 
\end{lemme}

\begin{proof}[Proof of Lemma \ref{lemme_product}]
By induction.
\end{proof}

\subsection{Proof of Theorem \ref{thm_strong_cv}}\label{strong_error_x}

First, we write
\begin{equation*}
\begin{split}
&\bb E \left[ | F(\overline x_T) - F(x_T)|^2 \right]
\\&=
\bb E \left[ \ind{\min(T_{\overline \tau^\dagger}, \overline T_{\overline \tau^\dagger}) \leq T} |F(\overline x_T) - F(x_T)|^2  \right]
+\bb E \left[ \ind{\min(T_{\overline \tau^\dagger}, \overline T_{\overline \tau^\dagger})  > T} |F(\overline x_T) - F(x_T)|^2  \right]
\\&=:
\overline P + \overline D,
\end{split}
\end{equation*}

\noindent where $\overline \tau^\dagger$ is defined in Definition \ref{tau*}. The order of the term $\overline P$ is the order of the probability that the discrete processes $(T_n,\theta_n)$ and $(\overline T_n, \overline \theta_n)$ differ on $[0,T]$. The order of the term $\overline D$ is given by the order of the Euler scheme squared because the discrete processes $(T_n,\theta_n)$ and $(\overline T_n,\overline \theta_n)$ are equal on $[0,T]$. 
In the following we prove that $\overline P = O(h)$ and that $\overline D = O(h^2)$.

\bigskip \noindent \textit{Step 1: estimation of $\overline P$.}  The function $F$ being bounded we have $ \overline P \leq 4 M_F^2 \bb P \left( \min(T_{\overline \tau^\dagger},\overline T_{\overline \tau^\dagger}) \leq T \right)$ where $M_F>0$. Moreover, for $k \geq 1$, $\left\{\overline \tau^\dagger = k \right\} = \left\{ \overline \tau^\dagger > k-1  \right\} \bigcap \left\{ \left(\tau_k,\theta_k \right) \neq \left(\overline \tau_k,\overline \theta_k \right) \right\}$. Hence
\begin{align*}
\bb P \left( \min(T_{\overline \tau^\dagger},\overline T_{\overline \tau^\dagger}) \leq T \right)
&=
\sum_{k \geq 1} \E \left[ \ind{\min(T_k,\overline T_k) \leq T} \ind{\overline \tau^\dagger = k} \right] \\
&=
\sum_{k \geq 1} \E \left[ \ind{\min(T_k,\overline T_k) \leq T} \ind{\overline \tau^\dagger > k-1}  \ind{\left(\tau_k,\theta_k \right) \neq \left(\overline \tau_k,\overline \theta_k \right)} \right] \\
&\leq
\sum_{k \geq 1} \overline J_k + 2 \overline I_k
\end{align*}

\noindent where
\begin{equation}\label{I_J}
\overline J_k
:=
\E \left[\ind{\min(T_k,\overline T_k) \leq T} \ind{\overline \tau^\dagger > k-1}   \ind{\tau_k =  \overline \tau_k} \ind{\theta_k \neq \overline \theta_k} \right],
\hspace{0,5cm}
\overline I_k
:=
\E \left[ \ind{\min(T_k,\overline T_k) \leq T} \ind{\overline \tau^\dagger > k-1}  \ind{\tau_k \neq  \overline \tau_k}  \right].
\end{equation}

\noindent We start with $\overline J_k$. 
First note that, for $k \geq 1$, $\{\tau_k = \overline \tau_k\} = \{T_k = \overline T_k\}$ and that on the event $\{T_k =  \overline T_k\}$, we have $\min(T_k,\overline T_k) = T_k$, so that
$\overline J_k
=
\E \left[\ind{T_k \leq T} \ind{\overline \tau^\dagger > k-1}   \ind{\tau_k =  \overline \tau_k} \ind{\theta_k \neq  \overline \theta_k}  \right].
$
We emphasize that it makes no difference in the rest of the proof if we choose $\min(T_k,\overline T_k) = \overline T_k$. Since $\{ \overline \tau^\dagger > k-1 \} = \bigcap_{i=0}^{k-1} \{ (\tau_i,\theta_i) = (\overline \tau_i,\overline \theta_i) \}$, we can rewrite $\overline J_k$ as follows
\begin{equation}\label{JST1}
\sum_{\Gfrac{1 \leq p_1 < \ldots < p_k}{\alpha_1,\ldots,\alpha_{k-1} \in \Theta}}
\E [ \ind{\{\tau_i= \overline \tau_i =p_i, 1\leq i \leq k\}} \ind{\{\theta_i = \overline \theta_i=\alpha_i, 1\leq i \leq k-1\}} \ind{T^*_{p_k} \leq T}  \ind{\theta_k \neq  \overline \theta_k} ].
\end{equation}

\noindent By construction we have $\theta_k = H((\theta_{k-1},\nu_k),V_k)$ and $\overline \theta_k = H((\overline \theta_{k-1},\overline \nu_k),V_k)$. The random variable $\ind{\{\tau_i= \overline \tau_i =p_i, 1\leq i \leq k\}} \ind{\{\theta_i = \overline \theta_i=\alpha_i, 1\leq i \leq k-1\}} \ind{T^*_{p_k} \leq T}$ depends on the vector $(U_i,1 \leq i \leq p_k, T^*_j, 1 \leq j \leq p_k, V_q, 1 \leq q \leq k-1 )$ which is independent of $V_k$. Conditioning by this vector in \eqref{JST1} and applying Lemma \ref{strong_error_lemme1} yields
\begin{align*}
&\E [ \ind{\{\tau_i= \overline \tau_i =p_i, 1\leq i \leq k\}} \ind{\{\theta_i = \overline \theta_i=\alpha_i, 1\leq i \leq k-1\}} \ind{T^*_{p_k} \leq T}  \ind{\theta_k \neq  \overline \theta_k} ] \\
&\leq
\E \left[ \ind{\{\tau_i= \overline \tau_i =p_i, 1\leq i \leq k\}} \ind{\{\theta_i = \overline \theta_i=\alpha_i, 1\leq i \leq k-1\}} \ind{T^*_{p_k} \leq T} \sum_{j=1}^{|\Theta|-1} \left| a_j( \alpha_{k-1},\overline\nu_k) - a_j(\alpha_{k-1},\nu_k) \right| \right].
\end{align*} 

\noindent From the definition of $a_j$ (see \eqref{aj}), the triangle inequality and since $Q$ is $L_Q$-Lipschitz, we have $\sum_{j=1}^{|\Theta|-1} \left| a_j( \alpha_{k-1},\overline\nu_k) - a_j(\alpha_{k-1},\nu_k) \right| \leq \frac{\left( |\Theta|-1 \right) |\Theta|}{2} L_{Q}  |\overline \nu_k -  \nu_k|$. Since we are on the event $\{\tau_i= \overline \tau_i =p_i, 1\leq i \leq k\} \bigcap \{\theta_i = \overline \theta_i=\alpha_i, 1\leq i \leq k-1\}$, the application of Lemma \ref{strong_error_lemme2} yields $|\overline \nu_k - \nu_k| \leq e^{L T^*_{p_k}} k C h$. Thus $\overline J_k \leq C_1 h \E[\ind{T_k \leq T}k]$ where $C_1$ is a constant independent of $h$. Moreover, $\sum_{k \geq 1 }\ind{T_k \leq T} k = \sum_{k=1}^{N_T} k \leq N_T^2 $ and $\E[ N_T^2] \leq \E[ (N^*_T)^2] < +\infty$ so that $\sum_{k \geq 1} \overline J_k = O(h)$.
From the definition of $\overline I_k$  (see \eqref{I_J}), we can write
\begin{equation*}
\begin{split}
\overline I_k
&=
\E \left[ \ind{\min(T_k,\overline T_k) \leq T} \ind{\overline \tau^\dagger > k-1} ( \ind{\tau_k <  \overline \tau_k} + \ind{\tau_k >  \overline \tau_k} ) \right] \\
&=
\E \left[ \ind{T_k \leq T} \ind{\overline \tau^\dagger > k-1}  \ind{\tau_k <  \overline \tau_k} \right]
+
\E \left[ \ind{\overline T_k \leq T} \ind{\overline \tau^\dagger > k-1} \ind{\tau_k >  \overline \tau_k}  \right]\\
&=:
\overline I^{(1)}_k + \overline I^{(2)}_k.
\end{split}
\end{equation*}

\noindent The second equality above follows since $\{\tau_k < \overline \tau_k \} = \{T_k < \overline T_k\}$ and $\{\tau_k > \overline \tau_k \} = \{T_k > \overline T_k\}$. We only treat the term $\overline I^{(1)}_k$, the term $\overline I^{(2)}_k$ can be treated similarly by interchanging the role of $(\tau_k,T_k)$ and $(\overline \tau_k,\overline T_k)$. Just as in the previous case, we can rewrite $\overline I^{(1)}_k$ as follows
\begin{equation}\label{JST2}
\sum_{\Gfrac{1 \leq p_1 < \ldots < p_k}{\alpha_1,\ldots,\alpha_{k-1} \in \Theta}}
\E [ \ind{\{\tau_i= \overline \tau_i =p_i, 1\leq i \leq k-1\}} \ind{\{\theta_i = \overline \theta_i=\alpha_i, 1\leq i \leq k-1\}} \ind{T^*_{p_k} \leq T} \ind{\tau_k =  p_k} \ind{p_k <  \overline \tau_k} ].
\end{equation}

\noindent In \eqref{JST2} we have $\{\tau_k =  p_k\} \cap \{p_k <  \overline \tau_k\} \subseteq \{ \lambda(\alpha_{k-1}, \overline \Phi_{\alpha_{k-1}}(T^*_{p_k}-T^*_{p_{k-1}},\overline \nu_{k-1})) < U_{p_k} \lambda^* \leq \lambda( \alpha_{k-1},  \Phi_{\alpha_{k-1}}(T^*_{p_k}-T^*_{p_{k-1}}, \nu_{k-1})) \}$. The random variable $\ind{\{\tau_i= \overline \tau_i =p_i, 1\leq i \leq k-1\}}$ $\ind{\{\theta_i = \overline \theta_i=\alpha_i, 1\leq i \leq k-1\}}$ $\ind{T^*_{p_k} \leq T}$ depends on $(U_i,1 \leq i \leq p_{k-1}, T^*_j, 1 \leq j \leq p_k, V_q, 1 \leq q \leq k-1 )$ which is independent of $U_{p_k}$. Conditioning by this vector in \eqref{JST2} yields
\begin{align*}
&\E [ \ind{\{\tau_i= \overline \tau_i =p_i, 1\leq i \leq k-1\}} \ind{\{\theta_i = \overline \theta_i=\alpha_i, 1\leq i \leq k-1\}} \ind{T^*_{p_k} \leq T} \ind{\tau_k =  p_k} \ind{p_k <  \overline \tau_k} ]\\
&\leq
\E [ \ind{\{\tau_i= \overline \tau_i =p_i, 1\leq i \leq k-1\}} \ind{\{\theta_i = \overline \theta_i=\alpha_i, 1\leq i \leq k-1\}} \ind{T^*_{p_k} \leq T} \\
&|\lambda(\alpha_{k-1}, \overline \Phi_{\alpha_{k-1}}(T^*_{p_k}-T^*_{p_{k-1}},\overline \nu_{k-1}))-\lambda( \alpha_{k-1},  \Phi_{\alpha_{k-1}}(T^*_{p_k}-T^*_{p_{k-1}}, \nu_{k-1})) | ]
\end{align*}

\noindent Using the Lipschitz continuity of $\lambda$ then Lemma \ref{strong_error_lemme2} we get that $\overline I^{(1)}_k \leq C_2 h \E[\ind{T_k \leq T} k]$ where $C_2$ is a constant independent of $h$. Concerning the term $\overline I^{(2)}_k$, we will end with the estimate $\overline I^{(2)}_k \leq C_2 h \E[\ind{\overline T_k \leq T} k]$. We conclude in the same way as in the estimation of $\overline J_k$ above that $\sum_{k \geq 1 } \overline I_k = O(h)$.

\bigskip \noindent \textit{Step 2: estimation of $\overline D$.} Note that for $n \geq 0$ we have
$
\{N_T = n \} \cap \{ \min( T_{\overline \tau^\dagger}, \overline T_{\overline \tau^\dagger}) > T \}
=
\{N_T = n \} \cap \{ \overline N_T = n \} \cap \{ \overline \tau^\dagger > n \},
$
\noindent where we can interchange the role of $\{N_T = n \}$ and $\{\overline N_T = n \}$. Thus, using the partition $ \{N_T = n,n \geq 0\}$, we have
\begin{equation*}
\overline D
=
\sum_{n \geq 0} \E \left[ \ind{N_T=n} \ind{\overline N_T=n} \ind{\overline \tau^\dagger > n} \left| F(\theta_n, \overline \Phi_{ \theta_n}(T- T_n,\overline \nu_n)) - F(\theta_n,\Phi_{\theta_n}(T-T_n,\nu_n)) \right|^2  \right] 
\end{equation*}

\noindent The application of the Lipschitz continuity of $F$ and of Lemma \ref{strong_error_lemme2} yields
\begin{equation*}
\left| F(\theta_n, \overline \Phi_{\theta_n}(T-T_n,\overline \nu_n)) - F(\theta_n,\Phi_{\theta_n}(T-T_n,\nu_n)) \right|
\leq 
L_F e^{ L T} (n+1)Ch.
\end{equation*} 

\noindent Then, we have $\overline D \leq C_3 h^2 \sum_{n \geq 0} \E \left[ \ind{N_T=n} (n+1)^2 \right]$ where $C_3$ is a constant independent of $h$. Since $\sum_{n \geq 0} \E \left[ \ind{N_T=n} (n+1)^2 \right] = \E[(N_T+1)^2] \leq \E[(N^*_T+1)^2] < +\infty$, we conclude that $\overline D = O(h^2)$.

\qed

\subsection{Proof of Theorem \ref{thm_strong_cv_tilde}}\label{strong_error_tilde_x}
First we reorder the terms in $\tilde R_T$. We write $\tilde R_T = \tilde{\mathsf Q}_T \tilde{\mathsf S}_T \tilde{\mathsf H}_T$ where
\begin{align}
& \tilde{\mathsf Q}_T = \prod_{l=1}^{\tilde N_T} \frac{Q(\tilde x_{T^*_{\tilde \tau_l}}^-,\tilde \theta_l)}{\tilde Q(\tilde \theta_{l-1},\tilde \theta_l)}\label{QT},\\
& \tilde{\mathsf S}_T= \prod_{l=1}^{\tilde N_T} \frac{\lambda(\tilde \theta_{l-1},\Phi_{\tilde \theta_{l-1}}(T^*_{\tilde \tau_l}-T^*_{\tilde \tau_{l-1}},\tilde \nu_{l-1}))}{\lambda^*} \prod_{k=\tilde \tau_{l-1}+1}^{\tilde \tau_l} (1 - \frac{\lambda(\tilde \theta_{l-1},\Phi_{\tilde \theta_{l-1}}(T^*_{k}-T^*_{\tilde \tau_{l-1}},\tilde \nu_{l-1}))}{\lambda^*})\label{ST} \\
&\prod_{l=\tilde \tau_{\tilde N_T}+1}^{N^*_T} (1 - \frac{\lambda(\tilde \theta_{\tilde N_T},\Phi_{\tilde \theta_{\tilde N_T}}(T^*_{l}-T^*_{\tilde \tau_{\tilde N_T}},\tilde \nu_{\tilde N_T}))}{\lambda^*}), \nonumber \\
& \tilde{\mathsf H}_T = \prod_{l=1}^{\tilde N_T} \left( \frac{\tilde \lambda(\tilde \theta_{l-1})}{\lambda^*} (1-\frac{\tilde \lambda(\tilde \theta_{l-1})}{\lambda^*})^{\tilde \tau_l - \tilde \tau_{l-1}-1} \right)^{-1} \left( (1-\frac{\tilde \lambda(\tilde \theta_{\tilde N_T})}{\lambda^*})^{N^*_T - \tilde \tau_{\tilde N_T}} \right)^{-1}. \label{HT}
\end{align} 

\noindent Likewise we reorder the terms in $\underline{\tilde R}_T$ writing $\underline{\tilde R}_T = \underline{\tilde{\mathsf Q}}_T \underline{\tilde{\mathsf S}}_T \tilde{\mathsf H}_T$ where $\underline{\tilde{\mathsf Q}}_T$ and $\underline{\tilde{\mathsf S}}_T$ are defined as \eqref{QT} and \eqref{ST} replacing $\tilde x$ and $\Phi$ by $\underline{\tilde x}$ and $\overline \Phi$. Since the processes $(\tilde \theta_n)$ and $(\tilde \tau_n)$ do not depend on $\Phi$ or $\overline \Phi$, the term $\tilde{\mathsf H}$ is the same in $\tilde R$ and $\underline{\tilde R}$ .
To prove Theorem \ref{thm_strong_cv_tilde}, let us decompose the problem and write
\begin{equation*}
\begin{split}
|F(\underline{\tilde x}_T) \underline{\tilde R}_T - F(\tilde x_T) \tilde R_T | 
&= 
| (F(\underline{\tilde x}_T) - F(\tilde x_T)) \underline{\tilde R}_T + (\underline{\tilde R}_T- \tilde R_T) F(\tilde x_T) | \\
&\leq
| F(\underline{\tilde x}_T) - F(\tilde x_T) | |\underline{\tilde R}_T| + |\underline{\tilde R}_T- \tilde R_T| |F(\tilde x_T)|,
\end{split}
\end{equation*}

\noindent so that
\begin{equation*}
\begin{split}
\bb E \left[ |F(\underline{\tilde x}_T) \underline{\tilde R}_T - F(\tilde x_T) \tilde R_T|^2 \right]
&\leq
2 \E \left[ |  F(\underline{\tilde x}_T) - F(\tilde x_T) |^2 |\underline{\tilde R}_T|^2 \right]
+
2 \E \left[ |\underline{\tilde R}_T- \tilde R_T|^2 |F(\tilde x_T) |^2 \right]
\\&=: 2 \overline D + 2 \overline C.
\end{split}
\end{equation*}

\noindent In the following we show  that $\overline C = O(h^2)$ and that $\overline D = O(h^2)$.

\bigskip \noindent \textit{Step 1: estimation of $\overline C$.} The function $F$ being bounded we have $\overline C \leq M_F^2 \E \left[ |\underline{\tilde R}_T - \tilde R_T|^2  \right]$ where $M_F$ is a positive constant. Moreover, for all $\theta \in \Theta$, we have $(1-\tilde \lambda(\theta)/\lambda^*)^{-1} \leq (1-\tilde \lambda_{\text{max}}/\lambda^*)^{-1}$ and $(\tilde \lambda(\theta)/\lambda^*)^{-1} \leq (\tilde \lambda_\text{min}/\lambda^*)^{-1}$. Thus, $\tilde H_T \leq \left( \frac{\tilde \lambda_\text{min}}{\lambda^*}(1-\frac{\tilde \lambda_\text{max}}{\lambda^*}) \right)^{- N^*_T}$ and using the definition of $\tilde R$ and $\underline{\tilde R}$ (see \eqref{QT}, \eqref{ST} and \eqref{HT}) we can write
\begin{equation*}
|\underline{\tilde R}_T - \tilde R_T|
\leq
\left( \frac{\tilde \lambda_\text{min}}{\lambda^*}(1-\frac{\tilde \lambda_\text{max}}{\lambda^*}) \right)^{- N^*_T} \left( |\underline{\tilde{\mathsf Q}}_T - \tilde{ \mathsf Q}_T| \tilde{\mathsf S}_T + |\underline{\tilde{\mathsf S}}_T - \tilde{\mathsf S}_T|  \underline{\tilde{\mathsf Q}}_T \right).
\end{equation*} 

\noindent We set $\overline J = |\underline{\tilde{\mathsf Q}}_T - \tilde{ \mathsf Q}_T| \tilde{\mathsf S}_T$ and $\overline I = |\underline{\tilde{\mathsf S}}_T - \tilde{\mathsf S}_T|  \underline{\tilde{\mathsf Q}}_T$. To provide the desired estimate for $\overline C$, we proceed as follows. First, we work $\omega$ by $\omega$ to determine (random) bounds for $\overline J$ and $\overline I$ from which we deduce a (random) bound for $|\underline{\tilde R}_T - \tilde R_T|$. Finally, we take the expectation. We start with $\overline I$. 
For all $(\theta,\nu) \in E$ and for all $t \geq 0$ we have, from Assumption \ref{hyp_lambda1}, that $1-\lambda(\theta,\Phi_\theta(t,\nu))/\lambda^* \leq 1$ and $\lambda(\theta,\Phi_\theta(t,\nu))/\lambda^* \leq 1$. Then, using Lemma \ref{lemme_product} (twice) we have
\begin{equation*}
|\underline{\tilde{\mathsf S}}_T - \tilde{\mathsf S}_T|
\leq
\frac{1}{\lambda^*} \sum_{l=1}^{\tilde N_T+1} \sum_{k=\tilde \tau_{l-1}+1}^{\tilde \tau_l \smin N^*_T} |\lambda(\tilde \theta_{l-1}, \overline \Phi_{\tilde \theta_{l-1}}(T^*_k-T^*_{\tilde \tau_{l-1}},\underline{\tilde \nu}_{l-1}))-\lambda(\tilde \theta_{l-1},  \Phi_{\tilde \theta_{l-1}}(T^*_k-T^*_{\tilde \tau_{l-1}},{\tilde \nu}_{l-1}))|.
\end{equation*}

\noindent Using the Lipschitz continuity of $\lambda$ and Lemma \ref{strong_error_lemme2}, we find that, for all $l=1,\ldots,\tilde N_T+1$ and $k=\tilde \tau_{l-1}+1,\ldots,\tilde \tau_l \smin N^*_T$,
\begin{align*}
|\lambda(\tilde \theta_{l-1}, \overline \Phi_{\tilde \theta_{l-1}}(T^*_k-T^*_{\tilde \tau_{l-1}},\underline{\tilde \nu}_{l-1}))-\lambda(\tilde \theta_{l-1},  \Phi_{\tilde \theta_{l-1}}(T^*_k-T^*_{\tilde \tau_{l-1}},{\tilde \nu}_{l-1}))|
\leq
e^{LT} C h l.
\end{align*}

\noindent Moreover, for all $l=1,\ldots,\tilde N_T+1$ we have $\tilde \tau_l \smin N^*_T-\tilde \tau_{l-1} \leq N^*_T$ so that $|\underline{\tilde{\mathsf S}}_T - \tilde{\mathsf S}_T| \leq N^*_T (N^*_T+1)^2 C_1 h$ where $C_1$ is a positive constant independent of $h$. Finally, since $\underline{\tilde{\mathsf Q}}_T \leq \rho^{-N^*_T}$ we have
\begin{equation}\label{overlineI}
\overline I 
\leq 
\rho^{-N^*_T} N^*_T (N^*_T+1)^2 C_1 h.
\end{equation}

\noindent Now, consider $\overline J$. Note that from Assumption \ref{hyp_lambda1} we have $\tilde{\mathsf S}_T \leq 1$. We use the same type of arguments as for $\overline I$. That is, we successively use Lemma \ref{lemme_product}, the Lipschitz continuity of $Q$ and Lemma \ref{strong_error_lemme2} to obtain
\begin{equation}\label{overlineJ}
\overline J \leq \rho^{- N^*_T} (N^*_T)^2 C_2 h,
\end{equation}

\noindent where $C_2$ is a positive constant independent of $h$.
Then, we derive from the previous estimates \eqref{overlineI} and \eqref{overlineJ} that
\begin{equation*}
|\underline{\tilde R}_T - \tilde R_T|
\leq
\Xi_1(N^*_T) C_3 h,
\end{equation*} 
 
\noindent where $\Xi_1(n) = \left( \rho \frac{\tilde \lambda_\text{min}}{\lambda^*}(1-\frac{\tilde \lambda_\text{max}}{\lambda^*}) \right)^{- n} n (n+1)^2$ and $C_3=\max(C_1,C_2)$. Finally, we have $\E[|\underline{\tilde R}_T - \tilde R_T|^2] \leq C_3 h^2 \E [ \Xi_1(N^*_T)^2 ]$. Since $\E [ \Xi_1(N^*_T)^2 ] < +\infty$ we conclude that $\overline C = O(h^2)$.

\bigskip \noindent \textit{Step 2: estimation of $\overline D$.} Recall that $\tilde x_T=(\tilde \theta_{\tilde N_T}, \Phi_{\tilde \theta_{\tilde N_T}}(T-\tilde T_{\tilde N_T},\tilde \nu_{\tilde N_T}))$ and $\underline{\tilde x}_T=(\tilde \theta_{\tilde N_T}, \overline \Phi_{\tilde \theta_{\tilde N_T}}(T-\tilde T_{\tilde N_T},\underline{\tilde \nu}_{\tilde N_T}))$. Then, using the Lipschitz continuity of $F$, Lemma \ref{strong_error_lemme2} and since $\tilde N_T \leq N^*_T$ we get
\begin{equation*}
|F(\underline{\tilde x}_T) - F(\tilde x_T)|
\leq
L_F e^{L T} (\tilde N_T+1)C h
\leq
L_F e^{L T} (N^*_T+1)C h.
\end{equation*}

\noindent Moreover, $|\underline{\tilde R}_T| \leq \left( \rho \frac{\tilde \lambda_\text{min}}{\lambda^*}(1-\frac{\tilde \lambda_\text{max}}{\lambda^*}) \right)^{-N^*_T}$ so that $\overline D \leq C_4 h^2 \E [\Xi_2(N^*_T)^2 ]$ where $C_4$ is a positive constant independent of $h$ and $\Xi_2(n) = (n+1)\left( \rho \frac{\tilde \lambda_\text{min}}{\lambda^*}(1-\frac{\tilde \lambda_\text{max}}{\lambda^*}) \right)^{-n}$. Since $\E [\Xi_2(N^*_T)^2 ] < +\infty$ we conclude that $\overline D = O(h^2)$.

\qed

\section{Weak error expansion}\label{section_weak_expansion}

In this section we are interested in a weak error expansion for the PDMP $(x_t)$ of section \ref{section_application} and its associated Euler scheme $(\overline x_t)$. First of all, we recall from \cite{davarticle} that the generator $\cali A$ of the process $(t,x_t)$ which acts on functions $g$ defined on $\bb R_+ \times E$ is given by
\begin{equation}\label{PDMP_generator}
\cali A g(t,x) = \gt(t,x) + f(x) \gnu(t,x) + \lambda(x) \int_E ( g(t,y) - g(t,x) ) Q(x,dy),
\end{equation}

\noindent where for notational convenience we have set $\gnu(t,x) := \frac{\partial g}{ \partial \nu}(t,\theta,\nu)$, $\gt(t,x) := \frac{\partial g}{ \partial t}(t,x)$ and $f(x) = f_\theta(\nu)$ for all $x =(\theta,\nu) \in E$. Below, we state the assumptions and the main theorem of this section. Its proof which is inspired by \cite{TaTu90} (see also \cite{pages2018} or \cite{talay}) is delayed in section \ref{errfaible}. 
\begin{hyp}\label{hyp_derivability}
For all $ \theta \in \Theta $ and for all $A \in \cali B(\Theta)$, the functions $\nu \mapsto Q \left( (\theta,\nu), A \right)$, $\nu \mapsto \lambda\left(\theta,\nu \right)$ and $\nu \mapsto f_\theta\left(\nu \right)$ are bounded and twice continuously differentiable with bounded derivatives.
\end{hyp}

\begin{hyp}\label{hyp_u}
The solution $u$ of the integro differential equation 
\begin{equation}\label{integro_diff}
\left\{
\begin{array}{ll}
\cali A u(t,x) = 0, \hspace{0.5cm} (t,x) \in [0,T[ \times E, \\
u(T,x) = F(x), \hspace{0.5cm} x \in E,
\end{array}
\right.
\end{equation}

\noindent with $F : E \rightarrow \bb R$ a bounded function and $\cali A$ given by (\ref{PDMP_generator}) is such that for all $\theta \in \Theta$, the function $(t,\nu)  \mapsto u(t,\theta,\nu)$ is bounded and two times differentiable with bounded derivatives. Moreover the second derivatives of $(t,\nu)  \mapsto u(t,\theta,\nu)$ are uniformly Lipschitz in $\theta$.
\end{hyp}

\begin{theorem}\label{weak_error_thm}
Let $(x_t, t \in [0,T])$ be a PDMP and $(\overline x_t, t \in [0,T])$ its approximation constructed in section \ref{section_application} with $x_0 = \overline x_0 = x$ for some $x \in E$. Under assumptions \ref{hyp_derivability}. and \ref{hyp_u}. for any bounded function $F: E \rightarrow \bb R$ there exists a constant $c_1$ independent of $h$ such that
\begin{equation}\label{weak_expansion}
\E [F(\overline x_T)] - \E [F(x_T)] 
=
h c_1 + O(h^2).
\end{equation}
\end{theorem} 

\begin{Rem}
If $(\tilde x_t)$ is a PDMP whose characteristics $\tilde \lambda, \tilde Q$ satisfy the assumptions of Proposition \ref{DECOMP2} and $(\underline{\tilde x}_t)$ is its approximation we deduce from Theorem \ref{weak_error_thm} that 
\begin{equation}\label{weak_tilde}
\E[F(\underline{\tilde x}_T) \underline{\tilde R}_T] - \E [F(\tilde x_T) \tilde R_T]= h c_1 + O(h^2).
\end{equation} 
\end{Rem}

\subsection{Further results on PDMPs: Itô and Feynman-Kac formulas}\label{Ito_FK_section}



\begin{defi}\label{def_generator_PDMP}
Let us define the following operators which act on functions $g$ defined on $\bb R_+ \times E$.
\begin{equation*}
\cali T g(t,x)
:=
\gt(t,x) + f(x) \gnu(t,x),
\end{equation*}
\begin{equation*}
\cali S g(t,x)
:=
\lambda(x) \int_E ( g(t,y) - g(t,x) ) Q(x,dy).
\end{equation*}
\end{defi}

\noindent From Definition \ref{def_generator_PDMP}, the generator $\cali A$ defined by \eqref{PDMP_generator} reads $\cali A g(t,x) = \cali T g(t,x) + \cali S g(t,x)$. We introduce the random counting measure $p$ associated to the PDMP $(x_t)$ defined by $p([0,t] \times A) := \sum_{n \geq 1} \ind{T_n \leq t} \ind{Y_n \in A}$ for $t \in [0,T]$ and for $A \in \cali B(E)$. The compensator of $p$, noted $ p'$, is given from \cite{davarticle} by 
\begin{equation*}
p'( [0,t] \times A ) = \int_0^t \lambda(x_s) Q(x_s,A) ds.
\end{equation*}

\noindent Hence, $q := p - p'$ is a martingale with respect to the filtration generated by $p$ noted $(\cali F_t^p)_{t \in [0,T]}$. Similarly, we introduce $\overline p$, $\overline p'$, $\overline q$ and $(\cali F_t^{\overline p})_{t \in [0,T]}$ to be the same objects as above but corresponding to the approximation $(\overline x_t)$. The fact that $\overline p'$ is the compensator of $\overline p$ and that $\overline q$ is a martingale derives from arguments of the marked point processes theory, see \cite{Bremaud}.

\begin{defi}\label{def_generators_approx}
Let us define the following operators which act on functions $g$ defined on $\bb R_+ \times E$.
\begin{equation*}
\overline{\cali T} g(t,x,y)
:=
\gt(t,x) + f(y) \gnu(t,x),
\end{equation*}
\begin{equation*}
\overline{\cali A} g(t,x,y)
:=
\overline{\cali T} g(t,x,y) + \cali S g(t,x).
\end{equation*}
\end{defi}

\begin{Rem}
For all functions $g$ defined on $\bb R_+ \times E$, $\overline{\cali T} g(t,x,x) = \cali T g(t,x)$, so that $\overline{\cali A} g(t,x,x) = \cali A g(t,x)$.
\end{Rem}

\noindent The next theorem provides Itô formulas for the PDMP $(x_t)$ and its approximation $(\overline x_t)$. For all $s \in [0,T]$, we set $\overline \eta(s) := \overline T_n + kh$ if $s \in [\overline T_n + kh, (\overline T_n + (k+1)h) \smin \overline T_{n+1}[$ for some $n \geq 0$ and for some $k \in \{0,\ldots, \lfloor (\overline T_{n+1} - \overline T_n)/h \rfloor\}$.
\begin{theorem}\label{Ito_formulas}
Let $(x_t, t \in [0,T])$ and $(\overline x_t, t \in [0,T])$ be a PDMP and its approximation respectively constructed in section \ref{section_application} with $x_0 = \overline x_0 = x$ for some $x \in E$. For all bounded functions $g : \bb R_+ \times E \rightarrow \bb R$ continuously differentiable with bounded derivatives, we have 
\begin{equation}\label{Ito_PDMP}
g(t,x_t) = g(0,x) + \int_0^t  \cali A g (s,x_s) ds + M^g_t,
\end{equation} 

\noindent where $M^g_t := \int_0^t \int_E ( g(s,y) - g(s,x_{s-}) ) q(dsdy)$ is a true $\cali F_t^p$-martingale, and 
\begin{equation}\label{Ito_approx}
g(t,\overline x_t) = g(0,x) + \int_{0}^{t} \overline{\cali A} g(s,\overline x_s,\overline x_{\overline \eta(s)}) ds + \overline M^g_t,
\end{equation}
where, $\overline M^g_t := \int_0^t \int_E ( g(s,y) - g(s,\overline x_{s-}) ) \overline q(dsdy)$ is a true $\cali F_t^{\overline p}$-martingale.
\end{theorem}

\begin{proof}[Proof of Theorem \ref{Ito_formulas}]
The proof of \eqref{Ito_PDMP} is given in \cite{davarticle}. We prove \eqref{Ito_approx} following the same arguments. Since $\overline q = \overline p - \overline p'$, we have
\begin{equation*}
\overline M^g_t
=
\sum_{k \geq 1 } \ind{\overline T_k \leq t} \left( g(\overline T_k, \overline x_{\overline T_k}) - g(\overline T_k, \overline x_{\overline T_k}^-)  \right) - \int_0^t \cali S g(s,\overline x_s)ds.
\end{equation*}

\noindent Consider the above sum. As in \cite{davarticle}, we write, on the event $\{\overline N_t = n\}$, that 
\begin{equation*}
\begin{split}
&\sum_{k \geq 1 } \ind{\overline T_k \leq t} \left( g(\overline T_k, \overline x_{\overline T_k}) - g(\overline T_k, \overline x_{\overline T_k}^-)  \right)
\\&=
g(t,\overline x_t) - g(0,x) - \left[ g(t,\overline x_t) - g(\overline T_n, \overline x_{\overline T_n}) + \sum_{k=0}^{n-1} g(\overline T_{k+1}, \overline x_{\overline T_{k+1}}^-) - g(\overline T_{k}, \overline x_{\overline T_{k}}) \right].
\end{split}
\end{equation*}

\noindent For all $k \leq n-1$, we decompose the increment $g(\overline T_{k+1},\overline x_{\overline T_{k+1}}^-) - g(\overline T_k,\overline x_{\overline T_k})$ as a sum of increments on the intervals $[\overline T_k + ih, (\overline T_k + (i+1)h)\smin \overline T_{k+1} ] \subset [\overline T_k, \overline T_{k+1} ]$. Without loss of generality we are led to consider increments of the form $g(t,\theta,\overline \phi_\theta(t,\nu)) - g(ih,\theta,\overline y_i(x))$ for some $i \geq 0$, $t \in [ih,(i+1)h]$ and for all $x = (\theta,\nu) \in E$ where we recall that $\overline \phi$ is defined by \eqref{continuous_euler_scheme}. The function $g$ is smooth enough to write
\begin{equation*}
g(t,\theta,\overline \phi_\theta(t,\nu)) - g(ih,\theta,\overline y_i(x))
=
\int_{ih}^{t} \left( \gt + f_{\theta}(\overline y_i(x)) \gx \right)(s,\theta,\overline \phi_{\theta}(s, \nu) ) ds.
\end{equation*}

\noindent Then, the above arguments together with definition \ref{def_generators_approx} yields
\begin{equation*}
g(t,\overline x_t) - g(\overline T_{n},\overline x_{\overline T_{n}}) + \sum_{k=0}^{n-1} g(\overline T_{k+1},\overline x_{\overline T_{k+1}}^-) - g(\overline T_k,\overline x_{\overline T_k})
=
\int_{0}^t \overline{\cali T} g (s,\overline x_s,\overline x_{\overline \eta(s)}) ds.
\end{equation*}
\end{proof}

\noindent The following theorem gives us a way to represent the solution of the integro-differential equation \eqref{integro_diff} as the conditional expected value of a functional of the terminal value of the PDMP $(x_t)$. It plays a key role in the proof of Theorem \ref{weak_error_thm}.

\begin{theorem}[PDMP's Feynman-Kac formula \cite{dav}]\label{FK_PDMP}
Let $F : E \rightarrow \bb R$ be a bounded function. Then the integro-differential equation 
\eqref{integro_diff} has a unique solution $u : \bb R_+ \times E \rightarrow \bb R$ given by
\begin{equation*}
u(t,x) = \E [ F(x_T) | x_t = x], \hspace{0.5cm} (t,x) \in [0,T] \times E.
\end{equation*}
\end{theorem}

\subsection{Proof of Theorem \ref{weak_error_thm}}\label{errfaible}

We provide a proof in two steps. First, we give an appropriate representation of the weak error $\E [F (\overline x_T)] - \E [F (x_T)]$. Then, we use this representation to identify the coefficient $c_1$ in \eqref{weak_expansion}.

\bigskip \noindent \textit{Step 1: Representing $\E [F (\overline x_T)] - \E [F (x_T)]$.} Let $u$ denote the solution of \eqref{integro_diff}. From Theorem \ref{FK_PDMP} we can write $\E [F (\overline x_T)] - \E [F (x_T)]=\E [u(T,\overline x_T)] - u(0,x)$. Then, the application of the Itô formula \eqref{Ito_approx} to $u$ at time $T$ yields
\begin{equation*}
u(T,\overline x_T) = u(0,x) + \int_{0}^{T} \overline{\cali A} u (s,\overline x_s,\overline x_{\overline \eta(s)}) ds + \overline M^u_T.
\end{equation*}

\noindent Since $(\overline M^u_t)$ is a true martingale, we obtain
\begin{equation*}
\E [u(T,\overline x_T) - u(0,x)] = \E \left[ \int_{0}^{T}  \overline{\cali A} u (s,\overline x_s,\overline x_{\overline \eta(s)}) ds \right].
\end{equation*}

\noindent For $s \in [0,T]$ we have 
$\overline{\cali A} u(s,\overline x_s,\overline x_{\overline \eta(s)})
=
\ut(s,\overline x_s) + f(\overline x_{\overline \eta(s)})\ux(s,\overline x_s) + \cali S u(s,\overline x_s)$ (see Definition \ref{def_generators_approx}).
From the regularity of $\lambda$, $Q$ and $u$ (see assumptions \ref{hyp_derivability} and \ref{hyp_u}), the functions $\ut$, $\ux$ and $\cali S u$ are smooth enough to apply the Itô formula \eqref{Ito_approx} between $\overline \eta(s)$ and $s$ respectively. This yields
\begin{equation*}
\ut(s,\overline x_s) = \ut(\overline \eta(s),\overline x_{\overline \eta(s)}) + \int_{\overline \eta(s)}^{s} \overline{\cali A} (\ut)(r,\overline x_r,\overline x_{\overline \eta(r)}) dr + \overline M^{\ut}_s - \overline M^{\ut}_{\overline \eta(s)},
\end{equation*}

\begin{equation*}
\ux(s,\overline x_s) = \ux(\overline \eta(s),\overline x_{\overline \eta(s)}) + \int_{\overline \eta(s)}^{s} \overline{\cali A} (\ux)(r,\overline x_r,\overline x_{\overline \eta(r)}) dr + \overline M^{\ux}_s - \overline M^{\ux}_{\overline \eta(s)},
\end{equation*}

\begin{equation*}
\cali S u(s,\overline x_s) = \cali S u(\overline \eta(s),\overline x_{\overline \eta(s)}) + \int_{\overline \eta(s)}^{s} \overline{\cali A} (\cali S u)(r,\overline x_r,\overline x_{\overline \eta(r)}) ds + \overline M^{\cali S u}_s - \overline M^{\cali S u}_{\overline \eta(s)}.
\end{equation*}

\noindent Moreover, since $\overline \eta(r) = \overline \eta(s)$ for $r \in [\overline \eta(s),s]$, we have
\begin{align*}
f(\overline x_{\overline \eta(s)})\ux(s,\overline x_s) 
&= 
f(\overline x_{\overline \eta(s)})\ux(\overline \eta(s),\overline x_{\overline \eta(s)}) 
\\&+ 
\int_{\overline \eta(s)}^{s} f(\overline x_{\overline \eta(r)})\overline{\cali A} (\ux)(r,\overline x_r,\overline x_{\overline \eta(r)}) dr + f(\overline x_{\overline \eta(s)})(\overline M^{\ux}_s - \overline M^{\ux}_{\overline \eta(s)}),
\end{align*}

\noindent so that
\begin{align*}
\overline{\cali A} u(s,\overline x_s,\overline x_{\overline \eta(s)})
&=
\overline{\cali A}u(\overline \eta(s),\overline x_{\overline \eta(s)},\overline x_{\overline \eta(s)}) + \int_{\overline \eta(s)}^s \Upsilon(r,\overline x_r,\overline x_{\overline \eta(r)})dr
\\&+
\overline M^{\ut}_s - \overline M^{\ut}_{\overline \eta(s)} + f(\overline x_{\overline \eta(s)})(\overline M^{\ux}_s - \overline M^{\ux}_{\overline \eta(s)}) + \overline M^{\cali S u}_s - \overline M^{\cali S u}_{\overline \eta(s)},
\end{align*}

\noindent where, 
\begin{equation}\label{Upsilon}
\Upsilon(t,x,y) := \left( \overline{\cali A} (\ut) + f(y) \overline{\cali A} (\ux) + \overline{\cali A} (\cali S u) \right)(t,x,y).
\end{equation}

\noindent Since $\overline{\cali A}u(t,x,x) = \cali Au(t,x)$, the first term in the above equality is 0 by Theorem \ref{FK_PDMP}. By using Fubini's theorem and the fact that $(\overline M^{\ut}_t)$ and $(\overline M^{\cali S u}_t)$ are true martingales, we obtain
\begin{equation*}
\E \left[ \int_0^T \overline M^{\ut}_s - \overline M^{\ut}_{\overline \eta(s)} ds \right]
=
\E \left[ \int_0^T \overline M^{\cali S u}_s - \overline M^{\cali S u}_{\overline \eta(s)} ds \right]
=
0.
\end{equation*}

\noindent Moreover, since $(\overline M^{\ux}_t)$ is a $\cali F^{\overline p}_{t}$-martingale, we have
\begin{align*}
\E \left[ \int_0^T f(\overline x_{\overline \eta(s)})(\overline M^{\ux}_s - \overline M^{\ux}_{\overline \eta(s)}) ds \right]
=
\int_0^T
\E \left[ f(\overline x_{\overline \eta(s)}) \E[ \overline M^{\ux}_s - \overline M^{\ux}_{\overline \eta(s)} | \cali F^{\overline p}_{\overline \eta(s)} ] \right]
ds
=
0.
\end{align*}
 
\noindent Collecting the previous results, we obtain 
$
\E [F(\overline x_T)] - \E [F(x_T)]
=
\E \left[ \int_0^T \int_{\overline \eta(s)}^s \Upsilon(r,\overline x_{r},\overline x_{\overline \eta(r)}) dr ds \right].
$
We can compute an explicit form of $\Upsilon$ in term of $u$, $f$, $\lambda$, $Q$ and their derivatives. Indeed, $\Upsilon$ is given by \eqref{Upsilon}, and we have
\begin{equation*}
\begin{split}
&\overline{\cali A}(\ut)(t,x,y) 
=
\utt(t,x) + f(y)\utx(t,x) + \cali S (\ut)(t,x),
\\\vspace{0.2cm}&
\left( f\overline{\cali A}(\ux) \right)(t,x,y)
=
f(y) \left( \utx(t,x) + f(y)\uxx(t,x) + \cali S(\ux)(t,x) \right),
\\\vspace{0.2cm}&
\overline{\cali A}(\cali S u)(t,x,y)
=
\partial_t (\cali S u)(t,x) + f(y) \partial_\nu (\cali S u)(t,x) + \cali S(\cali S u)(t,x).  
\end{split}
\end{equation*}

\noindent The application of the Taylor formula to the functions $\utt$, $\utx$, $\uxx$, $\cali S (\ut)$, $\cali S(\ux)$, $\partial_t (\cali S u)$, $\partial_\nu (\cali S u)$ and $\cali S(\cali S u)$ at the order 0 around $(\overline \eta(r),\overline x_{\overline \eta(r)})$ yields
$
\Upsilon(r,\overline x_r,\overline x_{\overline \eta(r)})
=
\Upsilon(\overline \eta(r),\overline x_{\overline \eta(r)},\overline x_{\overline \eta(r)}) + O(h).
$
Setting $\Psi(t,x) = \Upsilon(t,x,x) $ and recalling that for $r \in [\overline \eta(s),s]$, $\overline \eta(r) = \overline \eta(s)$ and that $|s - \overline \eta(s)| \leq h$, we obtain
\begin{equation*}
\E [F(\overline x_T)] - \E [F(x_T)]
=
\E \left[ \int_0^T (s - \overline \eta(s) )\Psi(\overline \eta(s),\overline x_{\overline \eta(s)}) ds \right] + O(h^2).
\end{equation*}

\noindent Consider the expectation in the right-hand side of the above equality. We decompose the integral into a (finite) sum of integrals on the intervals $[\overline T_n + kh, (\overline T_n + (k+1)h)\smin \overline T_{n+1} ]$ where $\Psi$ is constant. Without loss of generality, we are led to consider integrals of the form $\int_{kh}^t (s-kh)C ds $ for some $k \geq 0$, $t \in [kh,(k+1)h]$ and $C$ a bounded constant. We have $\int_{kh}^t (s-kh)C ds = \frac{t - kh}{2} \int_{kh}^tC ds $ moreover adding and subtracting $h$ in the numerator of $(t - kh)/2$ yields
\begin{equation*}
\int_{kh}^t (s-kh)C ds
=
\frac{h}{2} \int_{kh}^tC ds + \frac{t - (k+1)h}{2} \int_{kh}^tC ds.
\end{equation*}

\noindent Since $C$ is bounded we deduce that $\int_{kh}^t (s-kh)C ds = \frac{h}{2} \int_{kh}^tC ds + O(h^2)$. Since $\Psi$ is assumed bounded and $\E[\overline N_T] < +\infty$, the above arguments yields the following representation
\begin{equation}\label{representation_weak}
\E [F(\overline x_T)] - \E [F(x_T)]
=
\frac{h}{2}\E \left[ \int_0^T \Psi(\overline \eta(s),\overline x_{\overline \eta(s)}) ds \right] + O(h^2).
\end{equation} 

\bigskip \noindent \textit{Step 2: From the representation \eqref{representation_weak} to the expansion at the order one.} In this step, we show that $\E \left[ \int_0^T \Psi(\overline \eta(s),\overline x_{\overline \eta(s)}) ds \right]=\E \left[ \int_0^T \Psi(s,x_{s}) ds \right]+O(h)$.
First, we introduce the random variables $\overline \Gamma$ and $\Gamma$ defined by $\overline \Gamma :=\int_{0}^{ T}  \Psi (\overline \eta(s),\overline x_{\overline \eta(s)}) ds $ and $\Gamma :=\int_{0}^{ T}  \Psi (\overline \eta(s),x_{\overline \eta(s)}) ds $ and write
\begin{equation*}
\E [ |\overline \Gamma - \Gamma| ]
=
\E \left[\ind{\min(T_{\overline \tau^\dagger},\overline T_{\overline \tau^\dagger}) \leq T} |\overline \Gamma - \Gamma|  \right] 
+ 
\E \left[\ind{\min(T_{\overline \tau^\dagger},\overline T_{\overline \tau^\dagger}) > T} |\overline \Gamma - \Gamma|  \right],
\end{equation*} 

\noindent where $\overline \tau^\dagger$ is defined in Definition \ref{tau*}. Since $\Psi$ is bounded and $\bb P(\min(T_{\overline \tau^\dagger},\overline T_{\overline \tau^\dagger}) \leq T ) = O(h) $ (see the proof of Theorem \ref{thm_strong_cv}), we have 
$
\E \left[ |\overline \Gamma - \Gamma| \ind{\min(T_{\overline \tau^\dagger},\overline T_{\overline \tau^\dagger}) \leq T} \right]
=O(h).
$
Now, recall from \eqref{min_tau_dag} that, on the event $\{\min(T_{\overline \tau^\dagger},\overline T_{\overline \tau^\dagger}) > T\}$, we have $T_k = \overline T_k$ and $\theta_k = \overline \theta_k$ for all $k \geq 1$ such that $T_k \in [0,T]$. Thus, for all $n \leq \overline N_T$ and for all $s \in [\overline T_n,\overline T_{n+1}[$ we have $\overline x_{\overline \eta(s)} = (\overline \theta_n, \overline \phi_{\overline \theta_n}(\overline \eta(s)-\overline T_n,\overline \nu_n))$ and $ x_{\overline \eta(s)} = (\overline \theta_n,  \phi_{\overline \theta_n}(\overline \eta(s)-\overline T_n, \nu_n))$. Consequently, on the event $\{\min(T_{\overline \tau^\dagger},\overline T_{\overline \tau^\dagger}) > T \}$ we have
\begin{align*}
|\overline \Gamma - \Gamma|
\leq
\sum_{n=0}^{\overline N_T} \int_{\overline T_n}^{\overline T_{n+1} \smin T} | \Psi (\overline \eta(s),\overline \theta_n,\overline \phi_{\overline \theta_n}(\overline \eta(s)-\overline T_n,\overline \nu_n)) -  \Psi (\overline \eta(s),\overline \theta_n, \phi_{\overline \theta_n}(\overline \eta(s)-\overline T_n, \nu_n)) | ds.
\end{align*}

\noindent From the regularity assumptions \ref{hyp_derivability} and \ref{hyp_u}, the function $\nu \mapsto \Psi(t,\theta,\nu)$ is uniformly Lipschitz in $(t,\theta)$ with constant $L_\Psi$ as sum and product of bounded Lipschitz functions. Thus, from this Lipschitz property and the application of Lemma \ref{strong_error_lemme2}, we get
\begin{equation*}
| \Psi (\overline \eta(s),\overline \theta_n,\overline \phi_{\overline \theta_n}(\overline \eta(s)-\overline T_n,\overline \nu_n)) -  \Psi (\overline \eta(s),\overline \theta_n, \phi_{\overline \theta_n}(\overline \eta(s)-\overline T_n, \nu_n)) |
\leq
L_\Psi C e^{LT} (n+1) h.
\end{equation*}

\noindent From the above inequality, we find that $\E \left[ \ind{\min(T_{\overline \tau^\dagger},\overline T_{\overline \tau^\dagger}) > T} |\overline \Gamma - \Gamma| \right] \leq L_\Psi C e^{LT} T h \E [\overline N_T (\overline N_T+1)]$. Since $\overline N_T \leq N^*_T$ and $\E [ N^*_T (N^*_T+1)] < +\infty$ we conclude that $\E \left[ \ind{\min(T_{\overline \tau^\dagger},\overline T_{\overline \tau^\dagger}) > T} |\overline \Gamma - \Gamma| \right] = O(h)$. We have shown that $\E \left[ \int_0^T \Psi(\overline \eta(s),\overline x_{\overline \eta(s)}) ds \right]=\E \left[ \int_0^T \Psi(\overline \eta(s), x_{\overline \eta(s)}) ds \right]+O(h)$.
Secondly, from the regularity assumptions \ref{hyp_derivability} and \ref{hyp_u}, the function $(t,\nu) \mapsto \Psi(t,\theta,\nu)$ is uniformly Lipschitz in $\theta$. Moreover, for all $s \in [0,T]$ there exits $k \geq 0$ such that both $s$ and $\overline \eta(s)$ belong to the same interval $[\overline T_k,\overline T_{k+1}[$ so that $x_s = (\theta_k,\phi_{\theta_k}(s-\overline T_k,\nu_k))$ and $x_{\overline \eta(s)} = (\theta_k,\phi_{\theta_k}(\overline \eta(s)-\overline T_k,\nu_k))$. Thus, from the Lipschitz continuity of $\Psi$, from the fact that $|s-\overline \eta(s)| \leq h$ and since $f_\theta$ is uniformly bounded in $\theta$ we have $|\Psi(s,x_s) - \Psi (\overline \eta(s),x_{\overline \eta(s)})| \leq C h$ where $C$ is a constant independent of $h$. Then, we obtain
$
\sup_{s \in [0,T]} | \E [\Psi(s,x_s)] - \E [\Psi (\overline \eta(s),x_{\overline \eta(s)}) ] |
\leq C h
$
from which we deduce that
$
\left| \E \left[  \int_{0}^T  \Psi (\overline \eta(s),x_{\overline \eta(s)}) ds \right]
-
\E \left[  \int_{0}^T  \Psi (s,x_s) ds \right] \right|
\leq CT h.
$
Finally, the weak error expansion reads
\begin{equation*}
\E [F (\overline x_T)] - \E [F (x_T)]
=
\frac{h}{2} \E \left[  \int_{0}^T  \Psi (s,x_s) ds \right] +O(h^2).
\end{equation*}

\qed

\section{Numerical experiment}\label{section_numerical_res}

In this section, we use the theoretical results above to apply the MLMC method to the PDMP 2-dimensional Morris-Lecar (shortened PDMP 2d-ML).

\subsection{The PDMP 2-dimensional Morris-Lecar}\label{ML_model}

The deterministic Morris-Lecar model has been introduced in 1981 by Catherine Morris and Harold Lecar in \cite{lecar} to explain the dynamics of the barnacle muscle fiber. This model belongs to the family of conductance-based models (just as the Hodgkin-Huxley model \cite{hodgkin}) and takes the following form
\begin{equation}\label{deter_ML}
\left\{
\begin{array}{ll}
\frac{dv}{dt} = \frac{1}{C} \Big{(} I - g_{\text{Leak}} (v - V_{\text{Leak}}) - g_{\text{Ca}} M_{\infty}(v) (v - V_{\text{Ca}}) - g_{\text{K}} n (v - V_{\text{K}}) \Big{)},\\
\frac{dn}{dt} = (1 - n) \alpha_\text{K}(v) - n \beta_\text{K}(v),
\end{array}
\right.
\end{equation}

\noindent where $M_{\infty}(v) = (1 + \tanh[ (v-V_1 )/ V_2 ] )/2$,  $\alpha_\text{K}(v) = \lambda_\text{K}(v) N_{\infty}(v)$, $\beta_\text{K}(v) = \lambda_\text{K}(v) (1- N_{\infty}(v))$, $N_{\infty}(v) = (1 + \tanh[ (v-V_3 )/ V_4 ] )/2$, $\lambda_\text{K}(v) = \overline \lambda_\text{K} \cosh((v-V_3 )/ 2V_4 )$.

\bigskip \noindent In this section we consider the PDMP version of \eqref{deter_ML} that we denote by $(x_t, t \in [0,T])$, $T > 0$, whose characteristics $(f,\lambda,Q)$ are given by
\begin{itemize}
\item $f(\theta,\nu) = \frac{1}{C} \Big{(} I - g_{\text{Leak}} (\nu - V_{\text{Leak}}) - g_{\text{Ca}} M_{\infty}(\nu) (\nu - V_{\text{Ca}}) - g_{\text{K}} \frac{\theta}{N_\text{K}} (\nu - V_{\text{K}}) \Big{)}$,

\item $\lambda(\theta,\nu) = (N_\text{K} - \theta) \alpha_\text{K}(\nu) + \theta \beta_\text{K}(\nu)$,

\item $Q\Big{(} (\theta,\nu), \{\theta +1\} \Big{)} = \frac{(N_\text{K} - \theta) \alpha_\text{K}(\nu)}{\lambda(\theta,\nu)}$, \hspace{0.3cm} $Q\Big{(} (\theta,\nu), \{\theta - 1\} \Big{)} = \frac{\theta \beta_\text{K}(\nu)}{\lambda(\theta,\nu)}$.
\end{itemize}

\noindent The state space of the model is $E = \{0,\ldots,N_\text{K}\} \times \bb R$ where $N_\text{K} \geq 1$ stands for the number of potassium gates.
The values of the parameters used in the simulations are $V_1 = -1.2$ , $V_2 = 18$, $V_3 = 2$, $V_4 = 30$, $\overline \lambda_\text{K} = 0.04$, $C = 20$, $g_{\text{Leak}} = 2$, $V_{\text{Leak}} = -60$, $g_{\text{Ca}} = 4.4$, $V_{\text{Ca}} = 120$, $g_{\text{K}} = 8$, $V_{\text{K}} = -84$, $I = 60$, $N_{\text{K}} = 100$.

\begin{figure}[h]
  \begin{subfigure}{.5\linewidth}
    \resizebox{\columnwidth}{!}{\input{V_ML2d}}
    \caption{Membrane potential as a function of time. Red curves: stochastic potential. Black curve: deterministic potential.}
  \end{subfigure} 
  \begin{subfigure}{.5\linewidth}
    \resizebox{\columnwidth}{!}{\input{theta_ML2d}}
    \caption{Proportion of opened gates as a function of time. Red curves: stochastic gates ($\theta/N_{\text{K}}$). Black curve: deterministic gates ($n$).}
  \end{subfigure}
    \begin{subfigure}{.5\linewidth}
    \resizebox{\columnwidth}{!}{\input{kernel_ML2d}}
    \caption{Probability of opening a gate ($Q(x_t, \{\theta_t+1\})$) as a function of time.}
  \end{subfigure} 
  \begin{subfigure}{.5\linewidth}
    \resizebox{\columnwidth}{!}{\input{lambda_ML2d}}
            \caption{Jump rate ($\lambda(x_t)$) as a function of $t$.}
  \end{subfigure}
\caption{10 trajectories of the characteristics of the PDMP 2d-ML on $[0,100]$.}
\label{fig1}
\end{figure}
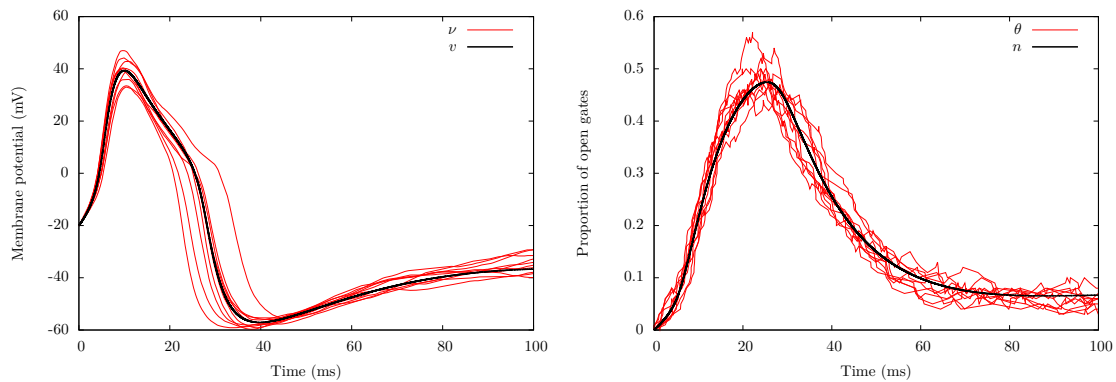
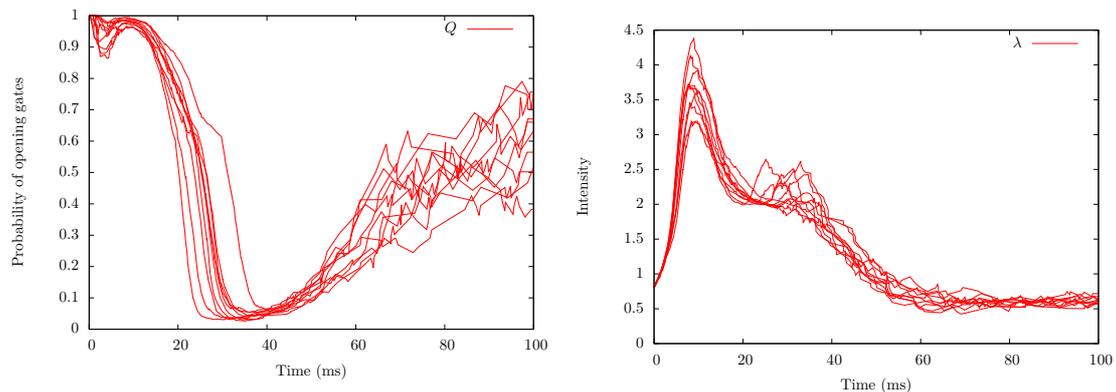

\subsection{Classical and Multilevel Monte Carlo estimators}\label{MC_MLMC_section}

In this section we introduce the classical and multilevel Monte Carlo estimators in order to estimate the quantity $\E \left[ F(x_T) \right]$  where $(x_t, t \in [0,T])$ is the PDMP 2d-ML and $F(\theta,\nu) = \nu$ for $(\theta,\nu) \in E$ so that $F(x_T)$ gives the value of the membrane potential at time $T$. Note that other possible choices are $F(\theta,\nu) = \nu^n$ or $F(\theta,\nu) = \theta^n$ for some $n \geq 2$. In those cases, the quantity $\E \left[ F(x_T) \right]$ gives the moments of the membrane potential or the number of open gates at time $T$ so that we can compute statistics on these biological variables. 

\noindent Let $X := F(x_T)$. In the sequel it will be convenient to emphasize the dependence of the Euler scheme $(\overline x_t)$ on a time step $h$. We introduce a family of random variables $(X_h, h > 0)$ defined by $X_h := F(\overline x_T)$ where for a given $h>0$ the corresponding PDP $(\overline x_t)$ is constructed as in section \ref{section_application} with time step $h$. In particular, the processes $(\overline x_t)$ for $h >0$ are correlated through the same randomness $(U_k)$, $(V_k)$ and $(N^*_t)$. We build a classical Monte Carlo estimator of $\E[X]$ based on the family $(X_h, h > 0)$ as follows 
\begin{equation}\label{crude_MC_estimator}
Y^{\text{MC}} = \frac{1}{N} \sum_{k=1}^N X_h^{k},
\end{equation}

\noindent where $(X_h^{k},k \geq 1)$ is an i.i.d sequence of random variables distributed like $X_h$. The parameters $h>0$ and $N \in \bb N$ have to be determined. We build a multilevel Monte Carlo estimator based on the family $(X_h, h > 0)$ as follows
\begin{equation}\label{MLMC_estimator}
Y^{\text{MLMC}} = \frac{1}{N_1} \sum_{k=1}^{N_1} X_{h^*}^{k} + \sum_{l=2}^L \frac{1}{N_l} \sum_{k=1}^{N_l} (X_{h_l}^{k} - X_{h_{l-1}}^{k}),
\end{equation}
 
\noindent where $\left( (X_{h_l}^{k}, X_{h_{l-1}}^{k}), k \geq 1 \right)$ for $l=2,\ldots,L$ are independent sequences of independent copies of the couple $(X_{h_l}, X_{h_{l-1}})$ and independent of the i.i.d sequence $(X_{h^*}^{k}, k \geq 1)$. The parameter $h^*$ is a free parameter that we fix in section \ref{section_numerical_results}. The parameters $L \geq 2$, $M \geq 2$, $N \geq 1$ and $q = (q_1, \ldots, q_L) \in ]0,1[^L$ with $\sum_{l=1}^L q_l = 1$ have to be determined, then we set $N_l := \lceil N q_l\rceil$, $h_l := h^* M^{-(l-1)}$.

\smallskip \noindent We also set $\tilde X := F(\tilde x_T) \tilde R_T $ where $\tilde R_T$ is defined as in Proposition \ref{DECOMP2} with an intensity $\tilde \lambda$ and a kernel $\tilde Q$ that will be specified in section \ref{section_numerical_results} and let $(\tilde X_h, h > 0)$ be such that $\tilde X_h := F(\underline{\tilde x}_T) \underline{\tilde R}_T$ for all $h >0$ .
By Proposition \ref{DECOMP2}, we have $\E[X] = \E[\tilde X]$ and $\E[X_h] = \E[\tilde X_h]$ for $h >0$. Consequently, we build likewise a multilevel estimator $\tilde Y^{\text{MLMC}}$ based on the family $(\tilde X_h, h > 0)$. 

\smallskip \noindent The complexity of the classical Monte Carlo estimator $Y^{\text{MC}}$ depends on the parameters $(h,N)$ and the one of the multilevel estimators $Y^{\text{MLMC}}$ and $\tilde Y^{\text{MLMC}}$ depends on $(L,q,N)$. In order to compare those estimators we proceed as in \cite{LP} (see also \cite{pages2018}), that is to say, for each estimator we determine the parameters which minimize the global complexity (or cost) subject to the constraint that the resulting L$^2$-error must be lower than a prescribed $\epsilon >0$. 

\noindent As in \cite{LP}, we call $V_1$, $c_1$, $\alpha$, $\beta$ and $\text{Var}(X)$ the structural parameters associated to the family $(X_h,h > 0)$ and $X$. We know theoretically from Theorem \ref{thm_strong_cv} (strong estimate) and Theorem \ref{weak_error_thm} (weak expansion) that $(\alpha,\beta) = (1,1)$ whereas $V_1$, $c_1$ and $\text{Var}(X)$ are not explicit (we explain how we estimate them in section \ref{section_methodo}). Moreover, the structural parameters $\tilde V_1$, $\tilde c_1$, $\tilde \alpha$, $\tilde \beta$ and $\text{Var}(\tilde X)$ associated to $(\tilde X_h, h > 0)$ and $\tilde X$ are such that $\tilde \alpha = \alpha$, $\tilde c_1 = c_1$ (see \eqref{weak_tilde}), $\tilde \beta = 2$ (see Theorem \ref{thm_strong_cv_tilde}) and $\tilde V_1$, $\text{Var}(\tilde X)$ are not explicit.

\noindent The classical and the multilevel estimators defined above are linear and of Monte Carlo type in the sense described in \cite{LP}. The optimal parameters of those estimators are then expressed in term of the corresponding structural parameters as follows (see \cite{LP} or \cite{pages2018}).
For a user prescribed $\epsilon > 0$, the classical Monte Carlo parameters $h$ and $N$ are
\begin{equation}\label{MC_parameters}
h(\epsilon) = (1 + 2\alpha)^{\frac{-1}{2\alpha}} \left( \frac{\epsilon}{|c_1|} \right)^{\frac{1}{\alpha}}, 
\hspace{0,5cm}
N(\epsilon) = \left( 1 + \frac{1}{2\alpha} \right) \frac{\text{Var}(X) \left( 1 + \rho h^{\beta/2}(\epsilon) \right)^2}{\epsilon^2},
\end{equation}

\noindent where $\rho = \sqrt{V_1 / \text{Var}(X)}$. The parameters of the estimator $Y^{\text{MLMC}}$ are given in Table \ref{MLMC_parameters}  where $n_l := M^{l-1}$ for $l = 1, \ldots, L$ with the convention $n_0 = n_0^{-1} = 0$. The parameters of $\tilde Y^{\text{MLMC}}$ are given in a similar way using $\tilde V_1$, $\tilde \beta$ and $\text{Var}(\tilde X)$. Finally, the parameter $M(\epsilon)$ is determined as in \cite{LP} section 5.1.

{\renewcommand{\arraystretch}{3}
{\setlength{\tabcolsep}{0.5cm}
\begin{table}
\begin{tabular}{c|c} 
$L$ & $\left\lceil 1 + \frac{\log(|c_1|^{\frac{1}{\alpha}}  h^*)}{\log(M)} + \frac{\log(A/\epsilon)}{\alpha \log(M)} \right\rceil$, \hspace{0.2cm} $A = \sqrt{1+2\alpha}$ \\
\hline
$q$ & 
\begin{tabular}{c} 
$q_1 = \mu^* (1 + \rho (h^*)^{\frac{\beta}{2}})$ \\ 
$q_j = \mu^* \rho (h^*)^{\frac{\beta}{2}} \left( \frac{n_{j-1}^{\frac{-\beta}{2}}+n_j^{\frac{-\beta}{2}}}{\sqrt{n_{j-1}+n_j}} \right), j=2,\ldots,L; \mu^* = 1/\sum_{1 \leq j \leq L} q_j$
\end{tabular} \\
\hline
$N$ & $\left(1+\frac{1}{2 \alpha} \right) \frac{ \text{Var}(X) \left( 1 + \rho (h^*)^{\frac{\beta}{2}} \sum_{j=1}^L \left( n_{j-1}^{\frac{-\beta}{2}}+n_j^{\frac{-\beta}{2}} \right) \sqrt{n_{j-1}+n_j} \right)^2 }{\epsilon^2 \sum_{j=1}^L q_j (n_{j-1}+n_j)}$ \\
\end{tabular}  
\caption{Optimal parameters for the MLMC estimator \eqref{MLMC_estimator}.}
\label{MLMC_parameters} 
\end{table}}}

\subsection{Methodology}\label{section_methodo}

We compare the classical and the multilevel Monte Carlo estimators in term of \textit{precision}, \textit{CPU-time} and \textit{complexity}. 
The \textit{precision} of an estimator $Y$ is defined by the L$^2$-error $\parallel Y - \E[X] \parallel_2 = \sqrt{ (\E[Y] - \E[X])^2 + \text{Var}(Y) }$ also known as the Root Mean Square Error (RMSE).
The \textit{CPU-time} represents the time needed to compute one realisation of an estimator. 
The \textit{complexity} is defined as the number of time steps involved in the simulation of an estimator.
Let $Y$ denote the estimator \eqref{crude_MC_estimator} or \eqref{MLMC_estimator}. We estimate the bias of $Y$ by
\begin{equation*}
\widehat b_R = \frac{1}{R} \sum_{k=1}^R Y^{k} - \E [X],
\end{equation*} 

\noindent where $Y^{1},\ldots,Y^{R}$ are $R$ independent replications of the estimator. We estimate the variance of $Y$ by
\begin{equation*}
\widehat v_R = \frac{1}{R} \sum_{k=1}^R v^{k},
\end{equation*}

\noindent where $v^{1},\ldots,v^{R}$ are $R$ independent replications of $v$ the empirical variance of $Y$. In the case where $Y$ is the crude Monte Carlo estimator we set
\begin{equation*}
v = \frac{1}{N(N-1)} \sum_{k=1}^N (X_h^{k} - m_N)^2,
\hspace{0.5cm}
m_N = \frac{1}{N} \sum_{k=1}^N X_h^{k}.
\end{equation*}

\noindent If $Y$ is the MLMC estimator, we set
\begin{equation*}
v = \frac{1}{N_1 (N_1 -1)} \sum_{k=1}^{N_1} (X_h^{k} - m^{(1)}_{N_1})^2 + \sum_{l = 2}^L \frac{1}{N_l (N_l-1)} \sum_{k=1}^{N_l} (X_{h_l}^{k}-X_{h_{l-1}}^{k} - m^{(l)}_{N_l})^2,
\end{equation*}

\noindent where $m^{(1)}_{N_1} = \frac{1}{N_1} \sum_{k=1}^{N_1} X_{h}^{k}$ and for $l \geq 2$, $m^{(l)}_{N_l} = \frac{1}{N_l} \sum_{k=1}^{N_l} X_{h_l}^{k}-X_{h_{l-1}}^{k}$. Then, we define the empirical RMSE $\widehat \epsilon_R$ by
\begin{equation}\label{empirical_RMSE}
\widehat \epsilon_R = \sqrt{ \widehat b_R^2 + \widehat v_R  }.
\end{equation} 

\noindent The numerical computation of \eqref{empirical_RMSE} for both estimators \eqref{crude_MC_estimator} and \eqref{MLMC_estimator} requires the computation of the optimal parameters given by \eqref{MC_parameters} and in table \ref{MLMC_parameters} of section \ref{MC_MLMC_section} which are expressed in term of the structural parameters $c_1$, $V_1$ and $\text{Var}(X)$. Moreover the computation of the bias requires the value $\E[X]$. Since there is no closed formula for the mean and variance of $X$ we estimate them using a crude Monte Carlo estimator with $h=10^{-5}$ and $N=10^6$. The constants $c_1$ and $V_1$ are not explicit, we use the same estimator of $V_1$ as in \cite{LP} section $5.1$, that is 
\begin{equation}\label{V1_estimator}
\widehat{V_1} = (1 + M^{-\beta/2})^{-2} h^{- \beta} \E \left[ |X_h - X_{h/M}|^2 \right],
\end{equation} 

\noindent and we use the following estimator of $c_1$
\begin{equation}\label{c1_estimator}
\widehat{c_1} = \left(1- M^{- \alpha} \right)^{-1} h^{-\alpha} \E \left[ X_{h/M} - X_h\right].
\end{equation}

\noindent The estimator of $c_1$ is obtained writing the weak error expansion for the two time steps $h$ and $h/M$, summing and neglecting the $O(h^2)$ term. In \eqref{V1_estimator} we use $(h,M)=(0.1,4)$ and in \eqref{c1_estimator}, we use $(h,M) = (1,4)$ and the expectations are estimated using a classical Monte Carlo of size $N=10^4$ on $(X_{h/M},X_h)$. We emphasize that we interested in the order of $c_1$ and $V_1$ so that we do not need a precise estimation here.

\subsection{Numerical results}\label{section_numerical_results}

In this section we first illustrate the results of Theorems \ref{thm_strong_cv} and \ref{thm_strong_cv_tilde} on the Morris-Lecar PDMP, then we compare the MC and MLMC estimators. The simulations were carried out on a computer with a processor Intel Core i5-4300U CPU @ 1.90GHz $\times$ 4. The code is written in C++ language. We implement the estimator $\tilde Y^{\text{MLMC}}$ (see section \ref{MC_MLMC_section}) for the following choices of the parameters $(\tilde \lambda, \tilde Q)$.

\smallskip \noindent \textbf{Case 1:} $\tilde \lambda(\theta) = 1$ and $\tilde Q\Big{(} \theta, \{\theta +1\} \Big{)} = \frac{N_\text{K} - \theta}{N_\text{K}}$, $\tilde Q\Big{(} \theta, \{\theta - 1\} \Big{)} = \frac{\theta}{N_\text{K}}$.

\smallskip \noindent \textbf{Case 2:} $\tilde \lambda(x,t) = \lambda(\theta, v(t))$ and $\tilde Q((x,t),dy) = Q((\theta,v(t)),dy)$ where $v$ denotes the first component of the solution of \eqref{deter_ML}.

\smallskip \noindent Cases 1 and 2 correspond to the application of Proposition \ref{DECOMP2}. Based on Corollary \ref{cor1_propr2} we also consider the following case.

\smallskip \noindent \textbf{Case 3:} Consider the quantity $\E [F(x_T) - F(\tilde x_T)]$ where $(x_t)$ and $(\tilde x_t)$ are PDPs with characteristics $(\Phi,\lambda,Q)$ and $(\tilde \Phi,\lambda,Q)$ respectively. By Corollary \ref{cor1_propr2}, we have $\E[F(\tilde x_T)] = \E[F(y_T) \tilde R_T]$ where $(y_t)$ is a PDP whose discrete component jumps in the same states and at the times as the discrete component of $(x_t)$ do and $(\tilde R_t)$ is the corresponding corrective process. Thus, we consider the quantity  $\E [F(x_T) - F(y_T) \tilde R_T]$ instead of $\E [F(x_T) - F(\tilde x_T)]$. 

\smallskip \noindent The case 3 implies to use the following MLMC estimator which is slightly different from \eqref{MLMC_estimator}.
\begin{equation*}
\tilde Y^{\text{MLMC}} = \frac{1}{N_1} \sum_{k=1}^{N_1} X_{h^*}^{k} + \sum_{l=2}^L \frac{1}{N_l} \sum_{k=1}^{N_l} X_{h_l}^{k} - \tilde X_{h_{l-1}}^{k},
\end{equation*}

\noindent where $\left( (X_{h_l}^{k}, \tilde X_{h_{l-1}}^{k}), k \geq 1 \right)$ for $l=2,\ldots,L$ are independent sequences of independent copies of the couple $(X_{h_l}, \tilde X_{h_{l-1}}) = (F(\overline x_T),F(\overline y_T)\tilde R_T)$ where $(\overline y_t)$ is a PDP whose discrete component jumps in the same states and at the same times as the Euler scheme $(\overline x_t)$ with time step $h_l$ do, whose deterministic motions are given by the approximate flows with time step $h_{l-1}$ and $(\tilde R_t)$ is the corresponding corrective process (see Corollary \ref{cor1_propr2}).

\smallskip \noindent The figure \ref{variance_graph} confirms numerically that $\E [ | X_{h_l} - X_{h_{l-1}} |^2 ] = O(h_l)$ and that $\E[ |\tilde X_{h_l} - \tilde X_{h_{l-1}}|^2 ] = O(h_l^2)$ for the cases 1,2 and 3 (see Theorems \ref{thm_strong_cv} and \ref{thm_strong_cv_tilde} respectively). Indeed, for $T=10$ (see figure \ref{variance_decay_10}), we observe that the curve corresponding to the decay of $\E[ | X_{h_l} -  X_{h_{l-1}}|^2 ]$ as $l$ increases is approximately parallel to a line of slope -1 and that the curves corresponding to the decay of $\E[ |\tilde X_{h_l} - \tilde X_{h_{l-1}}|^2 ]$ in the cases 1,2 and 3 are parallel to a line of slope -2. We also see that the curves corresponding to the cases 2 and 3 are approximately similar and that for some value of $l$ those curves go below the one corresponding to $\E[ | X_{h_l} -  X_{h_{l-1}}|^2 ]$. The curve corresponding to the case 1 is always above all the other ones, this indicates that the L$^2$-error (or the variance) in the case 1 is too big (w.r.t the others) and that is why we do do not consider this case in the sequel. As $T$ increases (see figures \ref{variance_decay_20} and \ref{variance_decay_30}), the theoretical order of the numerical schemes is still observed. However, for $T=20$, a slight difference begin to emerge between the cases 2 and 3 (the case 3 being better) and this difference is accentuated for $T=30$ so that we do not represent the case 2. 

\begin{figure}[h]
  \begin{subfigure}{0.5\textwidth}
    \resizebox{\textwidth}{!}{\input{var_decay_10}}
    \caption{T=10.}
    \label{variance_decay_10}
  \end{subfigure} 
  \begin{subfigure}{0.5\textwidth}
    \resizebox{\textwidth}{!}{\input{var_decay_20}}
    \caption{T=20.}
    \label{variance_decay_20}
  \end{subfigure}
  \begin{subfigure}{0.5\textwidth}
    \resizebox{\textwidth}{!}{\input{var_decay_30}}
    \caption{T=30.}
    \label{variance_decay_30}
  \end{subfigure} 
  \caption{The plots (a),(b) and (c) show the decay of $\E [(X_{h_l} - X_{h_{l-1}})^2]$ and $\E[\tilde X_{h_l} - \tilde X_{h_{l-1}})^2]$ ($y$-axis, $\log_M$ scale) as a function of $l$ with $h_l = h \times M^{-(l-1)}$, $h=1$, $M=4$, for different values of the final time $T$. For visual guide, we added black solid lines with slopes -1 and -2.}
\label{variance_graph}
\end{figure}
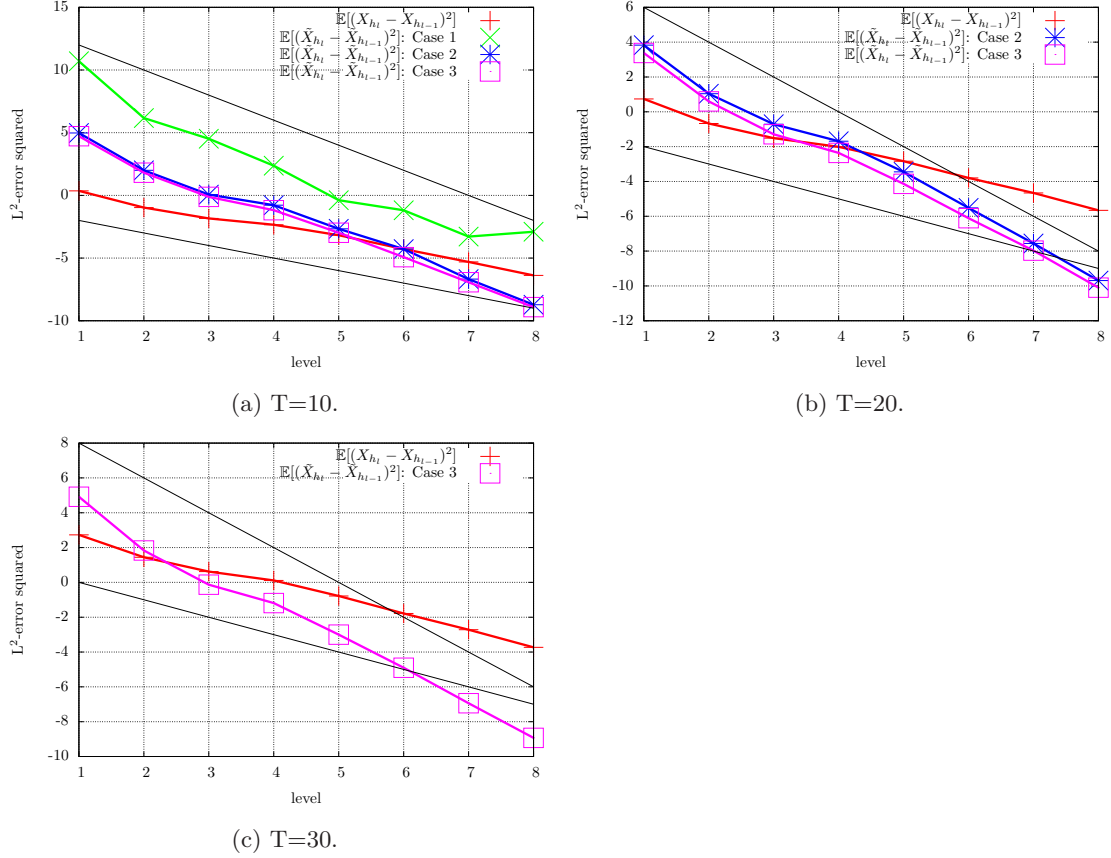

\noindent For the Monte Carlo simulations we set $T=30$, $\lambda^*=10$ and the time step involved in the first level of the MLMC is set to $h^*= 0.1$. We choose this value for $h^*$ because it represents (on average) the size of an interval $[T^*_n,T^*_{n+1}]$ of two successive jump times of the auxiliary Poisson process $(N^*_t)$. The estimation of the true value and variance leads $\E[X] = -31.4723$ and Var$(X) = 335$. Note that $v(30) = -35.3083$ where $v$ is the deterministic membrane potential solution of \eqref{deter_ML} so that there is an offset between the deterministic potential and the mean of the stochastic potential. We replicate 100 times the simulation of the classical and multilevel estimators to compute the empirical RMSE so that $R=100$ in \eqref{empirical_RMSE}.

\begin{table}[h]
\begin{tabular}{|c|c|c|c|c|c||c|c|c|} 
\hline
$k$ & $\epsilon = 2^{-k}$ & $\widehat \epsilon_{100}$ & $\widehat b_{100}$ & $\widehat v_{100}$ & time (sec) & $N$ & $h$ & cost \\
\hline
1 & 5.00e-01 & 4.32e-01 & 2.34e-01 & 1.52e-01 & 3.10e-01 & 2.16e+03 & 6.30e-02 & 3.43e+04 \\
\hline
2 & 2.50e-01 & 2.59e-01 & 1.69e-01 & 3.87e-02 & 1.55e+00 & 8.47e+03 & 3.15e-02 & 2.69e+05\\
\hline
3 & 1.25e-01 & 1.17e-01 & 6.25e-02 & 9.78e-03 & 8.80e+00 & 3.34e+04 & 1.58e-02 & 2.12e+06\\
\hline
4 & 6.25e-02 & 5.67e-02 & 2.73e-02 & 2.47e-03 & 5.62e+01 & 1.32e+05 & 7.88e-03 & 1.68e+07\\
\hline
5 & 3.12e-02 & 2.50e-02 & -1.78e-03 & 6.21e-04 & 3.93e+02 & 5.24e+05 & 3.94e-03 & 1.33e+08\\
\hline
\end{tabular} 
\caption{Results and parameters of the Monte Carlo estimator $Y^{\text{MC}}$. Estimated values of the structural parameters: $c_1=4.58$, $V_1=7.25$.}
\label{MC1_results} 
\end{table}

\begin{table}[h]
\begin{tabular}{|c|c|c|c|c|c||c|c|c|c|c|} 
\hline
$k$ & $\epsilon = 2^{-k}$ & $\widehat \epsilon_{100}$ & $\widehat b_{100}$ & $\widehat v_{100}$ & time (sec) & $L$ & $M$ & $h$ & $N$ & cost \\
\hline
1 & 5.00e-01 & 3.89e-01 & 1.14e-01 & 1.38e-01 & 3.62e-01 & 2 & 2 & 0.1 & 2.60e+03 & 2.82e+04\\
\hline
2 & 2.50e-01 & 2.29e-01 & 1.19e-01 & 3.83e-02 & 1.44e+00 & 2 & 4 & 0.1 & 1.04e+04 & 1.16e+05\\
\hline
3 & 1.25e-01 & 1.21e-01 & 6.24e-02 & 1.07e-02 & 5.76e+00 & 2 & 7 & 0.1 & 4.22e+04 & 4.85e+05\\
\hline
4 & 6.25e-02 & 5.91e-02 & 1.38e-02 & 3.30e-03 & 2.69e+01 & 3 & 4 & 0.1 & 1.90e+05 & 2.37e+06\\
\hline
5 & 3.12e-02 & 3.47e-02 & -1.39e-02 & 1.01e-03 & 1.08e+02 & 3 & 6 & 0.1 & 7.71e+05 & 9.99e+06\\
\hline
\end{tabular} 
\caption{Results and parameters of the Multilevel Monte Carlo estimator $ Y^{\text{MLMC}}$. Estimated values of the structural parameters: $c_1=4.58$, $V_1=7.25$.}
\label{MLMC1_results} 
\end{table}

\begin{table}[h]
\begin{tabular}{|c|c|c|c|c|c||c|c|c|c|c|} 
\hline
$k$ & $\epsilon = 2^{-k}$ & $\widehat \epsilon_{100}$ & $\widehat b_{100}$ & $\widehat v_{100}$ & time (sec) & $L$ & $M$ & $h$ & $N$ & cost \\
\hline
1 & 5.00e-01 & 4.28e-01 & 1.98e-01 & 1.44e-01 & 3.13e-01 & 2 & 2 & 0.1 & 2.38e+03 & 2.50e+04\\
\hline
2 & 2.50e-01 & 2.47e-01 & 1.55e-01 & 3.72e-02 & 1.26e+00 & 2 & 3 & 0.1 & 9.46e+03 & 1.00e+05\\
\hline
3 & 1.25e-01 & 1.36e-01 & 8.90e-02 & 1.05e-02 & 5.00e+00 & 2 	& 6 & 0.1 & 3.80e+04 & 4.11e+05\\
\hline
4 & 6.25e-02 & 6.22e-02 & 2.15e-02 & 3.41e-03 & 2.09e+01 & 3 & 4 & 0.1 & 1.58e+05 & 1.75e+06\\
\hline
5 & 3.12e-02 & 3.17e-02 & 6.07e-03 & 9.71e-04 & 8.35e+01 & 3 & 5 & 0.1 & 6.30e+05 & 7.02e+06\\
\hline
\end{tabular} 
\caption{Results and parameters of the Multilevel Monte Carlo estimator $ \tilde Y^{\text{MLMC}}$ (case 3). Estimated values of the structural parameters: $\tilde c_1=3.91$, $\tilde V_1=34.1$.}
\label{MLMC2_results} 
\end{table}

\noindent The results of the Monte Carlo simulations are shown in tables \ref{MC1_results} for the classical Monte Carlo estimator $Y^{\text{MC}}$ and in tables \ref{MLMC1_results} and \ref{MLMC2_results} for the multilevel estimators $Y^{\text{MLMC}}$ and $\tilde Y^{\text{MLMC}}$ (case 3). As an example, the first line of table \ref{MLMC1_results} reads as follows: for a user prescribed $\epsilon= 2^{-1} = 0.5$, the MLMC estimator $Y^{\text{MLMC}}$ is implemented with $L = 2 $ levels, the time step at the first level is $h^* = 0.1$, this time step is refined by a factor $n_l = M^{l-1}$ with $M=2$ at each levels and the sample size is $N = 2600$. For such parameters, the numerical complexity of the estimator is $\text{Cost}(Y^{\text{MLMC}} ) = 28200$, the empirical RMSE $\widehat \epsilon_{100} = 0.389$ and the computational time of one realisation of $Y^{\text{MLMC}}$ is $0.362$ seconds. We also reported the empirical bias $\widehat b_{100}$ and the empirical variance $\widehat v_{100}$ in view of \eqref{empirical_RMSE}.

\noindent The results indicate that the MLMC outperforms the classical MC. More precisely, for small values of $\epsilon$ (i.e $k=1,2,3$) the complexity and the CPU-time of the classical and the multilevel MC estimators are of the same order. As $\epsilon$ decreases (i.e as $k$ increases) the difference in complexity and CPU-time between classical and multilevel MC increases. Indeed, for $k=5$ the complexity of the estimator $Y^{\text{MC}}$ is approximately 13 times superior to the one of $Y^{\text{MLMC}}$ and 19 times superior to the one of $\tilde Y^{\text{MLMC}}$. The same fact appears when we look at the complexity ratio of the estimators $Y^{\text{MLMC}}$ and $\tilde Y^{\text{MLMC}}$ (i.e Cost($Y^{\text{MLMC}}$)/Cost($\tilde Y^{\text{MLMC}}$)) as $\epsilon$ decreases. However, the difference between the complexity of these two MLMC estimators increases more slowly than the one between a MC and a MLMC estimator. 
Recall that the computational benefit of the MLMC over the MC grows as the prescribed $\epsilon$ decreases.

\noindent Both classical and multilevel estimators provide an empirical RMSE which is close to the prescribed precision (see tables \ref{MC1_results}, \ref{MLMC1_results} and \ref{MLMC2_results}). We can conclude that the choice of the parameters is well adapted. 
For the readability, figures \ref{ratio_cost}, \ref{ratio_time} show the ratios of the complexities and the CPU-times of the three estimators $Y^{\text{MC}}$, $Y^{\text{MLMC}}$ and $\tilde Y^{\text{MLMC}}$ as a function of $\epsilon$.

\begin{figure}[h]
	\begin{subfigure}{.5\linewidth}
    		\resizebox{\columnwidth}{!}{\input{ratio_cost}}
    		\caption{Ratio of the complexities.}
    	\label{ratio_cost}
    \end{subfigure}
    \begin{subfigure}{.5\linewidth}
    		\resizebox{\columnwidth}{!}{\input{ratio_time}}
    		\caption{Ratio of the CPU-times.}
    	\label{ratio_time}
    \end{subfigure}
\caption{The plots (a) and (b) show the complexity and CPU-time ratios w.r.t the complexity and CPU-time of the estimator $\tilde Y^{\text{MLMC}}$ as a function of the prescribed $\epsilon$ ($\log_2$ scale for the $x$-axis, $\log$ scale for the $y$-axis).}
\end{figure}
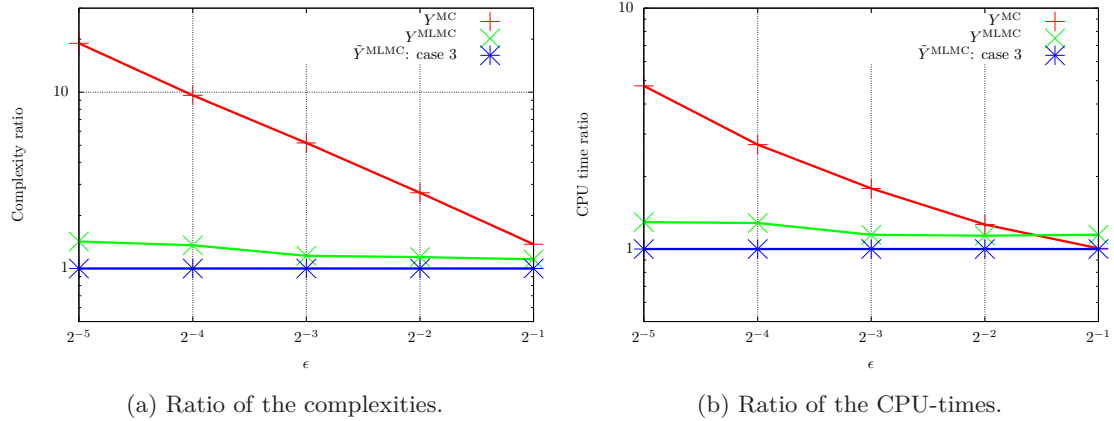 

\bibliographystyle{plain}
\bibliography{bibMLMC}

\end{document}

%% file: V_ML2d.tex
\begingroup
  \makeatletter
  \providecommand\color[2][]{%
    \GenericError{(gnuplot) \space\space\space\@spaces}{%
      Package color not loaded in conjunction with
      terminal option `colourtext'%
    }{See the gnuplot documentation for explanation.%
    }{Either use 'blacktext' in gnuplot or load the package
      color.sty in LaTeX.}%
    \renewcommand\color[2][]{}%
  }%
  \providecommand\includegraphics[2][]{%
    \GenericError{(gnuplot) \space\space\space\@spaces}{%
      Package graphicx or graphics not loaded%
    }{See the gnuplot documentation for explanation.%
    }{The gnuplot epslatex terminal needs graphicx.sty or graphics.sty.}%
    \renewcommand\includegraphics[2][]{}%
  }%
  \providecommand\rotatebox[2]{#2}%
  \@ifundefined{ifGPcolor}{%
    \newif\ifGPcolor
    \GPcolortrue
  }{}%
  \@ifundefined{ifGPblacktext}{%
    \newif\ifGPblacktext
    \GPblacktextfalse
  }{}%
  \let\gplgaddtomacro\g@addto@macro
  \gdef\gplbacktext{}%
  \gdef\gplfronttext{}%
  \makeatother
  \ifGPblacktext
    \def\colorrgb#1{}%
    \def\colorgray#1{}%
  \else
    \ifGPcolor
      \def\colorrgb#1{\color[rgb]{#1}}%
      \def\colorgray#1{\color[gray]{#1}}%
      \expandafter\def\csname LTw\endcsname{\color{white}}%
      \expandafter\def\csname LTb\endcsname{\color{black}}%
      \expandafter\def\csname LTa\endcsname{\color{black}}%
      \expandafter\def\csname LT0\endcsname{\color[rgb]{1,0,0}}%
      \expandafter\def\csname LT1\endcsname{\color[rgb]{0,1,0}}%
      \expandafter\def\csname LT2\endcsname{\color[rgb]{0,0,1}}%
      \expandafter\def\csname LT3\endcsname{\color[rgb]{1,0,1}}%
      \expandafter\def\csname LT4\endcsname{\color[rgb]{0,1,1}}%
      \expandafter\def\csname LT5\endcsname{\color[rgb]{1,1,0}}%
      \expandafter\def\csname LT6\endcsname{\color[rgb]{0,0,0}}%
      \expandafter\def\csname LT7\endcsname{\color[rgb]{1,0.3,0}}%
      \expandafter\def\csname LT8\endcsname{\color[rgb]{0.5,0.5,0.5}}%
    \else
      \def\colorrgb#1{\color{black}}%
      \def\colorgray#1{\color[gray]{#1}}%
      \expandafter\def\csname LTw\endcsname{\color{white}}%
      \expandafter\def\csname LTb\endcsname{\color{black}}%
      \expandafter\def\csname LTa\endcsname{\color{black}}%
      \expandafter\def\csname LT0\endcsname{\color{black}}%
      \expandafter\def\csname LT1\endcsname{\color{black}}%
      \expandafter\def\csname LT2\endcsname{\color{black}}%
      \expandafter\def\csname LT3\endcsname{\color{black}}%
      \expandafter\def\csname LT4\endcsname{\color{black}}%
      \expandafter\def\csname LT5\endcsname{\color{black}}%
      \expandafter\def\csname LT6\endcsname{\color{black}}%
      \expandafter\def\csname LT7\endcsname{\color{black}}%
      \expandafter\def\csname LT8\endcsname{\color{black}}%
    \fi
  \fi
  \setlength{\unitlength}{0.0500bp}%
  \begin{picture}(7200.00,5040.00)%
    \gplgaddtomacro\gplbacktext{%
      \csname LTb\endcsname%
      \put(814,704){\makebox(0,0)[r]{\strut{}-60}}%
      \put(814,1383){\makebox(0,0)[r]{\strut{}-40}}%
      \put(814,2061){\makebox(0,0)[r]{\strut{}-20}}%
      \put(814,2740){\makebox(0,0)[r]{\strut{} 0}}%
      \put(814,3418){\makebox(0,0)[r]{\strut{} 20}}%
      \put(814,4097){\makebox(0,0)[r]{\strut{} 40}}%
      \put(814,4775){\makebox(0,0)[r]{\strut{} 60}}%
      \put(946,484){\makebox(0,0){\strut{} 0}}%
      \put(2117,484){\makebox(0,0){\strut{} 20}}%
      \put(3289,484){\makebox(0,0){\strut{} 40}}%
      \put(4460,484){\makebox(0,0){\strut{} 60}}%
      \put(5632,484){\makebox(0,0){\strut{} 80}}%
      \put(6803,484){\makebox(0,0){\strut{} 100}}%
      \put(176,2739){\rotatebox{-270}{\makebox(0,0){\strut{}Membrane potential (mV)}}}%
      \put(3874,154){\makebox(0,0){\strut{}Time (ms)}}%
    }%
    \gplgaddtomacro\gplfronttext{%
      \csname LTb\endcsname%
      \put(5816,4602){\makebox(0,0)[r]{\strut{}$\nu$}}%
      \csname LTb\endcsname%
      \put(5816,4382){\makebox(0,0)[r]{\strut{}$v$}}%
    }%
    \gplbacktext
    \put(0,0){\includegraphics{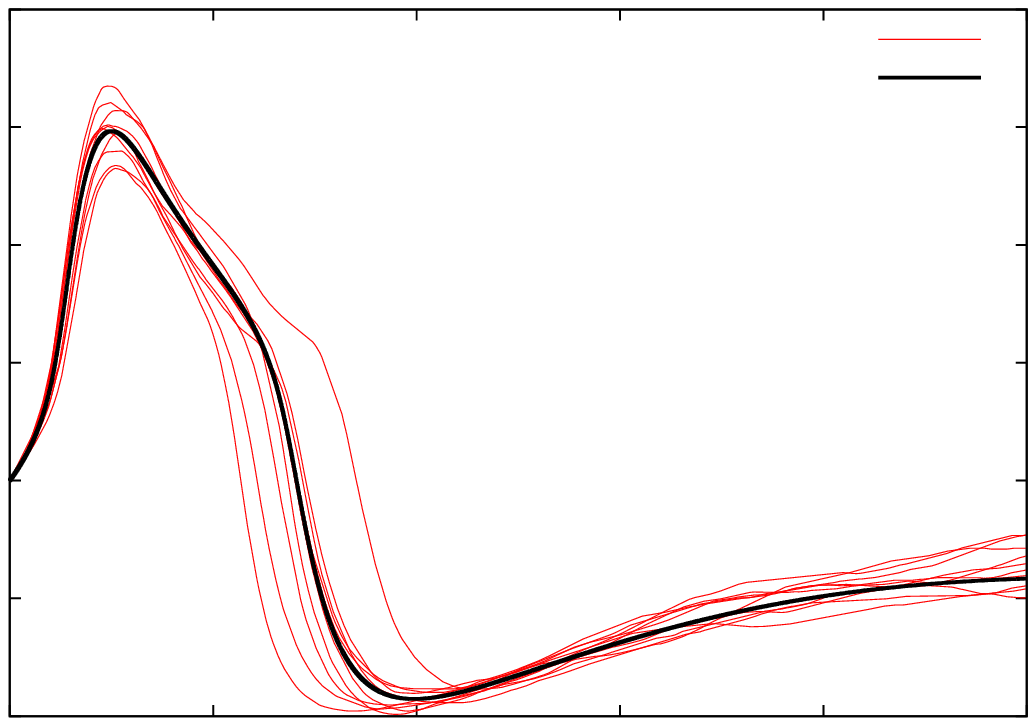}}%
    \gplfronttext
  \end{picture}%
\endgroup

%% file: theta_ML2d.tex
\begingroup
  \makeatletter
  \providecommand\color[2][]{%
    \GenericError{(gnuplot) \space\space\space\@spaces}{%
      Package color not loaded in conjunction with
      terminal option `colourtext'%
    }{See the gnuplot documentation for explanation.%
    }{Either use 'blacktext' in gnuplot or load the package
      color.sty in LaTeX.}%
    \renewcommand\color[2][]{}%
  }%
  \providecommand\includegraphics[2][]{%
    \GenericError{(gnuplot) \space\space\space\@spaces}{%
      Package graphicx or graphics not loaded%
    }{See the gnuplot documentation for explanation.%
    }{The gnuplot epslatex terminal needs graphicx.sty or graphics.sty.}%
    \renewcommand\includegraphics[2][]{}%
  }%
  \providecommand\rotatebox[2]{#2}%
  \@ifundefined{ifGPcolor}{%
    \newif\ifGPcolor
    \GPcolortrue
  }{}%
  \@ifundefined{ifGPblacktext}{%
    \newif\ifGPblacktext
    \GPblacktextfalse
  }{}%
  \let\gplgaddtomacro\g@addto@macro
  \gdef\gplbacktext{}%
  \gdef\gplfronttext{}%
  \makeatother
  \ifGPblacktext
    \def\colorrgb#1{}%
    \def\colorgray#1{}%
  \else
    \ifGPcolor
      \def\colorrgb#1{\color[rgb]{#1}}%
      \def\colorgray#1{\color[gray]{#1}}%
      \expandafter\def\csname LTw\endcsname{\color{white}}%
      \expandafter\def\csname LTb\endcsname{\color{black}}%
      \expandafter\def\csname LTa\endcsname{\color{black}}%
      \expandafter\def\csname LT0\endcsname{\color[rgb]{1,0,0}}%
      \expandafter\def\csname LT1\endcsname{\color[rgb]{0,1,0}}%
      \expandafter\def\csname LT2\endcsname{\color[rgb]{0,0,1}}%
      \expandafter\def\csname LT3\endcsname{\color[rgb]{1,0,1}}%
      \expandafter\def\csname LT4\endcsname{\color[rgb]{0,1,1}}%
      \expandafter\def\csname LT5\endcsname{\color[rgb]{1,1,0}}%
      \expandafter\def\csname LT6\endcsname{\color[rgb]{0,0,0}}%
      \expandafter\def\csname LT7\endcsname{\color[rgb]{1,0.3,0}}%
      \expandafter\def\csname LT8\endcsname{\color[rgb]{0.5,0.5,0.5}}%
    \else
      \def\colorrgb#1{\color{black}}%
      \def\colorgray#1{\color[gray]{#1}}%
      \expandafter\def\csname LTw\endcsname{\color{white}}%
      \expandafter\def\csname LTb\endcsname{\color{black}}%
      \expandafter\def\csname LTa\endcsname{\color{black}}%
      \expandafter\def\csname LT0\endcsname{\color{black}}%
      \expandafter\def\csname LT1\endcsname{\color{black}}%
      \expandafter\def\csname LT2\endcsname{\color{black}}%
      \expandafter\def\csname LT3\endcsname{\color{black}}%
      \expandafter\def\csname LT4\endcsname{\color{black}}%
      \expandafter\def\csname LT5\endcsname{\color{black}}%
      \expandafter\def\csname LT6\endcsname{\color{black}}%
      \expandafter\def\csname LT7\endcsname{\color{black}}%
      \expandafter\def\csname LT8\endcsname{\color{black}}%
    \fi
  \fi
  \setlength{\unitlength}{0.0500bp}%
  \begin{picture}(7200.00,5040.00)%
    \gplgaddtomacro\gplbacktext{%
      \csname LTb\endcsname%
      \put(946,704){\makebox(0,0)[r]{\strut{} 0}}%
      \put(946,1383){\makebox(0,0)[r]{\strut{} 0.1}}%
      \put(946,2061){\makebox(0,0)[r]{\strut{} 0.2}}%
      \put(946,2740){\makebox(0,0)[r]{\strut{} 0.3}}%
      \put(946,3418){\makebox(0,0)[r]{\strut{} 0.4}}%
      \put(946,4097){\makebox(0,0)[r]{\strut{} 0.5}}%
      \put(946,4775){\makebox(0,0)[r]{\strut{} 0.6}}%
      \put(1078,484){\makebox(0,0){\strut{} 0}}%
      \put(2223,484){\makebox(0,0){\strut{} 20}}%
      \put(3368,484){\makebox(0,0){\strut{} 40}}%
      \put(4513,484){\makebox(0,0){\strut{} 60}}%
      \put(5658,484){\makebox(0,0){\strut{} 80}}%
      \put(6803,484){\makebox(0,0){\strut{} 100}}%
      \put(176,2739){\rotatebox{-270}{\makebox(0,0){\strut{}Proportion of open gates}}}%
      \put(3940,154){\makebox(0,0){\strut{}Time (ms)}}%
    }%
    \gplgaddtomacro\gplfronttext{%
      \csname LTb\endcsname%
      \put(5816,4602){\makebox(0,0)[r]{\strut{}$\theta$}}%
      \csname LTb\endcsname%
      \put(5816,4382){\makebox(0,0)[r]{\strut{}$n$}}%
    }%
    \gplbacktext
    \put(0,0){\includegraphics{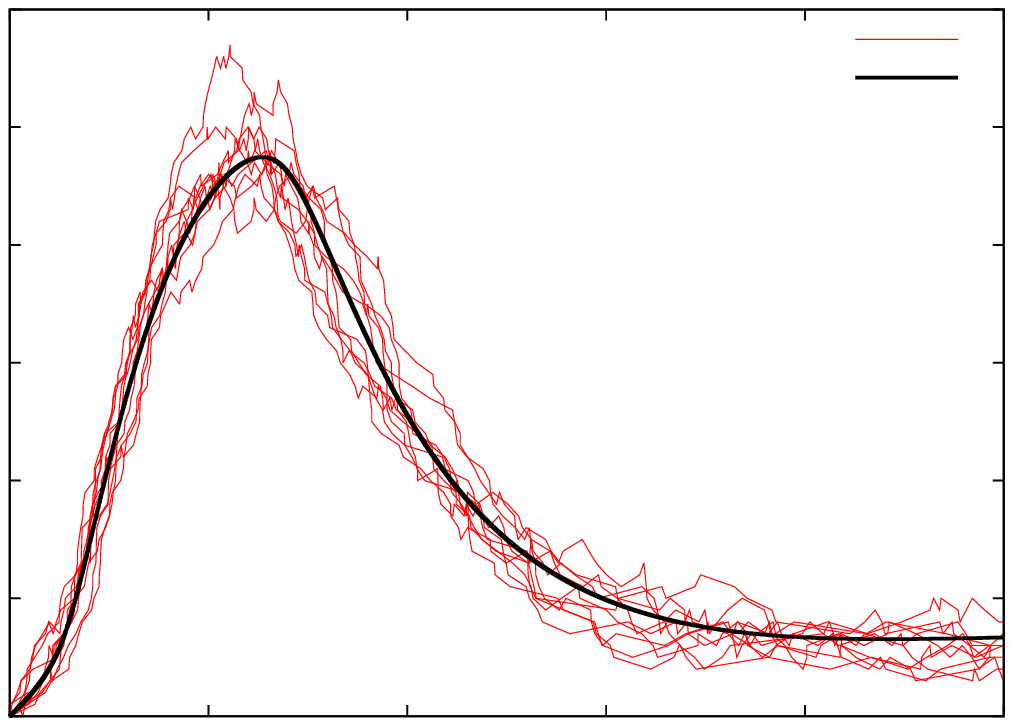}}%
    \gplfronttext
  \end{picture}%
\endgroup

%% file: kernel_ML2d.tex
\begingroup
  \makeatletter
  \providecommand\color[2][]{%
    \GenericError{(gnuplot) \space\space\space\@spaces}{%
      Package color not loaded in conjunction with
      terminal option `colourtext'%
    }{See the gnuplot documentation for explanation.%
    }{Either use 'blacktext' in gnuplot or load the package
      color.sty in LaTeX.}%
    \renewcommand\color[2][]{}%
  }%
  \providecommand\includegraphics[2][]{%
    \GenericError{(gnuplot) \space\space\space\@spaces}{%
      Package graphicx or graphics not loaded%
    }{See the gnuplot documentation for explanation.%
    }{The gnuplot epslatex terminal needs graphicx.sty or graphics.sty.}%
    \renewcommand\includegraphics[2][]{}%
  }%
  \providecommand\rotatebox[2]{#2}%
  \@ifundefined{ifGPcolor}{%
    \newif\ifGPcolor
    \GPcolortrue
  }{}%
  \@ifundefined{ifGPblacktext}{%
    \newif\ifGPblacktext
    \GPblacktextfalse
  }{}%
  \let\gplgaddtomacro\g@addto@macro
  \gdef\gplbacktext{}%
  \gdef\gplfronttext{}%
  \makeatother
  \ifGPblacktext
    \def\colorrgb#1{}%
    \def\colorgray#1{}%
  \else
    \ifGPcolor
      \def\colorrgb#1{\color[rgb]{#1}}%
      \def\colorgray#1{\color[gray]{#1}}%
      \expandafter\def\csname LTw\endcsname{\color{white}}%
      \expandafter\def\csname LTb\endcsname{\color{black}}%
      \expandafter\def\csname LTa\endcsname{\color{black}}%
      \expandafter\def\csname LT0\endcsname{\color[rgb]{1,0,0}}%
      \expandafter\def\csname LT1\endcsname{\color[rgb]{0,1,0}}%
      \expandafter\def\csname LT2\endcsname{\color[rgb]{0,0,1}}%
      \expandafter\def\csname LT3\endcsname{\color[rgb]{1,0,1}}%
      \expandafter\def\csname LT4\endcsname{\color[rgb]{0,1,1}}%
      \expandafter\def\csname LT5\endcsname{\color[rgb]{1,1,0}}%
      \expandafter\def\csname LT6\endcsname{\color[rgb]{0,0,0}}%
      \expandafter\def\csname LT7\endcsname{\color[rgb]{1,0.3,0}}%
      \expandafter\def\csname LT8\endcsname{\color[rgb]{0.5,0.5,0.5}}%
    \else
      \def\colorrgb#1{\color{black}}%
      \def\colorgray#1{\color[gray]{#1}}%
      \expandafter\def\csname LTw\endcsname{\color{white}}%
      \expandafter\def\csname LTb\endcsname{\color{black}}%
      \expandafter\def\csname LTa\endcsname{\color{black}}%
      \expandafter\def\csname LT0\endcsname{\color{black}}%
      \expandafter\def\csname LT1\endcsname{\color{black}}%
      \expandafter\def\csname LT2\endcsname{\color{black}}%
      \expandafter\def\csname LT3\endcsname{\color{black}}%
      \expandafter\def\csname LT4\endcsname{\color{black}}%
      \expandafter\def\csname LT5\endcsname{\color{black}}%
      \expandafter\def\csname LT6\endcsname{\color{black}}%
      \expandafter\def\csname LT7\endcsname{\color{black}}%
      \expandafter\def\csname LT8\endcsname{\color{black}}%
    \fi
  \fi
  \setlength{\unitlength}{0.0500bp}%
  \begin{picture}(7200.00,5040.00)%
    \gplgaddtomacro\gplbacktext{%
      \csname LTb\endcsname%
      \put(946,704){\makebox(0,0)[r]{\strut{} 0}}%
      \put(946,1111){\makebox(0,0)[r]{\strut{} 0.1}}%
      \put(946,1518){\makebox(0,0)[r]{\strut{} 0.2}}%
      \put(946,1925){\makebox(0,0)[r]{\strut{} 0.3}}%
      \put(946,2332){\makebox(0,0)[r]{\strut{} 0.4}}%
      \put(946,2740){\makebox(0,0)[r]{\strut{} 0.5}}%
      \put(946,3147){\makebox(0,0)[r]{\strut{} 0.6}}%
      \put(946,3554){\makebox(0,0)[r]{\strut{} 0.7}}%
      \put(946,3961){\makebox(0,0)[r]{\strut{} 0.8}}%
      \put(946,4368){\makebox(0,0)[r]{\strut{} 0.9}}%
      \put(946,4775){\makebox(0,0)[r]{\strut{} 1}}%
      \put(1078,484){\makebox(0,0){\strut{} 0}}%
      \put(2223,484){\makebox(0,0){\strut{} 20}}%
      \put(3368,484){\makebox(0,0){\strut{} 40}}%
      \put(4513,484){\makebox(0,0){\strut{} 60}}%
      \put(5658,484){\makebox(0,0){\strut{} 80}}%
      \put(6803,484){\makebox(0,0){\strut{} 100}}%
      \put(176,2739){\rotatebox{-270}{\makebox(0,0){\strut{}Probability of opening gates}}}%
      \put(3940,154){\makebox(0,0){\strut{}Time (ms)}}%
    }%
    \gplgaddtomacro\gplfronttext{%
      \csname LTb\endcsname%
      \put(5816,4602){\makebox(0,0)[r]{\strut{}$Q$}}%
    }%
    \gplbacktext
    \put(0,0){\includegraphics{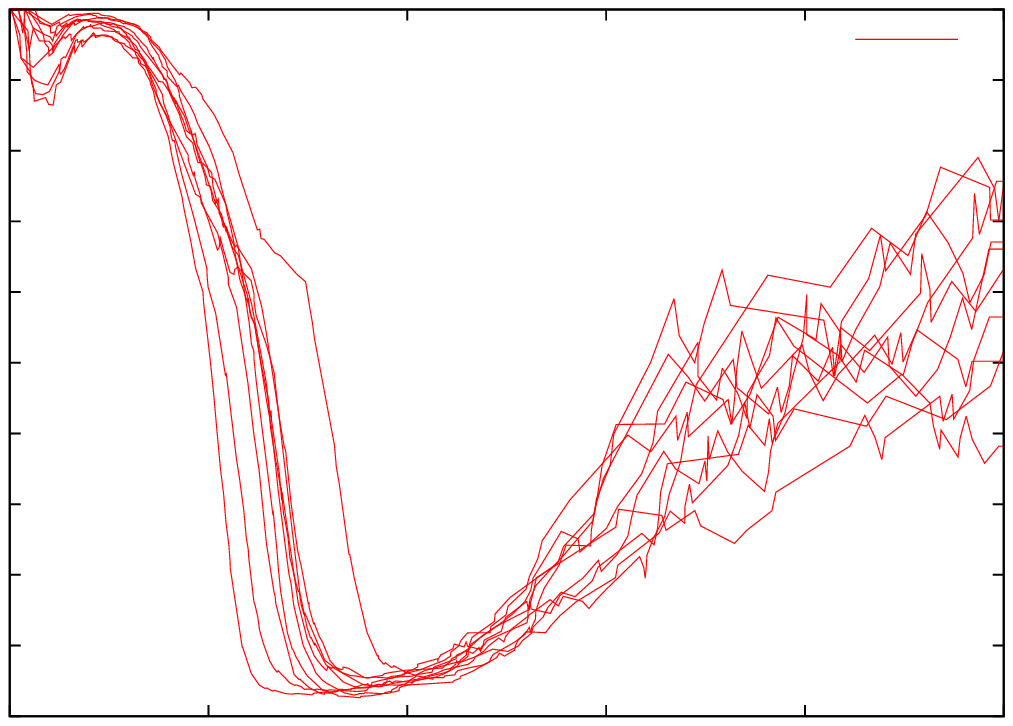}}%
    \gplfronttext
  \end{picture}%
\endgroup

%% file: lambda_ML2d.tex
\begingroup
  \makeatletter
  \providecommand\color[2][]{%
    \GenericError{(gnuplot) \space\space\space\@spaces}{%
      Package color not loaded in conjunction with
      terminal option `colourtext'%
    }{See the gnuplot documentation for explanation.%
    }{Either use 'blacktext' in gnuplot or load the package
      color.sty in LaTeX.}%
    \renewcommand\color[2][]{}%
  }%
  \providecommand\includegraphics[2][]{%
    \GenericError{(gnuplot) \space\space\space\@spaces}{%
      Package graphicx or graphics not loaded%
    }{See the gnuplot documentation for explanation.%
    }{The gnuplot epslatex terminal needs graphicx.sty or graphics.sty.}%
    \renewcommand\includegraphics[2][]{}%
  }%
  \providecommand\rotatebox[2]{#2}%
  \@ifundefined{ifGPcolor}{%
    \newif\ifGPcolor
    \GPcolortrue
  }{}%
  \@ifundefined{ifGPblacktext}{%
    \newif\ifGPblacktext
    \GPblacktextfalse
  }{}%
  \let\gplgaddtomacro\g@addto@macro
  \gdef\gplbacktext{}%
  \gdef\gplfronttext{}%
  \makeatother
  \ifGPblacktext
    \def\colorrgb#1{}%
    \def\colorgray#1{}%
  \else
    \ifGPcolor
      \def\colorrgb#1{\color[rgb]{#1}}%
      \def\colorgray#1{\color[gray]{#1}}%
      \expandafter\def\csname LTw\endcsname{\color{white}}%
      \expandafter\def\csname LTb\endcsname{\color{black}}%
      \expandafter\def\csname LTa\endcsname{\color{black}}%
      \expandafter\def\csname LT0\endcsname{\color[rgb]{1,0,0}}%
      \expandafter\def\csname LT1\endcsname{\color[rgb]{0,1,0}}%
      \expandafter\def\csname LT2\endcsname{\color[rgb]{0,0,1}}%
      \expandafter\def\csname LT3\endcsname{\color[rgb]{1,0,1}}%
      \expandafter\def\csname LT4\endcsname{\color[rgb]{0,1,1}}%
      \expandafter\def\csname LT5\endcsname{\color[rgb]{1,1,0}}%
      \expandafter\def\csname LT6\endcsname{\color[rgb]{0,0,0}}%
      \expandafter\def\csname LT7\endcsname{\color[rgb]{1,0.3,0}}%
      \expandafter\def\csname LT8\endcsname{\color[rgb]{0.5,0.5,0.5}}%
    \else
      \def\colorrgb#1{\color{black}}%
      \def\colorgray#1{\color[gray]{#1}}%
      \expandafter\def\csname LTw\endcsname{\color{white}}%
      \expandafter\def\csname LTb\endcsname{\color{black}}%
      \expandafter\def\csname LTa\endcsname{\color{black}}%
      \expandafter\def\csname LT0\endcsname{\color{black}}%
      \expandafter\def\csname LT1\endcsname{\color{black}}%
      \expandafter\def\csname LT2\endcsname{\color{black}}%
      \expandafter\def\csname LT3\endcsname{\color{black}}%
      \expandafter\def\csname LT4\endcsname{\color{black}}%
      \expandafter\def\csname LT5\endcsname{\color{black}}%
      \expandafter\def\csname LT6\endcsname{\color{black}}%
      \expandafter\def\csname LT7\endcsname{\color{black}}%
      \expandafter\def\csname LT8\endcsname{\color{black}}%
    \fi
  \fi
  \setlength{\unitlength}{0.0500bp}%
  \begin{picture}(7200.00,5040.00)%
    \gplgaddtomacro\gplbacktext{%
      \csname LTb\endcsname%
      \put(946,704){\makebox(0,0)[r]{\strut{} 0}}%
      \put(946,1156){\makebox(0,0)[r]{\strut{} 0.5}}%
      \put(946,1609){\makebox(0,0)[r]{\strut{} 1}}%
      \put(946,2061){\makebox(0,0)[r]{\strut{} 1.5}}%
      \put(946,2513){\makebox(0,0)[r]{\strut{} 2}}%
      \put(946,2966){\makebox(0,0)[r]{\strut{} 2.5}}%
      \put(946,3418){\makebox(0,0)[r]{\strut{} 3}}%
      \put(946,3870){\makebox(0,0)[r]{\strut{} 3.5}}%
      \put(946,4323){\makebox(0,0)[r]{\strut{} 4}}%
      \put(946,4775){\makebox(0,0)[r]{\strut{} 4.5}}%
      \put(1078,484){\makebox(0,0){\strut{} 0}}%
      \put(2223,484){\makebox(0,0){\strut{} 20}}%
      \put(3368,484){\makebox(0,0){\strut{} 40}}%
      \put(4513,484){\makebox(0,0){\strut{} 60}}%
      \put(5658,484){\makebox(0,0){\strut{} 80}}%
      \put(6803,484){\makebox(0,0){\strut{} 100}}%
      \put(176,2739){\rotatebox{-270}{\makebox(0,0){\strut{}Intensity}}}%
      \put(3940,154){\makebox(0,0){\strut{}Time (ms)}}%
    }%
    \gplgaddtomacro\gplfronttext{%
      \csname LTb\endcsname%
      \put(5816,4602){\makebox(0,0)[r]{\strut{}$\lambda$}}%
    }%
    \gplbacktext
    \put(0,0){\includegraphics{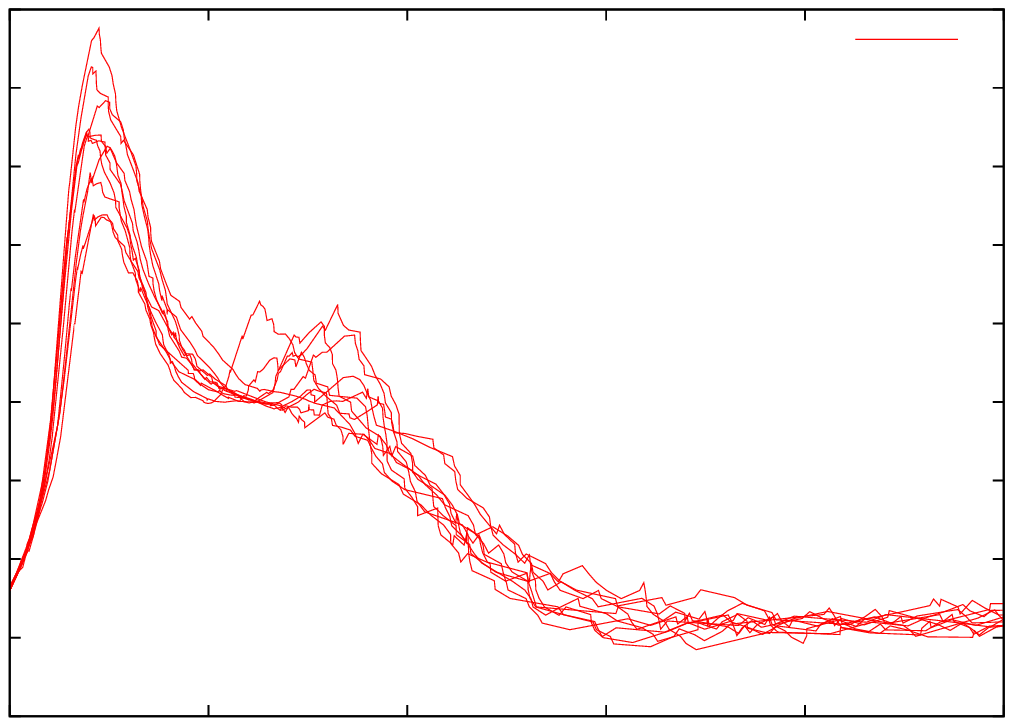}}%
    \gplfronttext
  \end{picture}%
\endgroup

%% file: var_decay_10.tex
\begingroup
  \makeatletter
  \providecommand\color[2][]{%
    \GenericError{(gnuplot) \space\space\space\@spaces}{%
      Package color not loaded in conjunction with
      terminal option `colourtext'%
    }{See the gnuplot documentation for explanation.%
    }{Either use 'blacktext' in gnuplot or load the package
      color.sty in LaTeX.}%
    \renewcommand\color[2][]{}%
  }%
  \providecommand\includegraphics[2][]{%
    \GenericError{(gnuplot) \space\space\space\@spaces}{%
      Package graphicx or graphics not loaded%
    }{See the gnuplot documentation for explanation.%
    }{The gnuplot epslatex terminal needs graphicx.sty or graphics.sty.}%
    \renewcommand\includegraphics[2][]{}%
  }%
  \providecommand\rotatebox[2]{#2}%
  \@ifundefined{ifGPcolor}{%
    \newif\ifGPcolor
    \GPcolortrue
  }{}%
  \@ifundefined{ifGPblacktext}{%
    \newif\ifGPblacktext
    \GPblacktextfalse
  }{}%
  \let\gplgaddtomacro\g@addto@macro
  \gdef\gplbacktext{}%
  \gdef\gplfronttext{}%
  \makeatother
  \ifGPblacktext
    \def\colorrgb#1{}%
    \def\colorgray#1{}%
  \else
    \ifGPcolor
      \def\colorrgb#1{\color[rgb]{#1}}%
      \def\colorgray#1{\color[gray]{#1}}%
      \expandafter\def\csname LTw\endcsname{\color{white}}%
      \expandafter\def\csname LTb\endcsname{\color{black}}%
      \expandafter\def\csname LTa\endcsname{\color{black}}%
      \expandafter\def\csname LT0\endcsname{\color[rgb]{1,0,0}}%
      \expandafter\def\csname LT1\endcsname{\color[rgb]{0,1,0}}%
      \expandafter\def\csname LT2\endcsname{\color[rgb]{0,0,1}}%
      \expandafter\def\csname LT3\endcsname{\color[rgb]{1,0,1}}%
      \expandafter\def\csname LT4\endcsname{\color[rgb]{0,1,1}}%
      \expandafter\def\csname LT5\endcsname{\color[rgb]{1,1,0}}%
      \expandafter\def\csname LT6\endcsname{\color[rgb]{0,0,0}}%
      \expandafter\def\csname LT7\endcsname{\color[rgb]{1,0.3,0}}%
      \expandafter\def\csname LT8\endcsname{\color[rgb]{0.5,0.5,0.5}}%
    \else
      \def\colorrgb#1{\color{black}}%
      \def\colorgray#1{\color[gray]{#1}}%
      \expandafter\def\csname LTw\endcsname{\color{white}}%
      \expandafter\def\csname LTb\endcsname{\color{black}}%
      \expandafter\def\csname LTa\endcsname{\color{black}}%
      \expandafter\def\csname LT0\endcsname{\color{black}}%
      \expandafter\def\csname LT1\endcsname{\color{black}}%
      \expandafter\def\csname LT2\endcsname{\color{black}}%
      \expandafter\def\csname LT3\endcsname{\color{black}}%
      \expandafter\def\csname LT4\endcsname{\color{black}}%
      \expandafter\def\csname LT5\endcsname{\color{black}}%
      \expandafter\def\csname LT6\endcsname{\color{black}}%
      \expandafter\def\csname LT7\endcsname{\color{black}}%
      \expandafter\def\csname LT8\endcsname{\color{black}}%
    \fi
  \fi
  \setlength{\unitlength}{0.0500bp}%
  \begin{picture}(7200.00,5040.00)%
    \gplgaddtomacro\gplbacktext{%
      \csname LTb\endcsname%
      \put(814,704){\makebox(0,0)[r]{\strut{}-10}}%
      \csname LTb\endcsname%
      \put(814,1518){\makebox(0,0)[r]{\strut{}-5}}%
      \csname LTb\endcsname%
      \put(814,2332){\makebox(0,0)[r]{\strut{} 0}}%
      \csname LTb\endcsname%
      \put(814,3147){\makebox(0,0)[r]{\strut{} 5}}%
      \csname LTb\endcsname%
      \put(814,3961){\makebox(0,0)[r]{\strut{} 10}}%
      \csname LTb\endcsname%
      \put(814,4775){\makebox(0,0)[r]{\strut{} 15}}%
      \csname LTb\endcsname%
      \put(946,484){\makebox(0,0){\strut{} 1}}%
      \csname LTb\endcsname%
      \put(1783,484){\makebox(0,0){\strut{} 2}}%
      \csname LTb\endcsname%
      \put(2619,484){\makebox(0,0){\strut{} 3}}%
      \csname LTb\endcsname%
      \put(3456,484){\makebox(0,0){\strut{} 4}}%
      \csname LTb\endcsname%
      \put(4293,484){\makebox(0,0){\strut{} 5}}%
      \csname LTb\endcsname%
      \put(5130,484){\makebox(0,0){\strut{} 6}}%
      \csname LTb\endcsname%
      \put(5966,484){\makebox(0,0){\strut{} 7}}%
      \csname LTb\endcsname%
      \put(6803,484){\makebox(0,0){\strut{} 8}}%
      \put(176,2739){\rotatebox{-270}{\makebox(0,0){\strut{}L$^2$-error squared}}}%
      \put(3874,154){\makebox(0,0){\strut{}level}}%
    }%
    \gplgaddtomacro\gplfronttext{%
      \csname LTb\endcsname%
      \put(5816,4602){\makebox(0,0)[r]{\strut{}$\mathbb{E} [(X_{h_l} - X_{h_{l-1}})^2]$}}%
      \csname LTb\endcsname%
      \put(5816,4382){\makebox(0,0)[r]{\strut{}$\mathbb{E} [(\tilde X_{h_l} - \tilde X_{h_{l-1}})^2]$: Case 1}}%
      \csname LTb\endcsname%
      \put(5816,4162){\makebox(0,0)[r]{\strut{}$\mathbb{E} [(\tilde X_{h_l} - \tilde X_{h_{l-1}})^2]$: Case 2}}%
      \csname LTb\endcsname%
      \put(5816,3942){\makebox(0,0)[r]{\strut{}$\mathbb{E} [(\tilde X_{h_l} - \tilde X_{h_{l-1}})^2]$: Case 3}}%
    }%
    \gplbacktext
    \put(0,0){\includegraphics{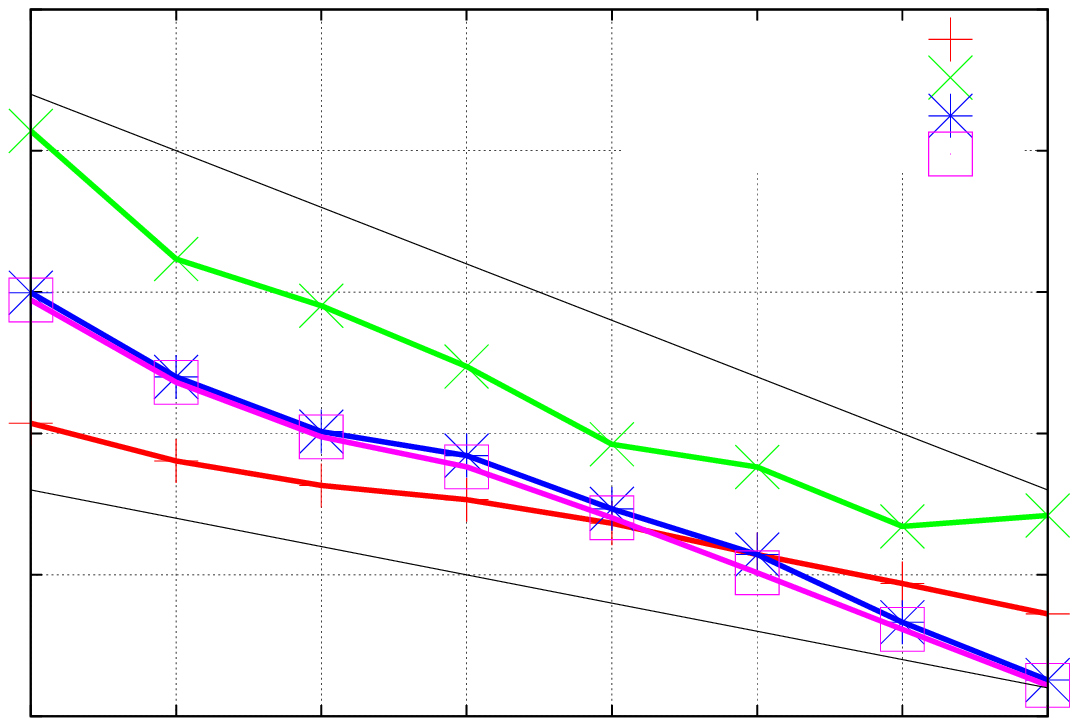}}%
    \gplfronttext
  \end{picture}%
\endgroup

%% file: var_decay_20.tex
\begingroup
  \makeatletter
  \providecommand\color[2][]{%
    \GenericError{(gnuplot) \space\space\space\@spaces}{%
      Package color not loaded in conjunction with
      terminal option `colourtext'%
    }{See the gnuplot documentation for explanation.%
    }{Either use 'blacktext' in gnuplot or load the package
      color.sty in LaTeX.}%
    \renewcommand\color[2][]{}%
  }%
  \providecommand\includegraphics[2][]{%
    \GenericError{(gnuplot) \space\space\space\@spaces}{%
      Package graphicx or graphics not loaded%
    }{See the gnuplot documentation for explanation.%
    }{The gnuplot epslatex terminal needs graphicx.sty or graphics.sty.}%
    \renewcommand\includegraphics[2][]{}%
  }%
  \providecommand\rotatebox[2]{#2}%
  \@ifundefined{ifGPcolor}{%
    \newif\ifGPcolor
    \GPcolortrue
  }{}%
  \@ifundefined{ifGPblacktext}{%
    \newif\ifGPblacktext
    \GPblacktextfalse
  }{}%
  \let\gplgaddtomacro\g@addto@macro
  \gdef\gplbacktext{}%
  \gdef\gplfronttext{}%
  \makeatother
  \ifGPblacktext
    \def\colorrgb#1{}%
    \def\colorgray#1{}%
  \else
    \ifGPcolor
      \def\colorrgb#1{\color[rgb]{#1}}%
      \def\colorgray#1{\color[gray]{#1}}%
      \expandafter\def\csname LTw\endcsname{\color{white}}%
      \expandafter\def\csname LTb\endcsname{\color{black}}%
      \expandafter\def\csname LTa\endcsname{\color{black}}%
      \expandafter\def\csname LT0\endcsname{\color[rgb]{1,0,0}}%
      \expandafter\def\csname LT1\endcsname{\color[rgb]{0,1,0}}%
      \expandafter\def\csname LT2\endcsname{\color[rgb]{0,0,1}}%
      \expandafter\def\csname LT3\endcsname{\color[rgb]{1,0,1}}%
      \expandafter\def\csname LT4\endcsname{\color[rgb]{0,1,1}}%
      \expandafter\def\csname LT5\endcsname{\color[rgb]{1,1,0}}%
      \expandafter\def\csname LT6\endcsname{\color[rgb]{0,0,0}}%
      \expandafter\def\csname LT7\endcsname{\color[rgb]{1,0.3,0}}%
      \expandafter\def\csname LT8\endcsname{\color[rgb]{0.5,0.5,0.5}}%
    \else
      \def\colorrgb#1{\color{black}}%
      \def\colorgray#1{\color[gray]{#1}}%
      \expandafter\def\csname LTw\endcsname{\color{white}}%
      \expandafter\def\csname LTb\endcsname{\color{black}}%
      \expandafter\def\csname LTa\endcsname{\color{black}}%
      \expandafter\def\csname LT0\endcsname{\color{black}}%
      \expandafter\def\csname LT1\endcsname{\color{black}}%
      \expandafter\def\csname LT2\endcsname{\color{black}}%
      \expandafter\def\csname LT3\endcsname{\color{black}}%
      \expandafter\def\csname LT4\endcsname{\color{black}}%
      \expandafter\def\csname LT5\endcsname{\color{black}}%
      \expandafter\def\csname LT6\endcsname{\color{black}}%
      \expandafter\def\csname LT7\endcsname{\color{black}}%
      \expandafter\def\csname LT8\endcsname{\color{black}}%
    \fi
  \fi
  \setlength{\unitlength}{0.0500bp}%
  \begin{picture}(7200.00,5040.00)%
    \gplgaddtomacro\gplbacktext{%
      \csname LTb\endcsname%
      \put(814,704){\makebox(0,0)[r]{\strut{}-12}}%
      \csname LTb\endcsname%
      \put(814,1156){\makebox(0,0)[r]{\strut{}-10}}%
      \csname LTb\endcsname%
      \put(814,1609){\makebox(0,0)[r]{\strut{}-8}}%
      \csname LTb\endcsname%
      \put(814,2061){\makebox(0,0)[r]{\strut{}-6}}%
      \csname LTb\endcsname%
      \put(814,2513){\makebox(0,0)[r]{\strut{}-4}}%
      \csname LTb\endcsname%
      \put(814,2966){\makebox(0,0)[r]{\strut{}-2}}%
      \csname LTb\endcsname%
      \put(814,3418){\makebox(0,0)[r]{\strut{} 0}}%
      \csname LTb\endcsname%
      \put(814,3870){\makebox(0,0)[r]{\strut{} 2}}%
      \csname LTb\endcsname%
      \put(814,4323){\makebox(0,0)[r]{\strut{} 4}}%
      \csname LTb\endcsname%
      \put(814,4775){\makebox(0,0)[r]{\strut{} 6}}%
      \csname LTb\endcsname%
      \put(946,484){\makebox(0,0){\strut{} 1}}%
      \csname LTb\endcsname%
      \put(1783,484){\makebox(0,0){\strut{} 2}}%
      \csname LTb\endcsname%
      \put(2619,484){\makebox(0,0){\strut{} 3}}%
      \csname LTb\endcsname%
      \put(3456,484){\makebox(0,0){\strut{} 4}}%
      \csname LTb\endcsname%
      \put(4293,484){\makebox(0,0){\strut{} 5}}%
      \csname LTb\endcsname%
      \put(5130,484){\makebox(0,0){\strut{} 6}}%
      \csname LTb\endcsname%
      \put(5966,484){\makebox(0,0){\strut{} 7}}%
      \csname LTb\endcsname%
      \put(6803,484){\makebox(0,0){\strut{} 8}}%
      \put(176,2739){\rotatebox{-270}{\makebox(0,0){\strut{}L$^2$-error squared}}}%
      \put(3874,154){\makebox(0,0){\strut{}level}}%
    }%
    \gplgaddtomacro\gplfronttext{%
      \csname LTb\endcsname%
      \put(5816,4602){\makebox(0,0)[r]{\strut{}$\mathbb{E} [(X_{h_l} - X_{h_{l-1}})^2]$}}%
      \csname LTb\endcsname%
      \put(5816,4382){\makebox(0,0)[r]{\strut{}$\mathbb{E} [(\tilde X_{h_l} - \tilde X_{h_{l-1}})^2]$: Case 2}}%
      \csname LTb\endcsname%
      \put(5816,4162){\makebox(0,0)[r]{\strut{}$\mathbb{E} [(\tilde X_{h_l} - \tilde X_{h_{l-1}})^2]$: Case 3}}%
    }%
    \gplbacktext
    \put(0,0){\includegraphics{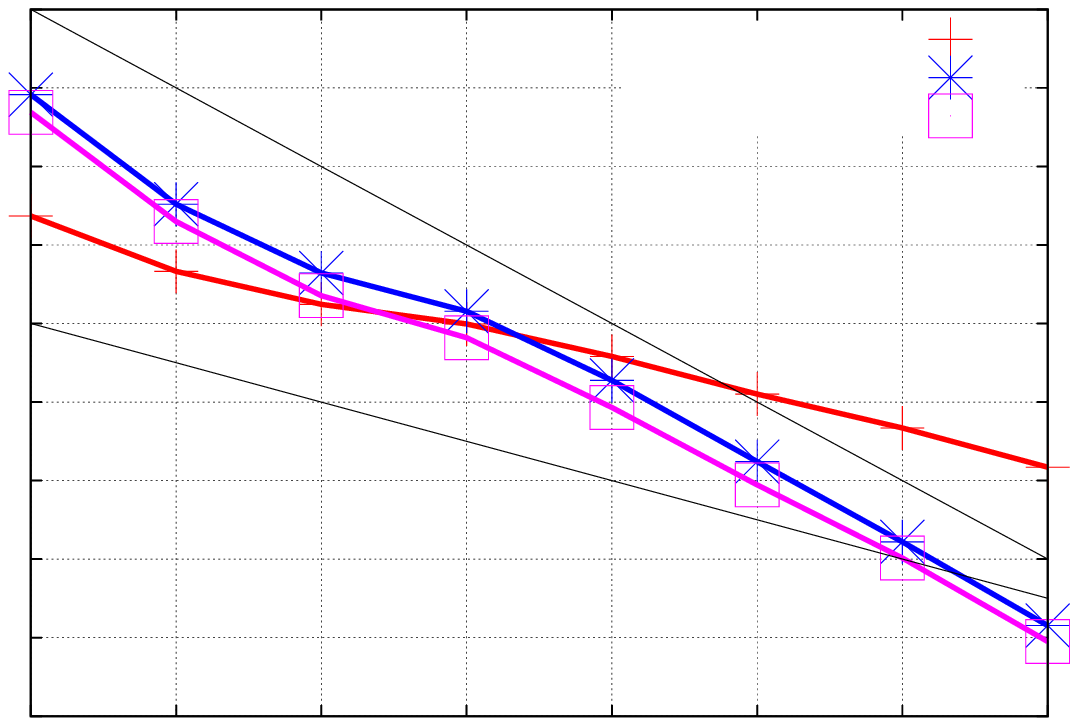}}%
    \gplfronttext
  \end{picture}%
\endgroup

%% file: var_decay_30.tex
\begingroup
  \makeatletter
  \providecommand\color[2][]{%
    \GenericError{(gnuplot) \space\space\space\@spaces}{%
      Package color not loaded in conjunction with
      terminal option `colourtext'%
    }{See the gnuplot documentation for explanation.%
    }{Either use 'blacktext' in gnuplot or load the package
      color.sty in LaTeX.}%
    \renewcommand\color[2][]{}%
  }%
  \providecommand\includegraphics[2][]{%
    \GenericError{(gnuplot) \space\space\space\@spaces}{%
      Package graphicx or graphics not loaded%
    }{See the gnuplot documentation for explanation.%
    }{The gnuplot epslatex terminal needs graphicx.sty or graphics.sty.}%
    \renewcommand\includegraphics[2][]{}%
  }%
  \providecommand\rotatebox[2]{#2}%
  \@ifundefined{ifGPcolor}{%
    \newif\ifGPcolor
    \GPcolortrue
  }{}%
  \@ifundefined{ifGPblacktext}{%
    \newif\ifGPblacktext
    \GPblacktextfalse
  }{}%
  \let\gplgaddtomacro\g@addto@macro
  \gdef\gplbacktext{}%
  \gdef\gplfronttext{}%
  \makeatother
  \ifGPblacktext
    \def\colorrgb#1{}%
    \def\colorgray#1{}%
  \else
    \ifGPcolor
      \def\colorrgb#1{\color[rgb]{#1}}%
      \def\colorgray#1{\color[gray]{#1}}%
      \expandafter\def\csname LTw\endcsname{\color{white}}%
      \expandafter\def\csname LTb\endcsname{\color{black}}%
      \expandafter\def\csname LTa\endcsname{\color{black}}%
      \expandafter\def\csname LT0\endcsname{\color[rgb]{1,0,0}}%
      \expandafter\def\csname LT1\endcsname{\color[rgb]{0,1,0}}%
      \expandafter\def\csname LT2\endcsname{\color[rgb]{0,0,1}}%
      \expandafter\def\csname LT3\endcsname{\color[rgb]{1,0,1}}%
      \expandafter\def\csname LT4\endcsname{\color[rgb]{0,1,1}}%
      \expandafter\def\csname LT5\endcsname{\color[rgb]{1,1,0}}%
      \expandafter\def\csname LT6\endcsname{\color[rgb]{0,0,0}}%
      \expandafter\def\csname LT7\endcsname{\color[rgb]{1,0.3,0}}%
      \expandafter\def\csname LT8\endcsname{\color[rgb]{0.5,0.5,0.5}}%
    \else
      \def\colorrgb#1{\color{black}}%
      \def\colorgray#1{\color[gray]{#1}}%
      \expandafter\def\csname LTw\endcsname{\color{white}}%
      \expandafter\def\csname LTb\endcsname{\color{black}}%
      \expandafter\def\csname LTa\endcsname{\color{black}}%
      \expandafter\def\csname LT0\endcsname{\color{black}}%
      \expandafter\def\csname LT1\endcsname{\color{black}}%
      \expandafter\def\csname LT2\endcsname{\color{black}}%
      \expandafter\def\csname LT3\endcsname{\color{black}}%
      \expandafter\def\csname LT4\endcsname{\color{black}}%
      \expandafter\def\csname LT5\endcsname{\color{black}}%
      \expandafter\def\csname LT6\endcsname{\color{black}}%
      \expandafter\def\csname LT7\endcsname{\color{black}}%
      \expandafter\def\csname LT8\endcsname{\color{black}}%
    \fi
  \fi
  \setlength{\unitlength}{0.0500bp}%
  \begin{picture}(7200.00,5040.00)%
    \gplgaddtomacro\gplbacktext{%
      \csname LTb\endcsname%
      \put(814,704){\makebox(0,0)[r]{\strut{}-10}}%
      \csname LTb\endcsname%
      \put(814,1156){\makebox(0,0)[r]{\strut{}-8}}%
      \csname LTb\endcsname%
      \put(814,1609){\makebox(0,0)[r]{\strut{}-6}}%
      \csname LTb\endcsname%
      \put(814,2061){\makebox(0,0)[r]{\strut{}-4}}%
      \csname LTb\endcsname%
      \put(814,2513){\makebox(0,0)[r]{\strut{}-2}}%
      \csname LTb\endcsname%
      \put(814,2966){\makebox(0,0)[r]{\strut{} 0}}%
      \csname LTb\endcsname%
      \put(814,3418){\makebox(0,0)[r]{\strut{} 2}}%
      \csname LTb\endcsname%
      \put(814,3870){\makebox(0,0)[r]{\strut{} 4}}%
      \csname LTb\endcsname%
      \put(814,4323){\makebox(0,0)[r]{\strut{} 6}}%
      \csname LTb\endcsname%
      \put(814,4775){\makebox(0,0)[r]{\strut{} 8}}%
      \csname LTb\endcsname%
      \put(946,484){\makebox(0,0){\strut{} 1}}%
      \csname LTb\endcsname%
      \put(1783,484){\makebox(0,0){\strut{} 2}}%
      \csname LTb\endcsname%
      \put(2619,484){\makebox(0,0){\strut{} 3}}%
      \csname LTb\endcsname%
      \put(3456,484){\makebox(0,0){\strut{} 4}}%
      \csname LTb\endcsname%
      \put(4293,484){\makebox(0,0){\strut{} 5}}%
      \csname LTb\endcsname%
      \put(5130,484){\makebox(0,0){\strut{} 6}}%
      \csname LTb\endcsname%
      \put(5966,484){\makebox(0,0){\strut{} 7}}%
      \csname LTb\endcsname%
      \put(6803,484){\makebox(0,0){\strut{} 8}}%
      \put(176,2739){\rotatebox{-270}{\makebox(0,0){\strut{}L$^2$-error squared}}}%
      \put(3874,154){\makebox(0,0){\strut{}level}}%
    }%
    \gplgaddtomacro\gplfronttext{%
      \csname LTb\endcsname%
      \put(5816,4602){\makebox(0,0)[r]{\strut{}$\mathbb{E} [(X_{h_l} - X_{h_{l-1}})^2]$}}%
      \csname LTb\endcsname%
      \put(5816,4382){\makebox(0,0)[r]{\strut{}$\mathbb{E} [(\tilde X_{h_l} - \tilde X_{h_{l-1}})^2]$: Case 3}}%
    }%
    \gplbacktext
    \put(0,0){\includegraphics{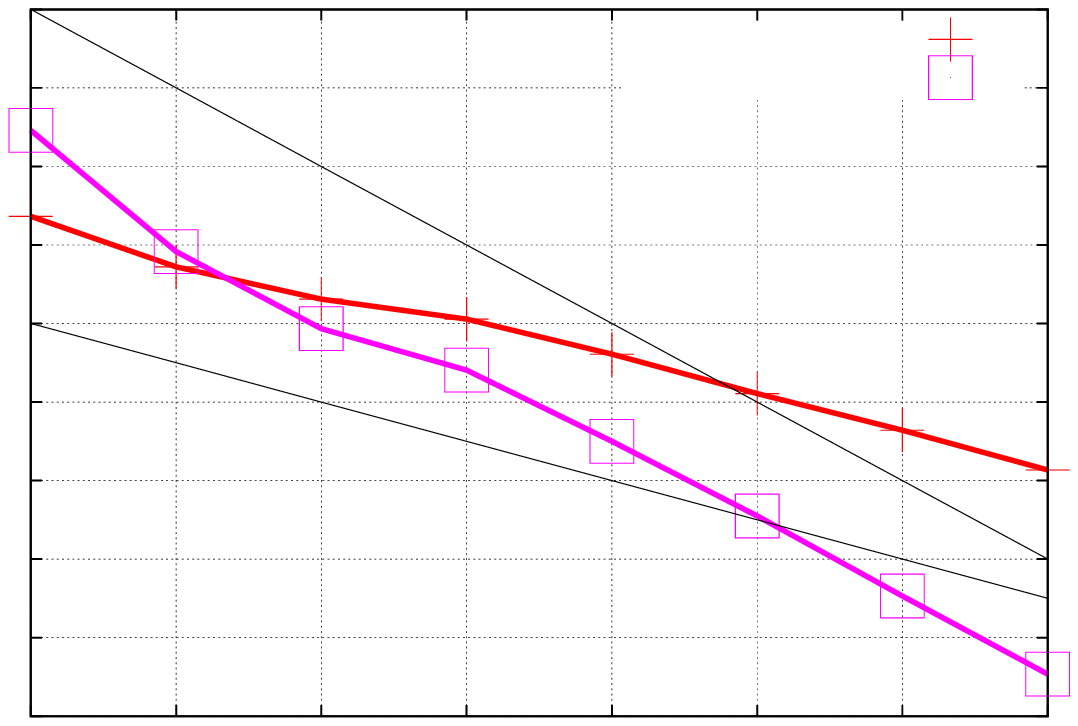}}%
    \gplfronttext
  \end{picture}%
\endgroup

%% file: ratio_cost.tex
\begingroup
  \makeatletter
  \providecommand\color[2][]{%
    \GenericError{(gnuplot) \space\space\space\@spaces}{%
      Package color not loaded in conjunction with
      terminal option `colourtext'%
    }{See the gnuplot documentation for explanation.%
    }{Either use 'blacktext' in gnuplot or load the package
      color.sty in LaTeX.}%
    \renewcommand\color[2][]{}%
  }%
  \providecommand\includegraphics[2][]{%
    \GenericError{(gnuplot) \space\space\space\@spaces}{%
      Package graphicx or graphics not loaded%
    }{See the gnuplot documentation for explanation.%
    }{The gnuplot epslatex terminal needs graphicx.sty or graphics.sty.}%
    \renewcommand\includegraphics[2][]{}%
  }%
  \providecommand\rotatebox[2]{#2}%
  \@ifundefined{ifGPcolor}{%
    \newif\ifGPcolor
    \GPcolortrue
  }{}%
  \@ifundefined{ifGPblacktext}{%
    \newif\ifGPblacktext
    \GPblacktextfalse
  }{}%
  \let\gplgaddtomacro\g@addto@macro
  \gdef\gplbacktext{}%
  \gdef\gplfronttext{}%
  \makeatother
  \ifGPblacktext
    \def\colorrgb#1{}%
    \def\colorgray#1{}%
  \else
    \ifGPcolor
      \def\colorrgb#1{\color[rgb]{#1}}%
      \def\colorgray#1{\color[gray]{#1}}%
      \expandafter\def\csname LTw\endcsname{\color{white}}%
      \expandafter\def\csname LTb\endcsname{\color{black}}%
      \expandafter\def\csname LTa\endcsname{\color{black}}%
      \expandafter\def\csname LT0\endcsname{\color[rgb]{1,0,0}}%
      \expandafter\def\csname LT1\endcsname{\color[rgb]{0,1,0}}%
      \expandafter\def\csname LT2\endcsname{\color[rgb]{0,0,1}}%
      \expandafter\def\csname LT3\endcsname{\color[rgb]{1,0,1}}%
      \expandafter\def\csname LT4\endcsname{\color[rgb]{0,1,1}}%
      \expandafter\def\csname LT5\endcsname{\color[rgb]{1,1,0}}%
      \expandafter\def\csname LT6\endcsname{\color[rgb]{0,0,0}}%
      \expandafter\def\csname LT7\endcsname{\color[rgb]{1,0.3,0}}%
      \expandafter\def\csname LT8\endcsname{\color[rgb]{0.5,0.5,0.5}}%
    \else
      \def\colorrgb#1{\color{black}}%
      \def\colorgray#1{\color[gray]{#1}}%
      \expandafter\def\csname LTw\endcsname{\color{white}}%
      \expandafter\def\csname LTb\endcsname{\color{black}}%
      \expandafter\def\csname LTa\endcsname{\color{black}}%
      \expandafter\def\csname LT0\endcsname{\color{black}}%
      \expandafter\def\csname LT1\endcsname{\color{black}}%
      \expandafter\def\csname LT2\endcsname{\color{black}}%
      \expandafter\def\csname LT3\endcsname{\color{black}}%
      \expandafter\def\csname LT4\endcsname{\color{black}}%
      \expandafter\def\csname LT5\endcsname{\color{black}}%
      \expandafter\def\csname LT6\endcsname{\color{black}}%
      \expandafter\def\csname LT7\endcsname{\color{black}}%
      \expandafter\def\csname LT8\endcsname{\color{black}}%
    \fi
  \fi
  \setlength{\unitlength}{0.0500bp}%
  \begin{picture}(7200.00,5040.00)%
    \gplgaddtomacro\gplbacktext{%
      \csname LTb\endcsname%
      \put(814,1393){\makebox(0,0)[r]{\strut{} 1}}%
      \csname LTb\endcsname%
      \put(814,3683){\makebox(0,0)[r]{\strut{} 10}}%
      \csname LTb\endcsname%
      \put(946,484){\makebox(0,0){\strut{}$2^{-5}$}}%
      \csname LTb\endcsname%
      \put(2413,484){\makebox(0,0){\strut{}$2^{-4}$}}%
      \csname LTb\endcsname%
      \put(3876,484){\makebox(0,0){\strut{}$2^{-3}$}}%
      \csname LTb\endcsname%
      \put(5340,484){\makebox(0,0){\strut{}$2^{-2}$}}%
      \csname LTb\endcsname%
      \put(6803,484){\makebox(0,0){\strut{}$2^{-1}$}}%
      \put(176,2739){\rotatebox{-270}{\makebox(0,0){\strut{}Complexity ratio}}}%
      \put(3874,154){\makebox(0,0){\strut{}$\epsilon$}}%
    }%
    \gplgaddtomacro\gplfronttext{%
      \csname LTb\endcsname%
      \put(5816,4602){\makebox(0,0)[r]{\strut{}$Y^{\text{MC}}$}}%
      \csname LTb\endcsname%
      \put(5816,4382){\makebox(0,0)[r]{\strut{}$Y^{\text{MLMC}}$}}%
      \csname LTb\endcsname%
      \put(5816,4162){\makebox(0,0)[r]{\strut{}$\tilde{Y}^{\text{MLMC}}$: case 3}}%
    }%
    \gplbacktext
    \put(0,0){\includegraphics{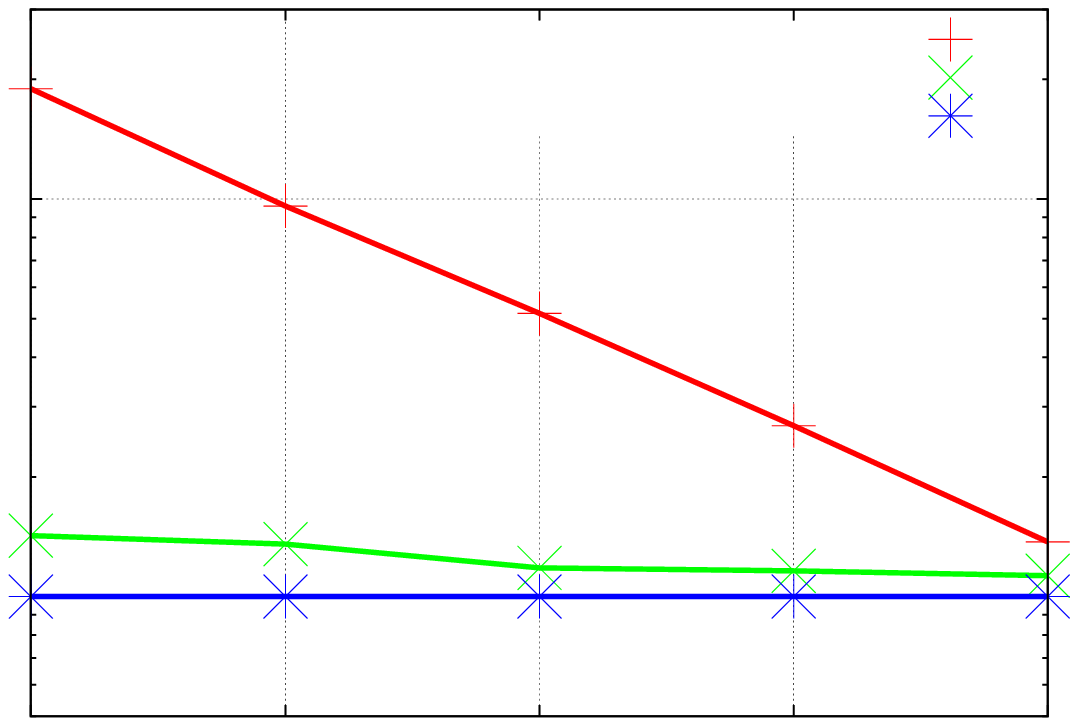}}%
    \gplfronttext
  \end{picture}%
\endgroup

%% file: ratio_time.tex
\begingroup
  \makeatletter
  \providecommand\color[2][]{%
    \GenericError{(gnuplot) \space\space\space\@spaces}{%
      Package color not loaded in conjunction with
      terminal option `colourtext'%
    }{See the gnuplot documentation for explanation.%
    }{Either use 'blacktext' in gnuplot or load the package
      color.sty in LaTeX.}%
    \renewcommand\color[2][]{}%
  }%
  \providecommand\includegraphics[2][]{%
    \GenericError{(gnuplot) \space\space\space\@spaces}{%
      Package graphicx or graphics not loaded%
    }{See the gnuplot documentation for explanation.%
    }{The gnuplot epslatex terminal needs graphicx.sty or graphics.sty.}%
    \renewcommand\includegraphics[2][]{}%
  }%
  \providecommand\rotatebox[2]{#2}%
  \@ifundefined{ifGPcolor}{%
    \newif\ifGPcolor
    \GPcolortrue
  }{}%
  \@ifundefined{ifGPblacktext}{%
    \newif\ifGPblacktext
    \GPblacktextfalse
  }{}%
  \let\gplgaddtomacro\g@addto@macro
  \gdef\gplbacktext{}%
  \gdef\gplfronttext{}%
  \makeatother
  \ifGPblacktext
    \def\colorrgb#1{}%
    \def\colorgray#1{}%
  \else
    \ifGPcolor
      \def\colorrgb#1{\color[rgb]{#1}}%
      \def\colorgray#1{\color[gray]{#1}}%
      \expandafter\def\csname LTw\endcsname{\color{white}}%
      \expandafter\def\csname LTb\endcsname{\color{black}}%
      \expandafter\def\csname LTa\endcsname{\color{black}}%
      \expandafter\def\csname LT0\endcsname{\color[rgb]{1,0,0}}%
      \expandafter\def\csname LT1\endcsname{\color[rgb]{0,1,0}}%
      \expandafter\def\csname LT2\endcsname{\color[rgb]{0,0,1}}%
      \expandafter\def\csname LT3\endcsname{\color[rgb]{1,0,1}}%
      \expandafter\def\csname LT4\endcsname{\color[rgb]{0,1,1}}%
      \expandafter\def\csname LT5\endcsname{\color[rgb]{1,1,0}}%
      \expandafter\def\csname LT6\endcsname{\color[rgb]{0,0,0}}%
      \expandafter\def\csname LT7\endcsname{\color[rgb]{1,0.3,0}}%
      \expandafter\def\csname LT8\endcsname{\color[rgb]{0.5,0.5,0.5}}%
    \else
      \def\colorrgb#1{\color{black}}%
      \def\colorgray#1{\color[gray]{#1}}%
      \expandafter\def\csname LTw\endcsname{\color{white}}%
      \expandafter\def\csname LTb\endcsname{\color{black}}%
      \expandafter\def\csname LTa\endcsname{\color{black}}%
      \expandafter\def\csname LT0\endcsname{\color{black}}%
      \expandafter\def\csname LT1\endcsname{\color{black}}%
      \expandafter\def\csname LT2\endcsname{\color{black}}%
      \expandafter\def\csname LT3\endcsname{\color{black}}%
      \expandafter\def\csname LT4\endcsname{\color{black}}%
      \expandafter\def\csname LT5\endcsname{\color{black}}%
      \expandafter\def\csname LT6\endcsname{\color{black}}%
      \expandafter\def\csname LT7\endcsname{\color{black}}%
      \expandafter\def\csname LT8\endcsname{\color{black}}%
    \fi
  \fi
  \setlength{\unitlength}{0.0500bp}%
  \begin{picture}(7200.00,5040.00)%
    \gplgaddtomacro\gplbacktext{%
      \csname LTb\endcsname%
      \put(814,1646){\makebox(0,0)[r]{\strut{} 1}}%
      \csname LTb\endcsname%
      \put(814,4775){\makebox(0,0)[r]{\strut{} 10}}%
      \csname LTb\endcsname%
      \put(946,484){\makebox(0,0){\strut{}$2^{-5}$}}%
      \csname LTb\endcsname%
      \put(2413,484){\makebox(0,0){\strut{}$2^{-4}$}}%
      \csname LTb\endcsname%
      \put(3876,484){\makebox(0,0){\strut{}$2^{-3}$}}%
      \csname LTb\endcsname%
      \put(5340,484){\makebox(0,0){\strut{}$2^{-2}$}}%
      \csname LTb\endcsname%
      \put(6803,484){\makebox(0,0){\strut{}$2^{-1}$}}%
      \put(176,2739){\rotatebox{-270}{\makebox(0,0){\strut{}CPU time ratio}}}%
      \put(3874,154){\makebox(0,0){\strut{}$\epsilon$}}%
    }%
    \gplgaddtomacro\gplfronttext{%
      \csname LTb\endcsname%
      \put(5816,4602){\makebox(0,0)[r]{\strut{}$Y^{\text{MC}}$}}%
      \csname LTb\endcsname%
      \put(5816,4382){\makebox(0,0)[r]{\strut{}$Y^{\text{MLMC}}$}}%
      \csname LTb\endcsname%
      \put(5816,4162){\makebox(0,0)[r]{\strut{}$\tilde{Y}^{\text{MLMC}}$: case 3}}%
    }%
    \gplbacktext
    \put(0,0){\includegraphics{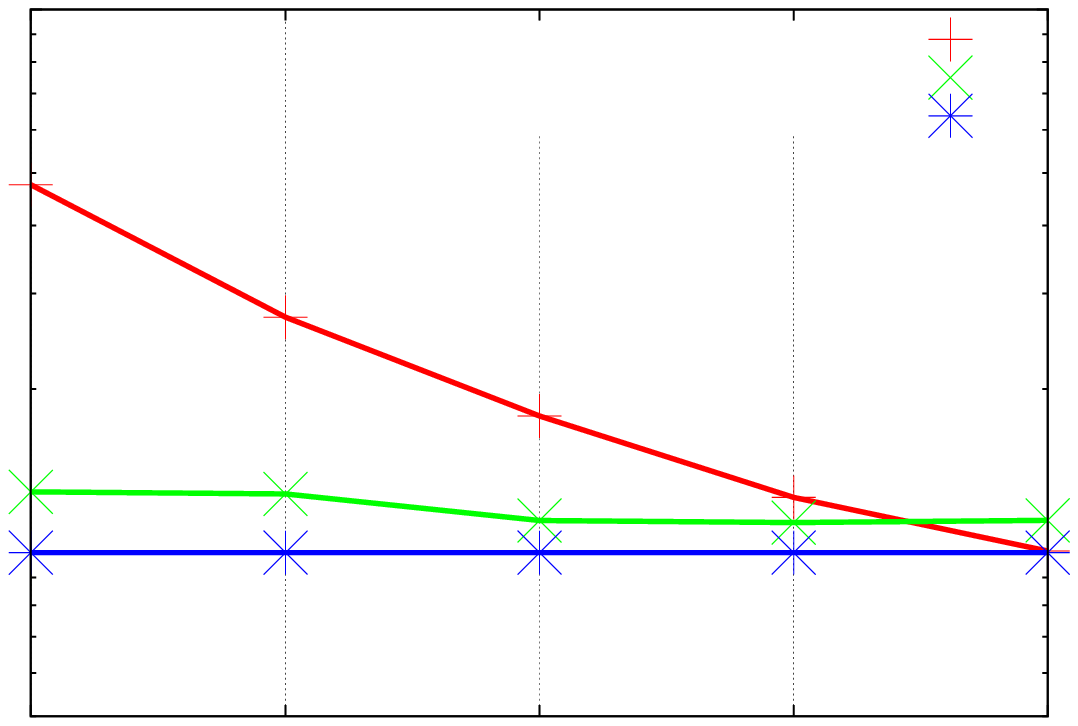}}%
    \gplfronttext
  \end{picture}%
\endgroup